\documentclass[11pt]{amsart}
\usepackage{epsfig}
\usepackage{indentfirst, latexsym, amssymb, enumerate, amsmath, graphicx}
\usepackage{colortbl}
\usepackage{subcaption}
\usepackage{mathtools}
\usepackage{mathrsfs}
\usepackage{amsfonts}
\usepackage{pgfplots}
\pgfplotsset{compat=1.12}
\pgfplotsset{soldot/.style={color=blue,only marks,mark=*}} \pgfplotsset{holdot/.style={color=blue,fill=white,only marks,mark=*}}

\numberwithin{equation}{section}

\allowdisplaybreaks

\title[Filippov trajectories and clustering  in the singular  Kuramoto model]
{Filippov trajectories and clustering  in the Kuramoto model with singular couplings}

\author[Jinyeong Park]{Jinyeong Park}
\address[Jinyeong Park]{\newline  Department of Mathematics and Research Institute of Natural Sciences,\newline Hanyang University, 222 Wangsimni-ro, Seongdong-gu, Seoul 04763, Korea} \email{jinyeongpark@hanyang.ac.kr}

\author[David Poyato]{David Poyato}
\address[David Poyato]{\newline Departamento de Matem\'{a}tica Aplicada and Excellence Research Unit ``Modeling Nature''
(MNat), Universidad de Granada, Granada, 18071, Spain} \email{davidpoyato@ugr.es}

\author[Juan Soler]{Juan Soler}
\address[Juan Soler]{\newline Departamento de Matem\'{a}tica Aplicada and Excellence Research Unit ``Modeling Nature''
(MNat), Universidad de Granada, Granada, 18071, Spain} \email{jsoler@ugr.es}

\newtheorem{theorem}{Theorem}[section]
\newtheorem{lemma}{Lemma}[section]
\newtheorem{corollary}{Corollary}[section]
\newtheorem{proposition}{Proposition}[section]
\newtheorem{example}{Example}[section]
\newtheorem{remark}{Remark}[section]

\newtheorem{definition}{Definition}[section]

\DeclareMathOperator{\rank}{rank}

\DeclareMathOperator{\co}{co}

\DeclareMathOperator{\esssup}{ess\,sup}
\DeclareMathOperator{\Skew}{Skew}

\begin{document}



\thanks{\textbf{Acknowledgment.} The work of J. Park has been supported by the National Research Foundation of Korea (NRF) grant funded by the Korea government (MSIT) (NRF-2018R1C1B5043861). The work of D. Poyato and J. Soler has been partially supported by the MINECO-Feder (Spain) research grant number MTM2014-53406-R, the Junta de Andaluc\'ia (Spain) Project FQM 954, and the MECD (Spain) research grant FPU14/06304 (D.P.). }

\begin{abstract}
We study the synchronization of a generalized Kuramoto system in which the coupling weights are determined by the phase differences between oscillators. We employ the fast-learning regime in a Hebbian--like plasticity rule so that the interaction between oscillators is enhanced by the approach of phases. First, we study the well-posedness problem for the singular weighted Kuramoto systems in which the Lipschitz continuity is deprived. We present the dynamics of the system equipped with singular weights in all the subcritical, critical and supercritical regimes of the singularity. A key fact is that solutions in the most singular cases must be considered in Filippov's sense. We characterize sticking of phases in the subcritical and critical case and we exhibit a continuation criterion of classical solutions after any collision state in the supercritical regime. Second, we prove that strong solutions to these systems of differential inclusions can be recovered as singular limits of regular weights. We also provide the emergence of synchronous dynamics for the singular and regular weighted Kuramoto models. 

\end{abstract}
\maketitle 

\section{Introduction}
\setcounter{equation}{0}
Synchronization is the natural collective behavior arising from agents-based interactions given by periodic rules. These rhythmical motions can be easily observed in various biological complex systems such as flashing of fireflies, beating of cardiac cells, etc. One of the most significant examples of synchronization appear in neurons. Associative or Hebbian learning \cite{H} proposes an explanation for the adaptation of neurons in the brain during the learning process. Such mechanism is founded in the assumption that synchronous activation of cells (firing of neurons) leads to selectively pronounced increases in synaptic strength between those cells. The consequence is that the pattern of activity will become self-organized. In Hebb's words: \textit{Any two cells or systems of cells that are repeatedly active at the same time will tend to become associated, so that activity in one facilitates activity in the other}. In neuroscience, this processes provide the neuronal basis of unsupervised learning of cognitive function in neural networks and can explain the phenomena that arises in the development of the nervous system. 

Since Kuramoto proposed a mathematical model for coupled oscillators in \cite{Ku1, Ku2}, the synchronization has received a lot of attention and has been studied extensively in various disciplines from this point of view \cite{AB}. In the classical Kuramoto model, the system of oscillators has an all-to-all coupling with uniform weights given by a constant coupling strength $K$:
\begin{equation}\label{KM}
\dot\theta_i = \Omega_i + \frac{K}{N} \sum_{j=1}^N \sin(\theta_j - \theta_i), \quad i=1, \ldots, N,
\end{equation}
where $\Omega_i$'s are the natural frequencies of oscillators. However, the uniform and constant couplings are a bit restrictive to explain the complicatedness of phenomena. Thus, it is more interesting to consider a generalization of the Kuramoto model which is equipped with plastic couplings introduced in \cite{A-A, H-L-L-P,H-N-P, N-E,  P-R, R-Z, S-Y-T}:
\begin{equation}\label{KM1}
\dot\theta_i = \Omega_i + \frac{1}{N}\sum_{j=1}^N K_{ij} \sin(\theta_j - \theta_i) , \quad i=1, \ldots, N,
\end{equation}
where $K_{ij}$ is the coupling between $i$-th and $j$-th oscillators which has its own dynamics depending on the phase configuration. The coupling $K_{ij}$  is assumed to be
\[
K_{ij} = K a_{ij},
\]
where $a_{ij}\in [0,1]$ measures the degree of connectedness between the $i$-th and $j$-th oscillators. They will be allowed to vary adaptively relying on the associated phases $\theta_i$ and $\theta_j$, via the dynamic learning law
\begin{equation}\label{adaptive}
 \dot{a}_{ij} = \eta(\Gamma(\theta_j - \theta_i) - a_{ij}),
\end{equation}
for some plasticity function $\Gamma$. Here, $\eta$ is regarded as the learning rate parameter such that a small $\eta$ delays the adaptation of weight $a_{ij}$. According to the choice of the function $\Gamma$, the dynamics of the system \eqref{KM1} shows various scenarios. In neural networks systems, the Hebbian--type dynamics is considered for the learning algorithm of couplings between oscillators. Such learning law amounts to saying that the weight of coupling increases if the phases of oscillators are close to each other. For example, in \cite{H-N-P, N-E, S-Y-T}, $\Gamma$ is assumed to be $\Gamma(\theta) = \cos \theta$ so that attraction between near oscillators is reinforced whereas repulsive interaction arises between apart phases. On the other hand, anti--Hebbian type is also considered such as $\Gamma(\theta) = |\sin \theta|$ in \cite{H-N-P, R-Z}. In this case, the synchronization emerges slowly due to the reduction of weight for nearby oscillators. Other types of adaptive rules are considered in \cite{H-L-L-P,P-R}.

We will consider a Hebbian--like $\Gamma$ for the dynamics of adaptive coupling so that the coupling is enhanced by approach of phases. Assume that the Hebbian--like plasticity function $\Gamma$ is given by
\begin{equation}\label{C-1}
\Gamma(\theta) := \frac{\sigma^{2\alpha}}{\big(\sigma^2 + c_{\alpha,\zeta}|\theta|^2_o\big)^\alpha},
\end{equation}
where $\sigma\in (0,\pi)$, $\zeta\in (0,1]$ and  $|\theta|_o$ is the orthodromic distance (to zero) over the unit circle, that can be defined by
\[
|\theta|_o :=\vert\bar\theta\vert\quad \text{for}\quad \bar\theta \equiv \theta \ \text{mod} ~ 2\pi, \quad \bar\theta\in (-\pi, \pi].
\]
Here, the parameter $c_{\alpha,\zeta}:=1-\zeta^{-1/\alpha}$ has been chosen so that whenever two phases $\theta_i$ and $\theta_j$ stay at orthodromic distance $\sigma$ or larger, then the adaptive function $\Gamma$ predicts a maximum degree of connectedness not larger than $\zeta$ between such oscillators.

Since the plasticity function $\Gamma$ in \eqref{C-1} is Lipschitz-continuous, we can apply the Tikhonov's theorem \cite{KS} to \eqref{KM1}-\eqref{adaptive} in order to rigorously derive the fast learning regime $\eta\rightarrow +\infty$. Then, we arrive at the following Kuramoto model with weighted coupling structure:
\begin{equation}\label{KM2}
\dot\theta_i = \Omega_i + \frac{K}{N} \sum_{j=1}^N \Gamma(\theta_j - \theta_i) \sin(\theta_j - \theta_i),
\end{equation}
that will play a central role in our work. If either the parameter $\alpha=0$ or $\zeta=1$, then our plasticity function \eqref{C-1} becomes $1$ everywhere. In such case, our system \eqref{KM2} reduces to the classical Kuramoto model \eqref{KM}. Hence, we will assume that $\alpha>0$ and $\zeta\in (0,1)$ from now on. Our main interest is to analyze the system \eqref{C-1}-\eqref{KM2} and compare it with the associated singular counterpart with singular plasticity function
\begin{equation}\label{C-2}
\Gamma(\theta):=\frac{1}{c_{\alpha,\zeta}^\alpha |\theta|_o^{2\alpha}}.
\end{equation}
In the next section we will derive this new singular model from the regular one through a singular limit of the parameters. In the regular case \eqref{C-1}, $\Gamma$ is Lipschitz-continuous function and the system \eqref{KM2} becomes the Kuramoto model with regular weights depending on the phase configuration. Then, the well--posedness of global-in-time classical solutions is standard. However, in the singular case \eqref{C-2}, the system \eqref{KM2} has a singular weight and we must deal with non-Lipschitz right hand side, where the Cauchy-Lipschitz theorem cannot guarantee the existence and uniqueness of global-in-time solutions. We will deal with three different regimes of the singularity $\alpha\in (0,\frac{1}{2})$, $\alpha=\frac{1}{2}$ and $\alpha\in (\frac{1}{2},1)$ that we respectively call the subcritical, critical and supercritical cases.

The main results of this paper are listed as follows. First we study, the well-posedness of the singular weighted system. Depending on the value of $\alpha$, the properties of the right hand side of \eqref{KM2} vary. Specifically, in the subcritical regime, we deal with systems of ODEs with H\"{o}lder-continuous right hand side while we face discontinuous right hand side of both bounded and unbounded type in the critical and supercritical cases. In addition, the type of uniqueness that we can expect in these systems is one-sided. Namely, a cluster of phases may eventually arise after a finite-time collision and oscillators belonging to such cluster might stay stuck together. This is a phenomenon that was recently found in other types of agent-based systems like the Cucker--Smale model with singular weights, see \cite{P1,P2}. 

Our second result characterizes the explicit conditions for sticking in the subcritical and critical regimes. In the former case, we show that only clusters of oscillators with the same natural frequencies can stick together. Nevertheless, in the latter case, cluster of oscillators with different natural frequencies may stick together as long as such frequencies fulfill an appropriate condition. Regarding the supercritical case, the analogue sticking condition becomes trivial and we can show a continuation procedure of classical solutions after finite-time collisions. Namely, after a cluster is formed in finite time, the cluster keep stuck together no matter which are the natural frequencies of the involved oscillators. 

The third  result consists in showing that these singular weights are physically relevant. Specifically, we will show that the system \eqref{KM2}-\eqref{C-2} with singular weights can be obtained as a rigorous singular limit of the regular model \eqref{C-1}-\eqref{KM2}. Again, the strategy will differ in each of the regimes. For the subcritical case, similar tools to those in \cite{P1,P2} for the singular Cucker--Smale model can be adapted. What is more, we can even obtain an analogue gain of extra $W^{1,1}$ piece-wise regularity of the frequencies of oscillators. For the critical and supercritical cases we cannot resort on the same ideas. Hence, we use the underlying gradient-flow structure to gain compactness of frequencies. Identifying the limit will be the heart of the matter in this part. 

Our last result faces the emergence of synchronization in each regime of the parameter $\alpha$. For identical oscillators, we show the emergence of complete phase synchronization in finite-time under appropriate assumptions on the initial diameter of phases. At least in the subcritical regime, where frequencies become more regular,  we study the asymptotic emergence of complete frequency synchronization of non-identical oscillators. Also, we study the stability properties of collision-less phase-locked states in all the three regimes.

The techniques are firstly inspired on a combination of results for the classical Kuramoto model, but these techniques require of a new perspective allowing for singular interactions. In this purpose, we introduce a well--posedness result ``\`a la Filippov'' that is valid for systems of ODEs with discontinuous right-hand sides. Specifically, we will rely on the study of absolutely continuous solutions of the differential inclusions associated with the Filippov's set-valued map. The values of such map are convex polytopes that are bounded and unbounded in the critical and supercritical case respectively. Hence, the classical theory can be used in the former case whereas new ideas are developed for the latter case. Also, we prove some one-sided uniqueness results for non--Lipschitzian interactions that rely on the structure of interaction kernel near the points of loss of Lipschitz-continuity. For the stability of equilibria, Lyapunov's first method entails a similar scenario to that of the classical Kuramoto model in the critical and supercritial regime. On the other hand, the subcritical regime requires a center manifold approach that yields the stability of the corresponding equilibria. What is more interesting is that we can still get some accurate control of the diameter of the system of singularly weighted coupled oscillators. Such control amounts to the corresponding finite-time and asymptotic synchronization for the identical and non-identical cases. Unfortunately, the emergence of phased-locked states independently of the initial configurations cannot be derived as in previous results for the classical Kuramoto model (see \cite{H-L-X}) because it is not clear whether the \L ojasiewicz gradient inequality \cite{Lo} holds for non-analytic systems with gradient structure like this. Regarding the singular limit of the regular coupling weights, the main goal is to prove that solutions of the regularized system converge towards absolutely continuous trajectories that fulfill the differential inclusion. For that, an appropriate H-representation (half-space representation) of such convex polytopes is obtained through convex analysis techniques. Then, the preceding gain of compactness of frequencies along with such geometric representation of the Filippov map will provide the necessary tools for the singular limit to work in the critical and supercritical regimes.

The rest of the paper is organized as follows. In Section \ref{preliminaries-section}, we present definitions, basic properties of the weighted Kuramoto model, the underlying gradient-flow structure, the passage from regular to singular plasticity function and the expected macroscopic equations. In Section \ref{well-posedness-singular-section}, we study the system with singular weights and we prove the well-posedness theory in each regime. In Section \ref{singular-limit-section}, we prove the rigorous singular limit in every regime and compare the model with previous results derived in other agent-based systems, in particular we compare with Cucker--Smale models. In Section \ref{synchro-singular-section}, we show the synchronization for the singular weighted system.  In Appendix \ref{regular-appendix}, we will recover some classical tools of the Kuramoto models that we apply to show the emergence of synchronization in the regular weighted system for the sake of clarity. Appendix \ref{Appendix-Hrep-Fil-map} shows the proofs of the H-representation of the Filippov set-valued map in the critical and supercritical cases. Finally, Appendix \ref{Appendix-sticking} introduces the explicit characterization of the sticking conditions.

\section{Preliminaries}\label{preliminaries-section}
\setcounter{equation}{0}

\subsection{Basic properties and definitions}
In this section, we study the basic properties of the weighted Kuramoto system and introduce some related results that will be useful in the following sections.
For simplicity, let us denote the interaction kernel by $h(\theta) := \Gamma(\theta)\sin\theta$ (here $\Gamma$ can be any even function, e.g., \eqref{C-1} or \eqref{C-2}). Then the system \eqref{KM2} can be expressed as
\begin{equation}\label{KM3}
\dot\theta_i = \Omega_i + \frac{K}{N} \sum_{j=1}^N h(\theta_j -\theta_i).
\end{equation}
For simplicity, we shall sometimes use vector notation in \eqref{KM3}. We define the vector field $H=H(\Theta)=(H_1(\Theta),\ldots,H_N(\Theta))$ whose components read
\begin{equation}\label{E-H}
H_i(\Theta)=\Omega_i+\frac{K}{N}\sum_{j=1}^Nh(\theta_j-\theta_i).
\end{equation}
Then, \eqref{KM3} can be restated as
\begin{equation}\label{KM3-vector}
\dot{\Theta}=H(\Theta).
\end{equation}
Since $h$ is an odd function, by taking sums on both sides of \eqref{KM3}, we have
\[
\sum_{j=1}^N \dot \theta_i = \sum_{j=1}^N \Omega_i,
\]
i.e., the average of frequencies is conserved. Thus, without loss of generality, we may assume that the average of natural frequencies is zero, $\frac{1}{N} \sum_{j=1}^N \Omega_i = 0$, in order to focus on the fluctuation from the constant average motion.

For the discussion in Section \ref{singular-limit-section}, we briefly introduce the second order augmentation of Kuramoto model, see \cite{H-K-P-Z}. By taking one more derivative on the system \eqref{KM3}, we have the second order model 
\begin{equation}\label{KM4}
\begin{dcases}
\ \dot\theta_i =& \hspace{-0.2cm}\omega_i, \\
\ \dot\omega_i =& \hspace{-0.2cm}\frac{K}{N}\sum_{j=1}^N h^\prime(\theta_j - \theta_i) (\omega_j - \omega_i).
\end{dcases}
\end{equation}
For both systems \eqref{KM3} and \eqref{KM4} we have the following equivalence.

\begin{theorem}\label{T-equiv}
The Kuramoto model \eqref{KM3}  is equivalent to an augmented Kuramoto model \eqref{KM4} in the following sense.
\begin{enumerate}
\item
If $\Theta = (\theta_1, \ldots, \theta_N )$ is a solution to \eqref{KM3} with initial data $\Theta_0$, then $( \Theta, \omega:= \dot \Theta)$ is a solution to \eqref{KM4} with well-prepared initial data $(\Theta_0, \omega_0)$:
\[ \omega_{i,0} :=  \Omega_i  +  \frac{\kappa}{N} \sum_{j=1}^N h(\theta_{j,0} - \theta_{i,0}). \]
\item
If $(\Theta, \omega)$ is a solution to \eqref{KM4} with initial data $(\Theta_0, \omega_0)$, then $\Theta$ is a solution to \eqref{KM3} with natural frequencies:
\[  \Omega_i :=   \omega_{i,0}  -  \frac{\kappa}{N} \sum_{j=1}^N h(\theta_{j,0} - \theta_{i,0}).  \] 
\end{enumerate}
\end{theorem}

For the regular cases \eqref{C-1}, the proof can be found in \cite{H-K-P-Z}. However, one has to take a special care with the time regularity of solutions in the singular cases \eqref{C-2} before we take derivatives in \eqref{KM3}. In that later case with $\alpha\in (0,\frac{1}{2})$, the type of solutions to be considered for \eqref{KM3} are absolutely continuous solutions while, for \eqref{KM4}, solutions have to be taken in weak sense with $C^1$ and piecewise $W^{2,1}$ regularity (see \cite{P1} for this concept of solution for the discrete Cucker--Smale model with singular influence function). The well--posedness of both singular systems \eqref{KM3} and \eqref{KM4} will be established in Sections \ref{well-posedness-singular-section} and \ref{singular-limit-section} (see Theorems \ref{T-wp-1}, \ref{T-wp-2}, \ref{singular-limit-subcritical-theorem}, \ref{singular-limit-W21-theorem} and Remark \ref{singular-limit-subcritical-remark}) and comparisons with Cucker--Smale models with singular influence function will be given in Subsection \ref{comparison-singular-CS-subsection}.

For the sake of completeness, we recall the different definitions of synchronization, \cite{H-H-K}.
\begin{definition}\label{D-sync}
Let $\Theta(t) = (\theta_1(t), \ldots, \theta_N(t))$ be the phase configuration of oscillators of which the dynamics is governed by the system \eqref{KM2}.
\begin{enumerate}
\item The system shows the complete phase synchronization asymptotically if, and only if, the following condition holds:
\[
\lim_{t\to\infty} |\theta_i(t) -\theta_j(t)| = 0, \quad \text{for all} \quad i\neq j.
\]
\item The system shows the complete frequency synchronization asymptotically if, and only if,  the following condition holds:
\[
\lim_{t\to\infty} |\dot\theta_i(t) -\dot\theta_j(t)| = 0, \quad \text{for all} \quad i\neq j.
\]
\item The system shows the emergence of a phase-locked state asymptotically if, and only if, there exist constants $\theta_{ij}^\infty$ such that
\[
\lim_{t\to\infty} |\theta_i(t) -\theta_j(t)| = \theta_{ij}^\infty, \quad \text{for all} \quad i\neq j.
\]
\end{enumerate}
Analogue definitions of synchronization will be considered if, instead of asymptotically, the emergent dynamics takes place in some finite time $T$. In such case $\infty$ will be replaced by such finite time $T$ in the above definitions.
\end{definition}
We note that the complete phase-synchronization is a special case of phase-locked state. It is obvious that if the solution shows the emergence of phase-locked state, then it implies the complete frequency synchronization. However, the converse is valid when the frequency synchronization occurs fast, i.e., integrable decay of frequency differences.

\subsection{Singular weighted model}\label{formal-singular-limit-subsection}
In this part, we introduce the formal derivation of the Kuramoto model with singular weights as singular limit of the regular weighted model. We note that the regular weighted model is \eqref{KM3} with interaction kernel given by
\[
h(\theta):=\frac{\sigma^{2\alpha}\sin\theta}{(\sigma^2+c_{\alpha,\zeta} |\theta|_o^{2})^{\alpha}}.
\]
Recall that the degree of connectedness is smaller than $\zeta$ for interparticle distances larger than $\sigma$ and $\alpha$ imposes the fall-off of the interactions. Consequently, $\sigma$ measures the effective range of interactions. Similarly, the parameter $K$ measures the maximum strength of interactions. Hence, one can propose the following scaling
\[
\sigma=\mathcal{O}(\varepsilon),\ \ K\sigma^{2\alpha}=\mathcal{O}(1),\ \ \mbox{ when }\varepsilon\rightarrow 0.
\]
Or more specifically, using the change of variables 
\[
\sigma\rightarrow\varepsilon\ \mbox{ and }\ K\rightarrow K\varepsilon^{-2\alpha},
\]
where $\varepsilon$ is a dimensionless parameter, we arrive at the next scaled system
\begin{equation}\label{KM-scaled}
\dot{\theta}_i=\Omega_i+\frac{K}{N}\sum_{j=1}^Nh_\varepsilon(\theta_j-\theta_i),
\end{equation}
where the scaled interaction kernel now reads
\begin{equation}\label{C-scaled}
h_\varepsilon(\theta):=\frac{\sin\theta}{(\varepsilon^2+c_{\alpha,\zeta}|\theta|_o^2 )^\alpha}.
\end{equation}
If we formally take limits when $\varepsilon\rightarrow 0$, then we arrive at the desired singular weighted Kuramoto model, whose singular interaction kernel is
\[
h(\theta):=\frac{\sin\theta}{c_{\alpha,\zeta}^\alpha | \theta |_o^{2\alpha}}.
\]
All these arguments are heuristic. However they might become rigorous depending on the value of $\alpha$. For a rigorous derivation of the singular limit in all the subcritical, critical and supercritical regimes, see Section \ref{singular-limit-section}.

\subsection{Emergence of clusters: collision and sticking of oscillators}\label{cluster-notation-subsection}
In this part we introduce some notation that will be used along the whole paper. We will denote the set of pair-wise collisions of the $i$-th and $j$-th oscillators by
\[
\mathcal{C}_{ij}:=\{\Theta\in \mathbb{R}^N:\,\bar\theta_i=\bar\theta_j\},
\]
where $\bar\theta$ denotes again the representative of $\theta$ in $(-\pi,\pi]$. Then, the set of collisions reads
\[
\mathcal{C}:=\bigcup_{i\neq j}\mathcal{C}_{ij}=\{\Theta\in\mathbb{R}^N:\,\exists\,i\neq j\mbox{ such that }\bar\theta_i= \bar\theta_j\}.
\]
Consider any phase configuration of the $N$ oscillators, i.e.,
\[
\Theta=(\theta_1,\ldots,\theta_N)\in\mathbb{R}^N.
\]
We will say that the $i$-th oscillator \textit{collides} with $j$-th oscillator when $\Theta\in \mathcal{C}_{ij}$ and we will say that $\Theta$ is a \textit{collision state} when $\Theta\in \mathcal{C}$. In order to manage with collisions, let us define the following binary relation
\[
i\overset{\Theta}{\sim}j\ \mbox{ when }\ \Theta\in \mathcal{C}_{ij}.
\]
Since it is an equivalence relation, we can denote its equivalence classes by
\begin{equation}\label{E-collision-classes}
\mathcal C_i(\Theta):=\{j\in\{1,\ldots,N\}:\,i\overset{\Theta}{\sim}j\}=\{j\in\{1,\ldots,N\}:\,\Theta\in \mathcal{C}_{ij}\}.
\end{equation}
As it is apparent from the definition, $\mathcal C_i(\Theta)$ is the set of \textit{indices of collision} with the $i$-th oscillator. Then, $\Theta$ is a collision state when some of its equivalence classes is non-trivial. Consequently, each of the equivalence classes can be regarded as a \textit{cluster of oscillators}. Let us denote by $\mathcal{E}(\Theta)$ the family of all the different equivalence classes that is, clusters. It is apparent that $\mathcal{E}(\Theta)$ establish a partition of $\{1,\ldots,N\}$, that we will call the \textit{collisional type} of $\Theta$. For simplicity of notation, we will enumerate the equivalence classes
\[
\mathcal{E}(\Theta)=\{E_1(\Theta),\ldots,E_{\kappa(\Theta)}(\Theta)\},\]
in such a way that the minimal representatives in each of them, i.e., $\iota_k(\Theta):=\min E_k(\Theta)$, are increasingly ordered. $\kappa(\Theta):=\#\mathcal{E}(\Theta)$ will denote the total amount of clusters in such a phase configuration $\Theta$ and we will denote the size of the $k$-th cluster, that is the number of particles which form the $k$-th cluster, by $n_k(\Theta):=\# E_k(\Theta)$, for each $k=1,\ldots,\kappa(\Theta)$.

Assume now that not only do we know some phase configuration at a particular time, but a whole absolutely continuous trajectory $t\mapsto\Theta(t)=(\theta_1(t),\ldots\theta_N(t))\in\mathbb{R}^N$ governing the dynamics of the $N$ oscillators. Then, as long as it is clear from the context, we will simplify the notation and will denote
\[
\mathcal C_i(t):=\mathcal C_i(\Theta(t)),\ \mathcal{E}(t):=\mathcal{E}(\Theta(t)),\ \kappa(t):=\kappa(\Theta(t)),\ n_k(t):=n_k(\Theta(t)).
\]
Similarly, time may be omitted in our notation for simplicity. Apart form collisions into clusters, it is important to characterize when those clusters remain stuck together. If the $i$-th and $j$-th oscillators have collided at time $t$, we will say that they \textit{stick together} when 
\[
\bar\theta_i(s)=\bar\theta_j(s),\ \mbox{ for all }s\geq t.
\]
Then, we can define the set of \textit{indices of sticking} with the $i$-th oscillator by
\begin{equation}\label{E-sticking-classes}
S_i(t):=\{j\in \mathcal C_i(t):\,\bar\theta_i(s)=\bar\theta_j(s),\ \mbox{ for all }s\geq t\}.
\end{equation}
In Section \ref{well-posedness-singular-section} we will introduce some results about the clustering and sticking behavior of solutions to our singular weighted Kuramoto model \eqref{KM-scaled} with $\varepsilon=0$.

\subsection{Gradient flow structure}\label{gradient-flow-subsection}
In this part, let us remark that our system \eqref{KM3} can be equivalently turned into a gradient flow system:
\begin{equation}\label{E-gradient-flow}
\dot\Theta = -\nabla V(\Theta),
\end{equation}
governed by a potential $V=V(\Theta)$ that is defined by
\begin{equation}\label{E-potential}
V(\Theta) = -\sum_{i=1}^N \Omega_i \theta_i +V_{int}(\Theta):=-\sum_{i=1}^N\Omega_i\theta_i+ \frac{K}{2N} \sum_{i\neq j} W(\theta_j - \theta_i).
\end{equation}
Here, $W$ is the primitive function of $h$ such that $W(0)=0$, i.e.,
\begin{equation}\label{E-W}
W(\theta):=\int_0^\theta h(\theta')\,d\theta'.
\end{equation}
The function $W$ can be regarded as the interaction potential of binary interactions while $V_{int}$ stands for the total interaction potential due to binary interactions. This approach is obviously formal and relies on specifying the regularity of the plasticity function $\Gamma$. For instance, if we choose $\Gamma$ to be analytic, then \eqref{KM3} can be regarded as a gradient flow system with analytic potential $V$. In such particular case, one can oversimplify the proof of emergence of synchronization like in the classical Kuramoto model, see \cite{H-K-R}. Specifically, some boundedness property of the trajectory is all we need to ensure the exponential convergence towards a phase-locked state by virtue of the \L ojasiewicz inequality for analytic functions. For the choices of plasticity function of interest in this paper, i.e., \eqref{C-1} and \eqref{C-2}, analyticity is missing and the same approach does not necessarily work. Nevertheless, we will focus on values of the parameter $\alpha$ that belong to the range $\alpha\in (0,1)$ and, consequently, $V$ will be globally a continuous function that is smooth outside the set of collisions. Since in general we are missing either analyticity or convexity of $V$, the gradient flow structure will not be used along this paper, except in Subsections \ref{singular-limit-supercritical-subsection} and \ref{singular-limit-critical-subsection}.

\subsection{Kinetic formulation of the problem}
In this part, we formally introduce the expected kinetic models associated with \eqref{KM-scaled}. The classical arguments to rigorously prove the mean field limit $N\rightarrow \infty$ are based on the analysis of propagation of chaos in the system as the number $N$ of particles becomes large, see \cite{J-W, L}. On the one-hand,  for every $\varepsilon>0$ the mean field limit is governed by the following Vlasov--McKean equation with regular kernels for the distribution function $f_\varepsilon=f_\varepsilon(t,\theta,\Omega)$ of oscillators
\begin{equation}\label{KM-kinetic-regular}
\frac{\partial f_\varepsilon}{\partial t} +\frac{\partial}{\partial \theta}\left[(\Omega-K h_\varepsilon\ast_{(\theta,\Omega)}f_\varepsilon)f_\varepsilon\right]=0,\ \ t\in\mathbb{R}_0^+,\,\theta\in [0,2\pi],\,\Omega\in \mathbb{R},
\end{equation}
where periodic boundary conditions in the variable $\theta$ are assumed. Similarly, when $\varepsilon=0$ the corresponding mean field limit is governed by the corresponding Vlasov--McKean equation with singular kernels for $f(t,\theta,\Omega)$, namely,
\begin{equation}\label{KM-kinetic-singular}
\frac{\partial f}{\partial t}+\frac{\partial}{\partial \theta}\left[(\Omega-K h\ast_{(\theta,\Omega)} f)f\right]=0,\ \ t\in\mathbb{R}_0^+,\,\theta\in [0,2\pi],\,\Omega\in \mathbb{R},
\end{equation}
with analogous periodic conditions in $\theta$. The derivation of the mean field limit is much more involved in this latter case. Indeed, it requires a sharp analysis,  in the same sense as in related singular models like \cite{C-C-H-2,H-J,M-P}, and will differ for each of the regimes of the exponent $\alpha$. Let us briefly comment on the main idea supporting the above mean field limit. Fix the following empirical measure as initial condition in \eqref{KM-kinetic-singular}
\[\mu^N_0(\theta,\Omega)=\frac{1}{N}\sum_{i=1}^N\delta_{\theta_{i,0}^N}(\theta)\delta_{\Omega_i^N}(\Omega),\]
associated to some initial configuration $\Theta_0^N=(\theta_{1,0}^N,\ldots,\theta_{N,0}^N)$. Because of the results in this paper, the Filippov solution $\Theta^N(t)=(\theta_1^N(t),\ldots,\theta_N^N(t))$ to the singular discrete model allows considering the next measure-valued solution to \eqref{KM-kinetic-singular}
\[\mu^N_t(\theta,\Omega)=\frac{1}{N}\sum_{i=1}^N\delta_{\theta_i^N(t)}(\theta)\delta_{\Omega_i^N}(\Omega).\]
The ultimate effort to be done is to show that any weak limit $f$ of $\mu^N$ consists in a measure-valued solution in some generalized sense to the singular macroscopic system. For a comprehensive analysis of the singular macroscopic model \eqref{KM-kinetic-singular} see \cite{P-P-S-2}. See \cite{M-P}, where a close approach has been developed in the Cucker--Smale model for the smaller range o parameters $\alpha\in (0,\frac{1}{4})$ of the subcritical regime. Analogue results in aggregation models and classical Kuramoto model has been studied in \cite{C-C-H-2, C-D-F-L-S,C-J-L-V} and \cite{C-C-H-K-K,L} respectively.

\section{Well--posedness of singular interaction}\label{well-posedness-singular-section}
\setcounter{equation}{0}
We now consider the Kuramoto model with singular coupling $\Gamma$, that we introduced in Section \ref{preliminaries-section} as a singular limit of regular weighted coupling
\begin{equation}\label{E-1}
\dot \theta_i = \Omega_i + \frac{K}{N}\sum_{j=1}^N \frac{\sin(\theta_j -\theta_i)}{|\theta_j - \theta_i|_o^{2\alpha}} \qquad i=1, \ldots, N.
\end{equation}
Recall that in the limit $\varepsilon\rightarrow 0$ of the regular kernel $h_\varepsilon$ we recover the singular interaction kernel of the model, i.e.,
\[
h(\theta):= \frac{\sin \theta }{|\theta|_o^{2\alpha}}.
\]
For simplicity, we will forget about the constant $c=c_{\alpha,\zeta}=1-\zeta^{-1/\alpha}$. Then, we can rewrite the system \eqref{E-1} into
\begin{equation}\label{E-2}
\dot\theta_i = \Omega_i + \frac{K}{N} \sum_{j=1}^N h(\theta_j -\theta_i) \qquad i=1, \ldots, N.
\end{equation}
Regarding the parameter $\alpha$, it belongs to the interval $(0,1)$ to allow for mild singularities. Note that the kernel is continuous for $\alpha\in (0,\frac{1}{2})$, it exhibits a jump discontinuity at $\theta\in 2\pi\mathbb{Z}$  for $\alpha=\frac{1}{2}$, and it shows essential discontinuities for $\alpha\in (\frac{1}{2},1)$, see Figure \ref{fig:alpha}.

\begin{figure}
\centering
\begin{subfigure}[b]{0.48\textwidth}
\begin{tikzpicture}
\begin{axis}[
  axis x line=middle, axis y line=middle,
  xmin=-pi, xmax=pi, xtick={-3,...,3}, xlabel=$\theta$,
  ymin=-1.4, ymax=1.4, ytick={-1,-0.5,0,0.5,1},
]
\addplot [
    domain=-pi:pi, 
    samples=200, 
    color=black,
]
{sin(deg(x))/pow(abs(x),2*0.25)};
\end{axis}
\end{tikzpicture}
\caption{$\alpha = 0.25$}
\label{fig:alpha1}
\end{subfigure}
\begin{subfigure}[b]{0.48\textwidth}
\begin{tikzpicture}
\begin{axis}[
  axis x line=middle, axis y line=middle,
  xmin=-pi, xmax=pi, xtick={-3,...,3}, xlabel=$\theta$,
  ymin=-1.4, ymax=1.4, ytick={-1,-0.5,0,0.5,1},
]
\addplot [
    domain=-pi:0, 
    samples=500, 
    color=black,
]
{sin(deg(x))/pow(abs(x),2*0.5)};
\addplot [
    domain=0:pi, 
    samples=500, 
    color=black,
]
{sin(deg(x))/pow(abs(x),2*0.5)};
 
\end{axis}
\end{tikzpicture}
\caption{$\alpha = 0.5$}
\label{fig:alpha2}
\end{subfigure}
\begin{subfigure}[b]{0.48\textwidth}
\begin{tikzpicture}
\begin{axis}[
  axis x line=middle, axis y line=middle,
  xmin=-pi, xmax=pi, xtick={-3,...,3}, xlabel=$\theta$,
  ymin=-25, ymax=25, ytick={-20,-10,0,10,20},
]
\addplot [
    domain=-pi:pi, 
    samples=200, 
    color=black,
]
{sin(deg(x))/pow(abs(x),2*0.75)};
 
\end{axis}
\end{tikzpicture}
\caption{$\alpha = 0.75$}
\label{fig:alpha3}
\end{subfigure}
\caption{Plot of the interaction kernel $h=h(\theta)$ in Equation \eqref{E-1} for the values $\alpha=0.25$, $\alpha=0.5$ and $\alpha=0.75$, respectively.}\label{fig:alpha}
\end{figure}
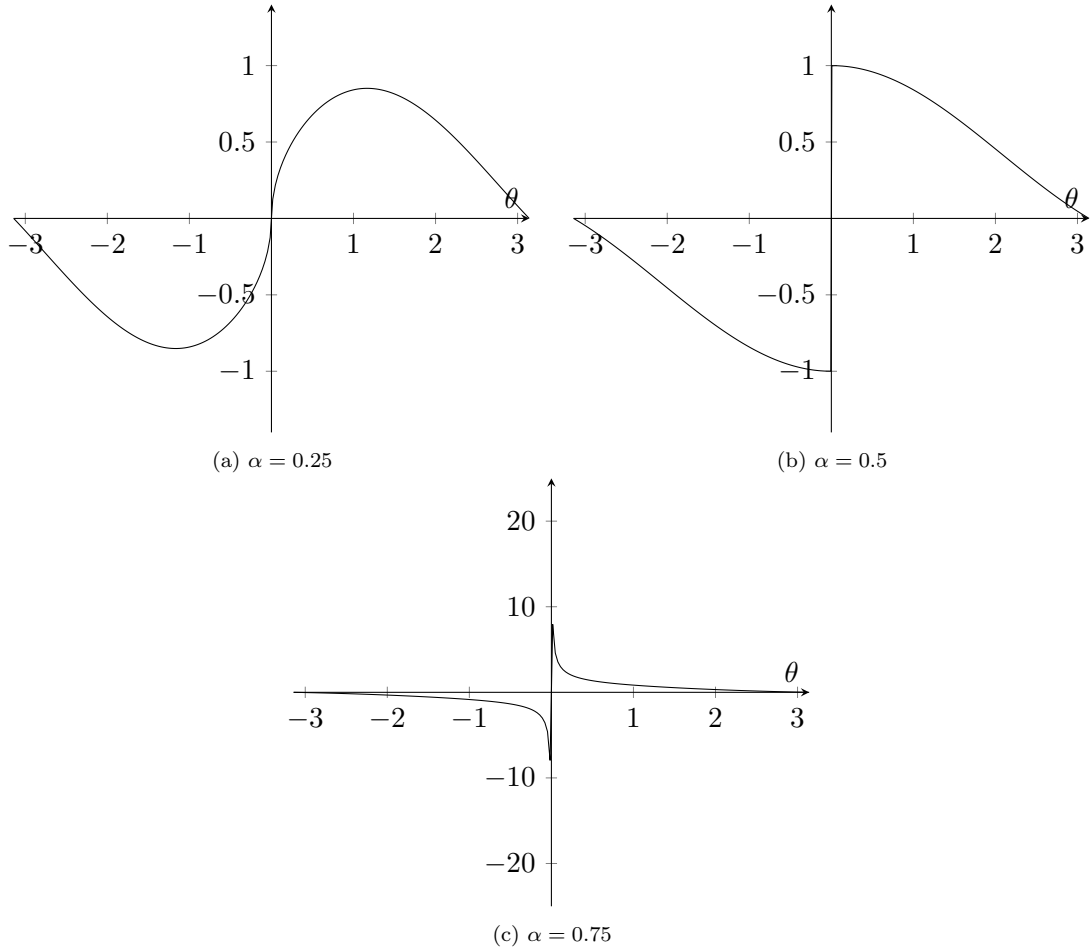

In this section, we will focus on developing the well--posedness theory of such system \eqref{E-1} of coupled ODEs. Note that the uniqueness is not a trivial even in the subcritical case. Indeed, due to the choice of singular plasticity function, the right hand side of the system \eqref{E-2} does not satisfy Lipschitz-continuity in any of the subcritical, critical and supercritical regimes. Thus, we need to inspect the existence and uniqueness of the solution to the system \eqref{E-1} before we proceed the study of synchronization. For the following discussion, we recall the definition of the vector field $H=H(\Theta)$ in \eqref{E-H} that allows dealing with the system \eqref{E-2} in the vector form \eqref{E-H}.

\subsection{Well--posedness in the subcritical regime}\label{well-posedness-subcritical-subsection}
In the subcritical case, namely $\alpha\in \left(0,\frac{1}{2}\right)$, the vector field $H=H(\Theta)$ in \eqref{E-H} is continuous. Therefore, it is a clear consequence of Peano's theorem that \eqref{E-1} has a local-in-time solution for every initial configuration $\Theta(0)=\Theta_0\in \mathbb R^N$. Unfortunately, note that $h(\theta)$ exhibits an infinite slope at the phase values $\theta\in 2\pi\mathbb{Z}$ and then, the classical Cauchy--Picard--Lindel\"{o}v theorem does not apply since $H=H(\Theta)$ is no longer a Lipschitz-continuous vector field. Nevertheless, one can still use an easy trick: it is enough to show that near the points of loss of Lipschitz-continuity our vector field can be locally split into the sum of a decreasing vector field and a Lipschitz-continuous vector field, then ensuring the local one-sided Lipschitz condition that is enough to obtain a one-sided uniqueness result. 
\begin{lemma}\label{L-Lip}
Let $F:\mathbb R^N\longrightarrow \mathbb R^N$ be a bounded and continuous vector field and assume that for every $x^*\in\mathbb R^N$ there exists some open neighborhood $V\subseteq \mathbb R^N$ and a positive constant $M$ so that $F$ verifies the one-sided Lipschitz condition in $V$
\[(F(x)-F(y))\cdot (x-y)\leq M |x-y|^2,\]
for every couple $x,y\in V$. Then, the following initial value problem (IVP) associated with any initial configuration $x_0\in \mathbb R^N$ enjoys one global-in-time solution, that is unique forward in time
\[
\begin{dcases}
\dot{x}=F(x), & t\geq 0,\\
x(0)=x_0.&
\end{dcases}
\]
\end{lemma}

Since the proof is classical, we omit it here. Let us now apply such result to our case of interest. To do so, it is enough to introduce a decomposition of the vector field $H=H(\Theta)$ in the Kuramoto model \eqref{E-2}. We first set the following split of the interaction function $h=h(\theta)$. First, consider $\overline{h}$ and $\tilde{\theta}\in \left(0,\frac{\pi}{2}\right)$ such that
\[\overline{h}:=\max_{0<r<\pi}h(r)\ \mbox{ and }\ 2\alpha\sin\tilde{\theta}=\tilde{\theta}\cos \tilde{\theta}.\]
Note that $\tilde{\theta}$ is uniquely defined as the value in $\left(0,\pi\right)$ where $h$ attains its maximum. Second, define the couple of functions $f=f(\theta)$ and $g=g(\theta)$ in $(-\pi,\pi)$ as follows
\begin{align*}
f(\theta)&:=\begin{dcases}
\overline{h}, & \mbox{for }\theta\in (-\pi,-\tilde{\theta}),\\
-h(\theta), & \mbox{for }\theta\in [-\tilde{\theta},\tilde{\theta}),\\
-\overline{h}, & \mbox{for }\theta\in [\tilde{\theta},\pi),
\end{dcases}
\\
g(\theta)&:=\begin{dcases}-\overline{h}-h(\theta), & \mbox{for }\theta\in (-\pi,-\tilde{\theta}),\\
0, & \mbox{for }\theta\in [-\tilde{\theta},\tilde{\theta}),\\
\overline{h}-h(\theta), & \mbox{for }\theta\in [\tilde{\theta},\pi).
\end{dcases}
\end{align*}
Notice that
\begin{equation}\label{E-3}
-h(\theta)=f(\theta)+g(\theta),\ \ \mbox{ for all }\theta\in (-\pi,\pi),
\end{equation}
as depicted in Figure \ref{fig:kuramoto-decomposition-025}.

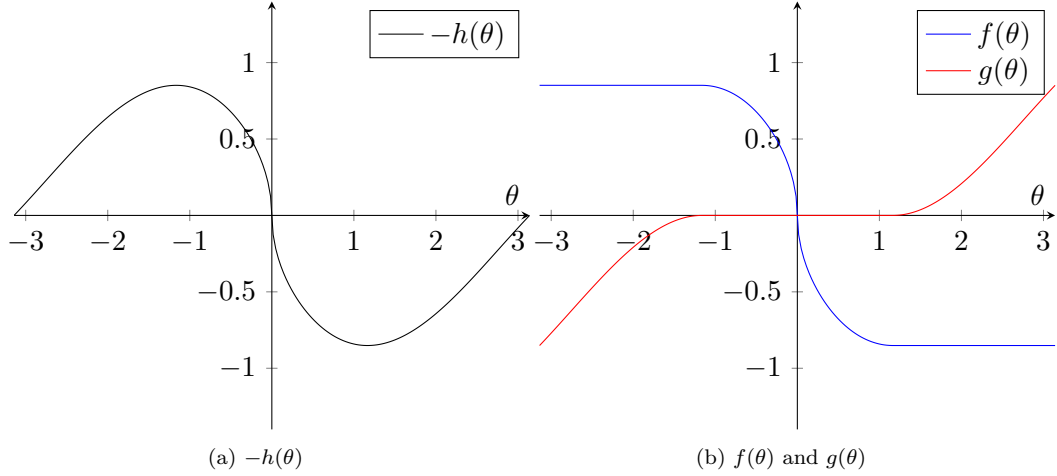
\begin{figure}
\centering
\begin{subfigure}[b]{0.45\textwidth}
\begin{tikzpicture}
\begin{axis}[
  axis x line=middle, axis y line=middle,
  xmin=-pi, xmax=pi, xtick={-3,...,3}, xlabel=$\theta$,
  ymin=-1.4, ymax=1.4, ytick={-1,-0.5,0,0.5,1},
]
\addplot [
    domain=-pi:pi, 
    samples=200, 
    color=black,
]
{-sin(deg(x))/pow(abs(x),2*0.25)};
\addlegendentry{$-h(\theta)$}
\end{axis}
\end{tikzpicture}
\caption{$-h(\theta)$}
\label{fig:kuramoto-decomposition-025-1}
\end{subfigure}
\begin{subfigure}[b]{0.45\textwidth}
\begin{tikzpicture}[
  declare function={
  f(\x)= (\x<=-1.165561185207211) * (sin(deg(1.165561185207211))/pow(1.165561185207211,2*0.25))  +
       and(\x>-1.165561185207211, \x<=1.165561185207211) * (-sin(deg(\x))/pow(abs(\x),2*0.25))  +
            (\x>1.165561185207211) * (-sin(deg(1.165561185207211))/pow(1.165561185207211,2*0.25));
  g(\x)= (\x<=-1.165561185207211) * (-sin(deg(1.165561185207211))/pow(1.165561185207211,2*0.25)-sin(deg(\x))/pow(abs(\x),2*0.25)))  +
       and(\x>-1.165561185207211, \x<=1.165561185207211) * (0)  +
            (\x>1.165561185207211) * (sin(deg(1.165561185207211))/pow(1.165561185207211,2*0.25)-sin(deg(\x))/pow(abs(\x),2*0.25)));
  }
]
\begin{axis}[
  axis x line=middle, axis y line=middle,
  xmin=-pi, xmax=pi, xtick={-3,...,3}, xlabel=$\theta$,
  ymin=-1.4, ymax=1.4, ytick={-1,-0.5,0,0.5,1},
]
\addplot[domain=-pi:pi, samples=200,color=blue]{f(x)};
\addlegendentry{$f(\theta)$}
\addplot[domain=-pi:pi, samples=200,color=red]{g(x)};
\addlegendentry{$g(\theta)$}
\end{axis}
\end{tikzpicture} 
\caption{$f(\theta)$ and $g(\theta)$}
\label{fig:kuramoto-decomposition-025-2}
\end{subfigure}
\caption{Graph of the function $-h(\theta)$ and the functions $f(\theta)$ and $g(\theta)$ in the decomposition for the value $\alpha=0.25$.}\label{fig:kuramoto-decomposition-025}
\end{figure}

\begin{remark}
Note that although $-h(\theta)$ is not a Lipschitz-continuous function because of the infinite slope at $\theta\in 2\pi\mathbb{Z}$, one can locally decompose it around such values in terms of a decreasing function $f(\theta)$ and a Lipschitz-continuous function $g(\theta)$.
\end{remark}

Finally, consider any value $\Theta^*=(\theta_1^*,\ldots,\theta_N^*)\in\mathbb R^N$ to locally decompose $H$ around it.
For $\Theta=(\theta_1,\ldots,\theta_N)$ in a small enough neighborhood $\mathcal{V}$ of $\Theta^*$ in $\mathbb R^N$, we set
\begin{align}
F_i(\Theta)&:=\frac{K}{N}\sum_{j\in \mathcal C_i(\Theta^*)}f\left(\overline{\theta_i-\theta_j}\right),\label{E-5}\\
G_i(\Theta)&:=\Omega_i+\frac{K}{N}\sum_{j\in \mathcal C_i(\Theta^*)}g\left(\overline{\theta_i-\theta_j}\right)-\frac{K}{N}\sum_{j\notin \mathcal C_i(\Theta^*)}h(\theta_i-\theta_j),\label{E-6}
\end{align}
where we recall that $\mathcal C_i(\Theta^*)$ stands for the set of indices of collision with the $i$-th oscillator in the phase configuration $\Theta^*$, see Subsection \ref{cluster-notation-subsection}.

\begin{proposition}\label{P-decomp-1}
Let $\Theta^*=(\theta_1^*,\ldots,\theta_N^*)\in\mathbb R^N$, and define the vector fields
\[F:\mathcal V \longrightarrow\mathbb R^N,\ \ G:\mathcal V \longrightarrow\mathbb R^N,\]
 via the formulas (\ref{E-5})-(\ref{E-6}), for a small enough neighborhood $\mathcal V$ of $\Theta^*$ in $\mathbb R^N$. Then,
\begin{enumerate}
\item $H=F+G$ in $\mathcal V$.
\item $F$ is decreasing in $\mathcal V$.
\item $G$ is Lipschitz-continuous in $\mathcal V$. 
\item $H$ is one-sided Lipschitz continuous in $\mathcal V$.
\end{enumerate}
\end{proposition}

\begin{proof}
The decomposition of $H$ is clear by virtue of the decomposition (\ref{E-3}) and the definitions (\ref{E-5})-(\ref{E-6}). Let us then focus on the last three properties. Fist, consider $\Theta=(\theta_1,\ldots,\theta_N),\widetilde{\Theta}=(\widetilde{\theta}_1,\ldots,\widetilde{\theta}_N)\in \mathbb R^N$ in a small enough neighborhood of $\Theta^*$. Without loss of generality, we will directly assume that $\theta_i-\theta_j$ and $\widetilde{\theta}_i-\widetilde{\theta}_j$ belong to $(-\pi,\pi]$. In other case, we just need to work with the representatives. On the one hand,
\[
(F(\Theta)-F(\widetilde{\Theta}))\cdot (\Theta-\widetilde{\Theta})=\frac{K}{N}\sum_{i=1}^N\sum_{j\in \mathcal C_i(\Theta^*)}(f(\theta_i-\theta_j)-f(\widetilde{\theta}_i-\widetilde{\theta}_j))(\theta_i-\widetilde{\theta}_i).
\]
Changing the indices $i$ and $j$ we obtain
\begin{align*}
(F(\Theta)-F(\widetilde{\Theta}))\cdot (\Theta-\widetilde{\Theta})&=\frac{K}{N}\sum_{j=1}^N\sum_{i\in \mathcal C_j(\Theta^*)}(f(\theta_j-\theta_i)-f(\widetilde{\theta}_j-\widetilde{\theta}_i))(\theta_j-\widetilde{\theta}_j)\\
&=-\frac{K}{N}\sum_{i=1}^N\sum_{j\in \mathcal C_i(\Theta^*)}(f(\theta_i-\theta_j)-f(\widetilde{\theta}_i-\widetilde{\theta}_j))(\theta_j-\widetilde{\theta}_j),
\end{align*}
where the properties of the sets $\mathcal C_i(\Theta^*)$ along with the antisymmetry of $f$ have been used in the last line. Taking the mean value of both expressions and using that $f$ is decreasing, we arrive at
\[
(F(\Theta)-F(\widetilde{\Theta}))\cdot (\Theta-\widetilde{\Theta})=\frac{K}{2N}\sum_{i=1}^N\sum_{j\in \mathcal C_i(\Theta^*)}(f(\theta_i-\theta_j)-f(\widetilde{\theta}_i-\widetilde{\theta}_j))((\theta_i-\theta_j)-(\widetilde{\theta}_i-\widetilde{\theta}_j))\leq 0,
\]
and, as a consequence, to the monotonicity of $F$. On the other hand,
\[
| G_i(\Theta)-G_i(\widetilde{\Theta})|\leq \frac{K}{N}\sum_{j\in \mathcal  C_i(\Theta^*)}| g(\theta_i-\theta_j)-g(\widetilde{\theta}_i-\widetilde{\theta}_j)| +\frac{K}{N}\sum_{j\notin \mathcal C_i(\Theta^*)}| h(\theta_i-\theta_j)-h(\widetilde{\theta}_i-\widetilde{\theta}_j)|.
\]
Since $g$ is Lipschitz-continuous in $(-\pi,\pi)$ and $h$ is locally Lipschitz-continuous in $(-\pi,\pi)\setminus\{0\}$, then there exists some constant $M=M(\mathcal V)$ so that
\[| G_i(\Theta)-G_i(\widetilde{\Theta})|\leq \frac{KM}{N}\sum_{j=1}^N| (\theta_i-\theta_j)-(\widetilde{\theta}_i-\widetilde{\theta}_j)|\leq \frac{N+1}{N}KM| \Theta-\widetilde{\Theta}|,\]
for every index $i\in\{1,\ldots,N\}$, thus yielding the Lipschitz-continuity of $G$ in $\mathcal V$.
The last part is a simple consequence: Namely, consider $x,y\in \mathcal V$ and note that
\[(H(x)-H(y))\cdot (x-y)=(F(x)-F(y))\cdot (x-y)+(G(x)-G(y))\cdot (x-y)\leq \frac{N+1}{N}KM |x-y|^2,\]
where the preceding two properties have been used along with the Cauchy--Schwartz inequality.
\end{proof}

Finally, putting together Lemma \ref{L-Lip} and Proposition \ref{P-decomp-1}, one concludes the following well--posedness property.
\begin{theorem}\label{T-wp-1}
There is one global-in-time strong solution to the system \eqref{E-2}, with $\alpha\in (0,\frac{1}{2})$,  which is unique forwards in time, for any initial configuration.
\end{theorem}

The next result is a simple consequence of the above well-posedness theorem and characterizes the eventual emergence of sticking in a cluster after a potential collision.

\begin{theorem}\label{T-sticking-1}
Consider $\Theta=(\theta_1,\ldots,\theta_N)$, the global-in-time solution in Theorem \ref{T-wp-1}. Assume that two oscillators collide at $t^*$, i.e., $\bar{\theta}_i(t^*)=\bar{\theta}_j(t^*)=\theta^*$ for some $i\neq j$. Then, the following two statements are equivalent:
\begin{enumerate}
\item $\theta_i$ and $\theta_j$ stick together at $t^*$.
\item Their natural frequencies agree, i.e.,
\begin{equation}\label{E-sticking-subcritical-explicit}
\Omega_i=\Omega_j.
\end{equation}
\end{enumerate}
\end{theorem}

\begin{proof}
Without loss of generality, let us assume that $i=1$, $j=2$ and $\theta_1(t^*)=\theta_2(t^*)\in (-\pi,\pi]$. Assume that the two particles keep stuck together after time $t^*$. Then, looking at the first two equations in system \eqref{E-2} it is clear that $\Omega_1=\Omega_2$. Conversely, let us assume that $\Omega_1=\Omega_2=:\Omega$ and consider the following system of $N-1$ ODEs.
\begin{align*}
\dot{\vartheta}&=\Omega+\frac{K}{N}\sum_{j=3}^N h(\vartheta_j-\vartheta),\\
\dot{\vartheta}_i&=\Omega_i+\frac{2K}{N}h(\vartheta-\vartheta_i)+\frac{K}{N}\sum_{j=3}^N h(\vartheta_j-\vartheta_i),\ \ i=3,\ldots,N,
\end{align*}
with initial data given by
\[
(\vartheta(t^*),\vartheta_3(t^*),\ldots,\vartheta_N(t^*))=(\theta^*,\theta_3(t^*),\ldots,\theta_N(t^*)).
\]
A similar technique to that in Theorem \ref{T-wp-1} clearly yields a global-in-time solution to such initial value problem. Hence, the following two trajectories in $\mathbb{R}^N$
\begin{align*}
t&\mapsto (\theta_1(t),\theta_2(t),\theta_3(t),\ldots,\theta_N(t)),\\
t&\mapsto (\vartheta(t),\vartheta(t),\vartheta_3(t),\ldots,\vartheta_N(t)),
\end{align*}
are both solutions to \eqref{E-2} such that at $t=t^*$ they take the value
\[
(\theta^*,\theta^*,\theta_3(t^*),\ldots,\theta_N(t^*)).
\]
By uniqueness they agree and, in particular, $\theta_1(t)=\vartheta(t)=\theta_2(t)$ for all $t\geq t^*$.
\end{proof}

\vspace{0.3cm}

\subsection{Well--posedness in the critical regime}\label{well-posedness-critical-subsection}
In the critical case, i.e. $\alpha=\frac{1}{2}$, the vector field $H=H(\Theta)$ is no longer continuous and the Peano existence theorem does not work. Nevertheless, in such case $H$ is still a measurable and essentially bounded vector field. Consequently, one can apply Filippov's existence criterion, see \cite{AubinCellina,F}.

\vspace{0.25cm}

We introduce the necessary notation that will be used here on: $2^{\mathbb{R}^N}$ stands for the power set of $\mathbb R^N$, $| \mathcal N|$ for the Lebesgue measure of any measurable set $\mathcal N\subseteq \mathbb{R}^N$, $\co(A)$ is the convex hull of $A$ and $\overline{\co}(A)=\overline{\co(A)}$ is its closure. For every convex set $C$ we denote by $m(C)$ the element of minimal norm of $C$, i.e. $m(C)=\pi_C(0)$, where $\pi_C$ is the orthogonal projection operator over the convex set $C$. The main ingredient will be the \textit{Fillipov set-valued map} of a given single-valued measurable map.

\begin{definition}\label{D-Fil-map}
Let $F:\mathbb R^N\longrightarrow\mathbb R^N$ be any measurable map. The Filippov set-valued map $\mathcal{F}:\mathbb{R}^N\longrightarrow 2^{\mathbb{R}^N}$ is defined for any $x\in \mathbb{R}^N$ as follows
\[\mathcal{F}(x):=\bigcap_{\delta>0}\bigcap_{ | \mathcal N|=0}\overline{\co}\left(F(B_\delta(x)\setminus \mathcal N)\right).\]
\end{definition}

The main interest in considering such map can be summarized in the next couple of results, see \cite[Theorem 2.1.3, Theorem 2.1.4, Proposition 2.1.1]{AubinCellina}.

\begin{lemma}\label{L-Fil-map-prop}
Let $F:\mathbb{R}^N\longrightarrow\mathbb{R}^N$ be any measurable map and consider its Filippov set-valued map $\mathcal{F}$. Then,
\begin{enumerate}
\item $\mathcal{F}(x)$ is a closed and convex set for every $x\in\mathbb{R}^N$.
\item $F(x)\in \mathcal{F}(x)$ for almost every $x\in\mathbb{R}^N$.
\item If $F$ is continuous at $x\in\mathbb{R}^n$, then $\mathcal{F}(x)=\{F(x)\}$.
\item If $\mathcal{F}$ takes non-empty values, then $\mathcal{F}$ is has closed graph.
\item If $\mathcal{F}$ has closed graph and $m(\mathcal{F})(U_x)$ lies in a compact set for some neighborhood $U_x$ of each $x\in\mathbb{R}^N$, then $\mathcal{F}$ is upper semicontinuous.
\item If $F$ is locally essentially bounded, then $\mathcal{F}$ is upper semicontinuous, it takes non-empty values and $m(\mathcal{F})(U_x)$ lies in a compact set for some neighborhood $U_x$ of each $x\in\mathbb{R}^N$. 
\item If $F$ is essentially bounded, then $\mathcal{F}$ is upper semicontinuous, it takes non-empty values and $m(\mathcal{F})(\mathbb{R}^N)$ lies in a compact set. 
\end{enumerate}
Here $m(\mathcal{F})$ stands for the map $m(\mathcal{F})(x):=m(\mathcal{F}(x))$ for every $x\in\mathbb{R}^N$.
\end{lemma}

\begin{lemma}\label{L-wp-diff-incl}
Let $\mathcal{F}:\mathbb R^N\longrightarrow 2^{\mathbb{R}^N}$ be any set valued-map with non-empty closed and convex values. Assume that $\mathcal{F}$ is upper semicontinuous and consider the following initial value problem (IVP) associated with any given initial datum $x_0\in\mathbb{R}^N$:
\[
\left\{\begin{array}{l}
\dot{x}\in \mathcal{F}(x),\\
x(0)=x_0.
\end{array}\right.
\]
\begin{enumerate}
\item If $m(\mathcal{F})(U_x)$ lies in a compact set for some neighborhood $U_x$ of any $x\in\mathbb{R}^N$, then (IVP) has an absolutely continuous local-in-time solution.
\item If $m(\mathcal{F})(\mathbb{R}^N)$ lies in a compact set, then (IVP) has an absolutely continuous global-in-time solution.
\end{enumerate}
\end{lemma}

Putting together Lemmas \ref{L-Fil-map-prop} and \ref{L-wp-Fil} we arrive at the next result.

\begin{lemma}\label{L-wp-Fil}
Let $F:\mathbb{R}^N\longrightarrow\mathbb{R}^N$ be any measurable map and consider its Filippov set-valued map $\mathcal{F}$. Consider the following initial value problem (IVP) associated with any given initial datum $x_0\in\mathbb{R}^N$:
\[
\left\{\begin{array}{l}
\dot{x}\in \mathcal{F}(x),\\
x(0)=x_0.
\end{array}\right.
\]
\begin{enumerate}
\item If $F$ is locally essentially bounded, then (IVP) has an absolutely continuous local-in-time solution.
\item If, in addition, $F$ is globally essentially bounded, then such a solution is indeed global.
\end{enumerate}
\end{lemma}

The solutions to such differential inclusion are called \textit{solutions in Filippov's sense} to the original discontinuous dynamical system. To deal with uniqueness we first introduce the next technical result.

\begin{lemma}\label{L-Lip-Fil}
Let $F:\mathbb R^N\longrightarrow\mathbb R^N$ be a measurable and locally essentially bounded map and consider its associated Filippov set-valued map $\mathcal{F}:\mathbb R^N\longrightarrow 2^{\mathbb{R}^N}$. If $F$ verifies the one-sided Lipschitz-condition a.e., then $\mathcal{F}$ also verifies it in the set-valued sense. Namely, there exists a positive constant $M$ such that
\[(X-Y)\cdot (x-y)\leq M| x-y|^2,\]
for every $x,y\in \mathbb R^N$ and every $X\in \mathcal{F}(x),Y\in \mathcal{F}(y)$.
\end{lemma}

\begin{proof}
Consider any couple $x,y\in \mathbb R^N$ and fix $X\in \mathcal{F}(x),Y\in\mathcal{F}(y)$.  Also fix any $\delta>0$ (assume $\delta<1$ without loss of generality) and any negligible set $\mathcal{N}$. Using the definition of $\mathcal{H}$, the following properties hold true
\[X\in \overline{\co}\big(F(B_\delta(x)\setminus \mathcal N)\big)\ \mbox{ and }\ Y\in \overline{\co}\big(F(B_\delta(y)\setminus \mathcal N)\big).\]
Then, one can take a couple of sequences $\{X_n\}_{n\in\mathbb{N}}\subseteq \mathbb{R}^N$ and $\{Y_n\}_{n\in \mathbb{N}}\subseteq \mathbb{R}^N$ such that $X_n\rightarrow X$, $Y_n\rightarrow Y$ and
\[X_n\in \co\big(F(B_\delta(x)\setminus \mathcal N)\big)\ \mbox{ and }\ Y_n\in \co\big(F(B_\delta(y)\setminus \mathcal N)\big),\]
for every $n\in\mathbb{N}$. Therefore, the Caratheodory theorem from convex analysis allows restating $X_n$ and $Y_n$ as a convex combination
\[
X_n=\sum_{i=1}^{N+1}\alpha_i^n F(x_i^n)\quad  \mbox{and} \quad Y_n=\sum_{j=1}^{N+1}\beta_j^n F(y_j^n),
\]
where $x_i^n\in B_\delta(x)\setminus \mathcal N$, $y_j^n\in B_\delta(y)\setminus \mathcal N$ and the coefficients $\alpha_i^n,\beta_j^n\in [0,1]$ verify
\[\sum_{i=1}^{N+1}\alpha_i^n=1=\sum_{j=1}^{N+1}\beta_j^n.\]
Note that
\[
X_n=\sum_{i=1}^{N+1} \sum_{j=1}^{N+1} \alpha_i^n\beta_j^n F(x_i^n)  \quad \text{and} \quad Y_n=\sum_{i=1}^{N+1} \sum_{j=1}^{N+1} \alpha_i^n \beta_j^n F(y_j^n).
\]
By defining the constants
\[
M_x:=\esssup_{z\in B_1(x)}\limits| F(z)|\ \mbox{ and }\ M_y:=\esssup_{z\in B_1(y)}\limits| F(z)|,
\]
we have
\begin{align*}
(X_n-Y_n)\cdot (x-y) &= \Big(\sum_{i,j=1}^{N+1}\alpha_i^n\beta_j^n\big(F(x_i^n)-F(y_j^n)\big) \Big)\cdot (x-y) \\
&=\sum_{i,j=1}^{N+1}\alpha_i^n\beta_j^n\Big(\big(F(x_i^n)-F(y_j^n)\big)\cdot (x-y)\Big)\\
&=\sum_{i,j=1}^{N+1}\alpha_i^n\beta_j^n\Big( \big(F(x_i^n)-F(y_j^n)\big)\cdot (x_i^n-y_j^n) \\
& \hspace{2.2 cm} + \big(F(x_i^n)-F(y_j^n)\big)\cdot \big((x-x_i^n)-(y-y_j^n)\big) \Big)\\
&\leq \sum_{i,j=1}^{N+1}\alpha_i^n\beta_j^n \Big( M | x_i^n-y_j^n|^2 +2(M_x+M_y)\delta \Big)\\
&\leq \sum_{i,j=1}^{N+1}\alpha_i^n\beta_j^n \Big( M (|x-y| + 2\delta)^2 +2(M_x+M_y)\delta \Big)\\
&=M (|x-y| + 2\delta)^2 +2(M_x+M_y)\delta .
\end{align*}
Since the above property holds for arbitrary $n\in\mathbb{N}$ and $0<\delta<1$, we obtain
\[(X-Y)\cdot (x-y)\leq M| x-y|^2.\]
\end{proof}

\begin{lemma}\label{L-IVP-Fil}
Let $F:\mathbb R^N\longrightarrow\mathbb R^N$ be a measurable and essentially bounded vector field and consider the Filippov set--valued map $\mathcal{F}:\mathbb R^N\longrightarrow 2^{\mathbb{R}^N}.$ In addition, assume that $F$ verifies the local one-sided Lipschitz condition. Then, the following initial value problem (IVP) associated with any initial configuration $x_0\in\mathbb R^N$ enjoys one global-in-time absolutely continuous solution, that is unique forwards in time
\[
\begin{dcases}
\dot{x}\in \mathcal{F}(x), \qquad t\geq 0,\\
x(0)=x_0. 
\end{dcases}
\]
\end{lemma}

\begin{proof}
The existence of global-in-time Filippov's solutions follows from Lemma \ref{L-wp-Fil}. Let us just discuss the uniqueness of solution. We consider two Filippov solutions $x_1=x_1(t)$ and $x_2=x_2(t)$ with the same initial datum $x_0$ and define
\[T:=\inf\{t>0: x_1(t)\neq x_2(t)\}.\]
Our main goal is to prove that $T=+\infty$ by contradiction. We assume that $T<+\infty$. Let us define $x^*:=x_1(T)=x_2(T)$ and take a small enough neighborhood $\mathcal{V}$ of $x^*$ so that $F$ verifies the one-sided Lipschitz condition in it. By continuity there is some $\varepsilon>0$ so that $x_1(t),x_2(t)\in \mathcal{V}$ for every $t\in [T,T+\varepsilon]$. Consequently,
\[
\frac{d}{dt}\frac{1}{2}| x_1-x_2|^2\in (\mathcal{F}(x_1(t))-\mathcal{F}(x_2(t)))\cdot (x_1(t)-x_2(t)).
\]
By the one-sided Lipschitz condition, there exists some constant $M$ depending on $x^*$ such that
\[\frac{d}{dt}| x_1-x_2|^2\leq M| x_1-x_2|^2\]
for every $t\in [T,T+\varepsilon]$. By Gronwall's inequality, one then obtains $x_1(t)=x_2(t)$, for every $t\in [T,T+\varepsilon]$, and this contradicts the assumption on $T<+\infty$.
\end{proof}

Let us now explicitly compute the Filippov set-valued map $\mathcal{H}=\mathcal{H}(\Theta)$ of our particular vector field $H=H(\Theta)$ for the critical case $\alpha=\frac{1}{2}$. Recall Subsection \ref{cluster-notation-subsection} about the collision equivalence relation and the necessary notation to deal with clusters of oscillators.

\begin{proposition}\label{P-explicit-Fil-map-critical}
In the critical regime $\alpha=\frac{1}{2}$, the Filippov set-valued map $\mathcal{H}=\mathcal{H}(\Theta)$ associated with $H=H(\Theta)$ stands for the convex and compact polytope consisting of the points $(\omega_1,\ldots,\omega_N)\in\mathbb{R}^N$ such that
\[
\omega_i=\Omega_i+\frac{K}{N}\sum_{j\notin \mathcal C_i(\Theta)}h(\theta_j-\theta_i)+\frac{K}{N}\sum_{j\in \mathcal C_i(\Theta)\setminus \{i\}}y_{ij},\ \ \mbox{ for all }i=1,\ldots,N,
\]
for some $Y=(y_{ij})_{1\leq i,j\leq N}\in\Skew_N([-1,1])$. 
\end{proposition}
Since the proof is clear by definition of the Filippov set-valued map, we omit it here.

\begin{remark}\label{R-av-frequency}
Notice that for every $(\omega_1,\ldots,\omega_N)\in \mathcal{H}(\Theta)$ the next property holds true
\[
\sum_{i=1}^N\omega_i=\sum_{i=1}^N\Omega_i.
\]
In particular, every Filippov solution $(\theta_1,\ldots,\theta_N)$ to \eqref{E-2}, in the case $\alpha=\frac{1}{2}$, verifies

\[
\sum_{i=1}^N\dot{\theta}_i(t)=\sum_{i=1}^N\Omega_i,\mbox{ for a.e. }t\geq 0.
\]
Hence, the Filippov solutions in the critical case still preserve the average frequency like classical solutions do, for the subcritical case or the original Kuramoto model.
\end{remark}

\begin{example}\label{Ex-explicit-Fil-map-critical}
In order to gain some intuition about those sets, let us exhibit some particular examples:
\begin{enumerate}
\item For every $N\in\mathbb{N}$, if $\Theta\notin\mathcal{C}$, then $\mathcal{H}(\Theta)=\{H(\Theta)\}$.
\item For $N=2$, if $\Theta=(\theta_1,\theta_2)\in \mathcal{C}_{12}$, then $\mathcal{H}(\Theta)$ is the polytope consisting of points $(\omega_1,\omega_2)\in\mathbb{R}^2$ such that
\begin{equation*}
\begin{aligned}
\omega_1&=\Omega_1+\frac{K}{2}y_{12},\\
\omega_2&=\Omega_2-\frac{K}{2}y_{12},
\end{aligned}
\end{equation*}
for some $y_{12}\in [-1,1]$.
\item For $N=3$, if $\Theta=(\theta_1,\theta_2,\theta_3)\in\mathcal{C}_{12}\setminus\mathcal{C}_{13}$, then $\mathcal{H}(\Theta)$ is the polytope consisting of the points $(\omega_1,\omega_2,\omega_3)\in\mathbb{R}^3$ such that
\begin{equation*}
\begin{aligned}
\omega_1&=\Omega_1+\frac{K}{3}h(\theta_3-\theta_1)+\frac{K}{3}y_{12},\\
\omega_2&=\Omega_2+\frac{K}{3}h(\theta_3-\theta_2)-\frac{K}{3}y_{12},\\
\omega_3&=\Omega_3+\frac{K}{3}h(\theta_1-\theta_3)+\frac{K}{3}h(\theta_2-\theta_3),
\end{aligned}
\end{equation*}
for some $y_{12}\in [-1,1]$.
\item For $N=3$, if $\Theta=(\theta_1,\theta_2,\theta_3)\in \mathcal{C}_{12}\cap\mathcal{C}_{13}$, then $\mathcal{H}(\Theta)$ is the polytope consisting of the points $(\omega_1,\omega_2,\omega_3)\in\mathbb{R}^3$ such that
\begin{equation*}
\begin{aligned}
\omega_1&=\Omega_1+\frac{K}{3}y_{12}+\frac{K}{3}y_{13},\\
\omega_2&=\Omega_2-\frac{K}{3}y_{12}+\frac{K}{3}y_{23},\\
\omega_3&=\Omega_3-\frac{K}{3}y_{13}-\frac{K}{3}y_{23}.
\end{aligned}
\end{equation*}
for some $y_{12},y_{13},y_{23}\in [-1,1]$.
\end{enumerate}
The second and third cases yield line segments and the last one is a regular hexagon as it can be depicted in Figure \ref{fig:filippov}.
\end{example}

\begin{figure}
\centering
\begin{subfigure}[b]{0.45\textwidth}
\includegraphics[scale=0.8]{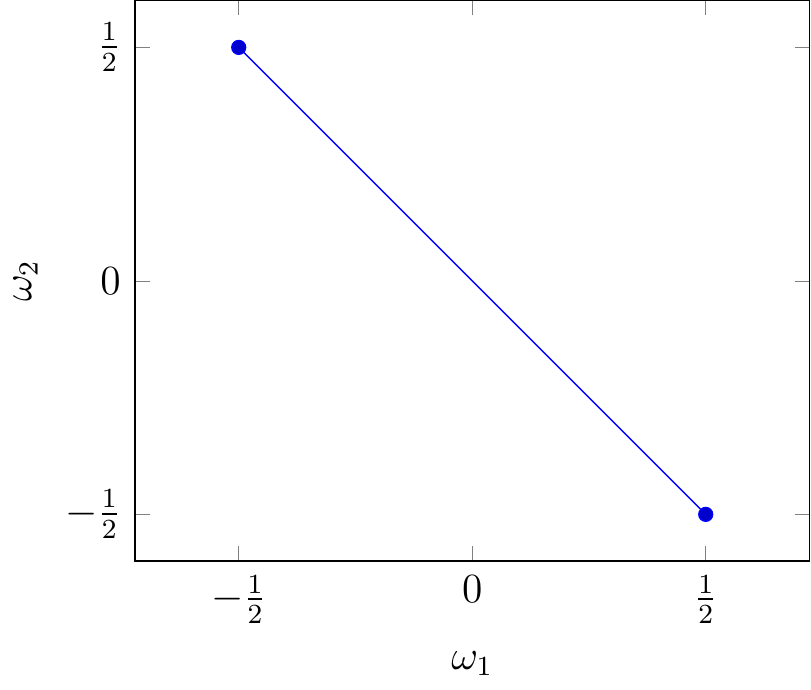}
\caption{$N=2$}
\label{fig:poly-2}
\end{subfigure}
\begin{subfigure}[b]{0.45\textwidth}
\includegraphics[scale=0.9]{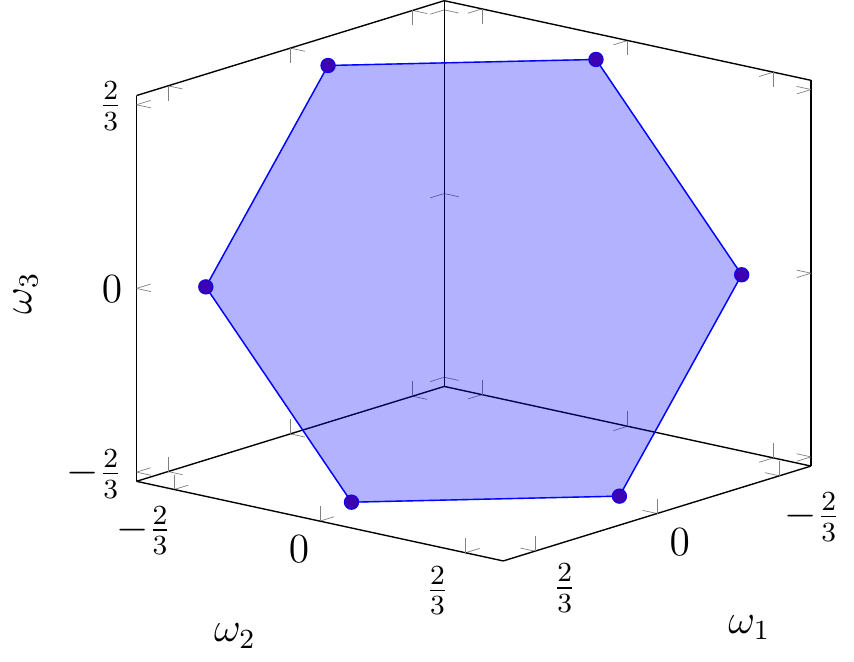}
\caption{$N=3$}
\label{fig:poly-3}
\end{subfigure}
\caption{Pictures of the Filippov set-valued map in the critical case at a total collision phase configuration. In Figure \ref{fig:poly-2}, $N=2$ and the polytope is a line segment joining $\left(\Omega_1\pm \frac{K}{2},\Omega_2\mp \frac{K}{2}\right)$. In Figure \ref{fig:poly-3},  $N=3$ and the polytope is a regular hexagon with vertices $\left(\Omega_1\pm \frac{2K}{3},\Omega_2\mp \frac{2K}{3},\Omega_3\right)$, $\left(\Omega_1\pm \frac{2K}{3},\Omega_2,\Omega_3\mp \frac{2K}{3}\right)$ and $\left(\Omega_1,\Omega_2\pm \frac{2K}{3},\Omega_3\mp \frac{2K}{3}\right)$. For simplicity, the natural frequencies are set to zero and $K=1$ in the figures.}
\label{fig:filippov}
\end{figure}

Finally, let us apply Lemma \ref{L-IVP-Fil} to construct the unique Filippov solutions of our particular system \eqref{E-2} in the critical case $\alpha=\frac{1}{2}$. The way to go is similar to that in the preceding Subsection \ref{well-posedness-subcritical-subsection} and relies on a good decomposition of $-h$. Define the couple of function $f=f(\theta)$ and $g=g(\theta)$ in $(-\pi,\pi)$ as follows
\begin{align*}
f(\theta)&:=\begin{dcases}
1 & \mbox{for}\quad \theta\in (-\pi,0),\\
-1, & \mbox{for}\quad \theta\in [0,\pi),
\end{dcases}
\\
g(\theta)&:=\begin{dcases}
-1-h(\theta), & \mbox{for}\quad \theta\in (-\pi,0),\\
1-h(\theta), & \mbox{for}\quad \theta\in [0,\pi).
\end{dcases}
\end{align*}
Notice that
\begin{equation}\label{E-7}
-h(\theta)=f(\theta)+g(\theta),\ \ \mbox{ for all }\theta\in (-\pi,\pi),
\end{equation}
as depicted in Figure \ref{fig:kuramoto-decomposition-05}.

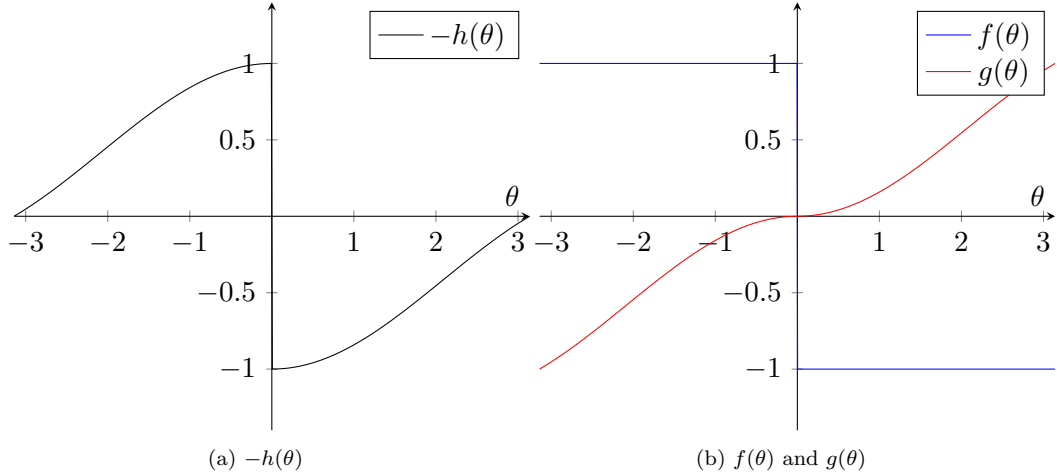
\begin{figure}
\centering
\begin{subfigure}[b]{0.45\textwidth}
\begin{tikzpicture}
\begin{axis}[
  axis x line=middle, axis y line=middle,
  xmin=-pi, xmax=pi, xtick={-3,...,3}, xlabel=$\theta$,
  ymin=-1.4, ymax=1.4, ytick={-1,-0.5,0,0.5,1},
]
\addplot [
    domain=-pi:0, 
    samples=500, 
    color=black,
]
{-sin(deg(x))/pow(abs(x),2*0.5)};
\addplot [
    domain=0:pi, 
    samples=500, 
    color=black,
]
{-sin(deg(x))/pow(abs(x),2*0.5)};
\addlegendentry{$-h(\theta)$}
\end{axis}
\end{tikzpicture}
\caption{$-h(\theta)$}
\label{fig:kuramoto-decomposition-05-1}
\end{subfigure}
\begin{subfigure}[b]{0.45\textwidth}
\begin{tikzpicture}[
  declare function={
  f(\x)= (\x<=0) * (1)  +
            (\x>0) * (-1);
  g(\x)= (\x<=0) * (-1-sin(deg(\x))/pow(abs(\x),2*0.5)))  +
            (\x>0) * (1-sin(deg(\x))/pow(abs(\x),2*0.5)));
  }
]
\begin{axis}[
  axis x line=middle, axis y line=middle,
  xmin=-pi, xmax=pi, xtick={-3,...,3}, xlabel=$\theta$,
  ymin=-1.4, ymax=1.4, ytick={-1,-0.5,0,0.5,1},
]
\addplot[domain=-pi:pi, samples=1000,color=blue]{f(x)};
\addlegendentry{$f(\theta)$}
\addplot[domain=-pi:pi, samples=200,color=red]{g(x)};
\addlegendentry{$g(\theta)$}
\end{axis}
\end{tikzpicture} 
\caption{$f(\theta)$ and $g(\theta)$}
\label{fig:kuramoto-decomposition-05-2}
\end{subfigure}
\caption{Graph of the function $-h(\theta)$ and the functions $f(\theta)$ and $g(\theta)$ in the decomposition for the value $\alpha=0.5$.}\label{fig:kuramoto-decomposition-05}
\end{figure}

\begin{remark}
Note that although $-h(\theta)$ is a continuous function because of the jump discontinuities at $\theta\in 2\pi\mathbb{Z}$, one can locally decompose it around such values in terms of a decreasing function $f(\theta)$ and a Lipschitz-continuous function $g(\theta)$.
\end{remark}

Finally, for every $\Theta^*=(\theta_1^*,\ldots,\theta_N^*)\in\mathbb R^N$ we locally decompose $H$ around it as follows 
\begin{align}
F_i(\Theta)&:=\frac{K}{N}\sum_{j\in \mathcal C_i(\Theta^*)}f\left(\overline{\theta_i-\theta_j}\right),\label{E-8}\\
G_i(\Theta)&:=\Omega_i+\frac{K}{N}\sum_{j\in \mathcal C_i(\Theta^*)}g\left(\overline{\theta_i-\theta_j}\right)-\frac{K}{N}\sum_{j\notin \mathcal C_i(\Theta^*)}h(\theta_i-\theta_j),\label{E-9}
\end{align}
where the above functions are defined almost everywhere (note that $f$ does not make sense at $0$, thus $F_i$ just makes sense a.e.). Again,  we recall that $\bar\theta$ is its representative modulo $2\pi$ in the interval $(-\pi,\pi]$, for any $\theta\in\mathbb R$.

\begin{proposition}\label{P-decomp-2}
Let $\Theta^*=(\theta_1^*,\ldots,\theta_N^*)\in\mathbb R^N$ and  define the vector fields
\[F:\mathcal V\longrightarrow\mathbb R^N,\ \ G:\mathcal V\longrightarrow\mathbb R^N,\]
via the formulas (\ref{E-8})-(\ref{E-9}), for a small enough neighborhood $\mathcal V$ of $\Theta^*$ in $\mathbb R^N$. Then,
\begin{enumerate}
\item $H=F+G$ in $\mathcal V$.
\item $F$ is decreasing in $\mathcal V$.
\item $G$ is Lipschitz-continuous in $\mathcal V$. 
\item $H$ is one-sided Lipschitz continuous in $\mathcal V$.
\end{enumerate}
\end{proposition}
\begin{proof}
The proof is analogous to Proposition \ref{P-decomp-1}.
\end{proof}
Finally, putting Lemma \ref{L-wp-Fil}-\ref{L-IVP-Fil} and Proposition \ref{P-decomp-2} together, one concludes the following well--posedness result.
\begin{theorem}\label{T-wp-2}
There is one global-in-time Filippov solution to the system \eqref{E-2} with $\alpha=\frac{1}{2}$ for any initial configuration, that is unique forwards in time.
\end{theorem}

Again, we can characterize the eventual emergence of sticking of a cluster after a potential collision in a similar way as we did in Theorem \ref{T-sticking-1}. We introduce the following notation.

For any $N\in\mathbb{N}$, each $1\leq m\leq N$ and every permutation $\sigma$ of $\{1,\ldots,N\}$ we define the following couple of $m\times m$ matrices:
\begin{equation}\label{D-matrices}
M_m^\sigma(\Omega):=(\Omega_{\sigma_i}-\Omega_{\sigma_j})_{1\leq i,j\leq m},\qquad \mathbf{J}_m=(1)_{1\leq i,j\leq m},
\end{equation}
i.e., $M_m^\sigma(\Omega)$ stands for the matrix of relative natural frequencies of the only $m$ oscillators with indices $i=\sigma_1,\ldots,\sigma_m$ and $\mathbf{J}_m$ is a $m\times m$ matrix of which the components are one.

\begin{theorem}\label{T-sticking-2}
Consider $\Theta=(\theta_1,\ldots,\theta_N)$ the global-in-time Filippov solution in Theorem \ref{T-wp-2}. Assume that $t^*$ is some collision time and fix any cluster $E_k(t^*)\equiv E_k$ with $k=1,\ldots,\kappa(t^*)$. Then, the following two statements are equivalent:
\begin{enumerate}
\item The $n_k(t^*) = \# E_k(t^*) $ oscillators in such cluster stick all together at time $t^*$.
\item There exists a bijection $\sigma:\{1\ldots,n_k\}\rightarrow E_k$ and $Y\in\Skew_{n_k}([-1,1])$ such that
\begin{equation}\label{E-sticking-critical-implicit-wp}
M_{n_k}^\sigma(\Omega)=\frac{K}{N}\left(Y\cdot \mathbf{J}_{n_k}+\mathbf{J}_{n_k}\cdot Y\right).
\end{equation}
\end{enumerate}
\end{theorem}

\begin{proof}
Let us call $n:=n_k$ for simplicity and assume that the oscillators in such cluster agree precisely with the first $n$ oscillators, i.e., $E_k=\{1,\ldots,n\}$. By continuity, let us take some small $\varepsilon>0$ such that $\bar\theta_j(t)\neq \bar\theta_i(t)$, for every $t\in [t^*,t^*+\varepsilon]$, any $i\in E_k$ and each $j\notin E_k$. First, let us assume that the former statement holds true. Without loss of generality we might assume that $\theta_1(t)=\cdots=\theta_n(t)$ for all $t\geq t^*$ and we define  $\theta(t):=\theta_1(t)=\cdots=\theta_n(t)$ for all $t\geq t^*$. Then, looking at the explicit expression in Proposition \ref{P-explicit-Fil-map-critical} of the Filippov set-valued map $\mathcal{H}$ the following identities are fulfilled
\begin{align*}
\dot{\theta}_i=\Omega_i+\frac{K}{N}\sum_{j=n+1}^Nh(\theta_j(t)-\theta(t))+\frac{K}{N}\sum_{j=1}^ny_{ij}(t),
\end{align*}
for a.e. $t\in [t^*,t^*+\varepsilon]$ and every $i=1,\ldots,n$, where $y_{ij}\in L^\infty(t^*,t^*+\varepsilon)$ and $Y(t)=(y_{ij}(t))_{1\leq i,j\leq n}\in \Skew_n([-1,1])$ for almost all $t\in [t^*,t^*+\varepsilon]$. Since $\dot{\theta}_{i}=\dot{\theta}_{j}$ a.e., for every $i,j=1,\ldots,n$, then we  obtain the next system of equations
\begin{align*}
\Omega_i-\Omega_j=- \frac{K}{N}\sum_{\substack{l=1\\ l\neq i}}^ny_{il}(t) + \frac{K}{N}\sum_{\substack{l=1\\ l\neq j}}^ny_{jl}(t),
\end{align*}
for a.e. $t\in [t^*,t^*+\varepsilon]$. In particular, \eqref{E-sticking-critical-implicit-wp} holds. Conversely, let us assume that \eqref{E-sticking-critical-implicit-wp} is verified for some $Y\in \Skew_n([-1,1])$, then we have
\[
\Omega_i+\frac{K}{N}\sum_{\substack{l=1\\ l\neq i}}^ny_{il}=\Omega_j+\frac{K}{N}\sum_{\substack{l=1\\ l\neq j}}^ny_{jl}=:\widehat{\Omega}.
\]
Let us now consider the vector field 
\[
\widehat{H}^n=(\widehat{H}_0^n,\widehat{H}_{n+1}^n,\ldots, \widehat{H}_N^n):\mathbb{R}^{N-n+1}\longrightarrow \mathbb{R}^{N-n+1}
\]
given by the formulas
\begin{align*}
\widehat{H}^n_0(\vartheta,\vartheta_{n+1},\ldots,\vartheta_N)&=\widehat{\Omega}+\frac{K}{N}\sum_{j=n+1}^Nh(\vartheta_j-\vartheta),\\
\widehat{H}^n_i(\vartheta,\vartheta_{n+1},\ldots,\vartheta_N)&=\Omega_i+\frac{nK}{N}h(\vartheta-\vartheta_i)+\frac{K}{N}\sum_{j=n+1}^N h(\vartheta_j-\vartheta_i),
\end{align*}
for every $i=n+1,\ldots,N$. Also, consider its associated Filippov set-valued map $\widehat{\mathcal{H}}^n$ and the associated differential inclusion
\[
(\dot{\vartheta},\dot{\vartheta}_{n+1},\ldots,\dot{\vartheta}_N)\in \widehat{\mathcal{H}}^n(\vartheta,\vartheta_{n+1},\ldots,\vartheta_N),
\]
with initial datum given by
\[
(\vartheta(t^*),\vartheta_{n+1}(t^*),\ldots,\vartheta_N(t^*))=(\theta^*,\theta_{n+1}(t^*),\ldots,\theta_N(t^*)).
\]
A similar well-posedness result to that in Theorem \ref{T-wp-2} shows that such IVP enjoys one global-in-time solution. In addition, by definition it is apparent that whenever we pick $(\omega,\omega_{n+1},\ldots,\omega_N)\in\widehat{\mathcal{H}}^n(\vartheta,\vartheta_{n+1},\ldots,\vartheta_N)$, then we obtain
\[
\big(\underbrace{\omega,\ldots,\omega}_{n \, \text{pairs}},\omega_{n+1},\ldots,\omega_N\big)\in \mathcal{H}\big(\underbrace{\vartheta,\ldots,\vartheta}_{n \, \text{pairs}},\vartheta_{n+1},\vartheta_{N}\big).
\]
Consequently, the following two trajectories in $\mathbb{R}^N$
\begin{align*}
t&\mapsto (\theta_1(t),\theta_2(t),\ldots,\theta_n(t),\theta_{n+1}(t),\ldots,\theta_N(t)),\\
t&\mapsto (\underbrace{\vartheta(t),\ \vartheta(t),\ \ldots,\, \vartheta(t)}_{n \, \text{pairs}},\vartheta_{n+1}(t),\ldots,\vartheta_N(t)),
\end{align*}
are Filippov solutions to \eqref{E-2} such that they take the same value at $t=t^*$, namely,
\[
\big(\underbrace{\theta^*,\ldots,\theta^*}_{n \, \text{pairs}},\theta_{n+1}(t^*),\ldots,\theta_N(t^*)\big).
\]
By uniqueness they agree and, in particular, 
\[
\theta_i(t)=\vartheta(t)\mbox{ for all }t\geq t^*\mbox{ and every }i=1,\ldots,n.
\]
\end{proof}

The sticking condition \eqref{E-sticking-critical-implicit-wp} can be characterized in a much more explicit manner by convex analysis techniques supported by \textit{Farkas' alternative}. See Appendix \ref{Appendix-sticking} and, in particular, the characterization of condition \eqref{E-sticking-critical-implicit-wp} in Lemma \ref{L-farkas-application-appendix-3}. Such ideas can be arranged in the next result.

\begin{corollary}\label{C-sticking-2}
Under the same assumptions as in Theorem \ref{T-sticking-2}. The following two assertions are equivalent:
\begin{enumerate}
\item The $n_k$ oscillators in the cluster $E_k$ stick all together at time $t^*$.
\item We have
\begin{equation}\label{E-sticking-critical-explicit-wp}
\frac{1}{m}\sum_{i\in I}\Omega_i-\frac{1}{n_k}\sum_{i\in E_k}\Omega_i\in \left[-\frac{K}{N}(n_k-m),\frac{K}{N}(n_k-m)\right],
\end{equation}
for every $1\leq m\leq n_k$ and every $I\subseteq E_k$ such that $\# I=m$.
\end{enumerate}
\end{corollary}

\begin{remark}\label{R-Zeno}
Notice that in Theorem \ref{T-sticking-2} and Corollary \ref{C-sticking-2} we have characterized when the whole cluster $E_k$ remains stuck together, but not when a subcluster of a given size instantaneously splits from the remaining oscillators of the cluster. The main problem to extend the above proof is that it is hard to quantify the way in which an oscillator splits from the subcluster. Specifically, it is possible that an oscillator departs from the cluster exhibiting a left accumulation of switches of state where it instantaneously splits and collides with the formed subcluster. Although this accumulating phenomenon will cause some problems throughout the paper, we will show how can we overcome them. 

Let us mention that such phenomenon is called \textit{left Zeno behavior} in the literature. It appears in the Filippov solutions of some systems like the reversed bouncing ball. For instance, in \cite[p. 116]{F}  Filippov proposed a discontinuous first order system with solutions exhibiting Zeno behavior. In \cite[Theorem 2.10.4]{F}, the same author considered absence of Zeno behavior as part of the sufficient conditions (but not necessary) guaranteeing forwards uniqueness. We skip the analysis of Zeno behavior here and will address it in a future work.
\end{remark}

\subsection{Well--posedness in the supercritical regime}\label{well-posedness-supercritical-subsection}
Recall that in the supercritical regime, i.e., $\alpha>\frac{1}{2}$, the vector field $H=H(\Theta)$ is not only discontinuous at the collision states but it is also unbounded near those points, see Figure \ref{fig:alpha}. Thus, the classical theory for well-posedness cannot be applied either and one might seek for a notion of generalized solutions in the same sense as in the critical case $\alpha=\frac{1}{2}$ (see Subsection \ref{well-posedness-critical-subsection}). Hence, one strategy could be to turn again the differential equation of interest into an augmented differential inclusion given by the associated Filippov set-valued map. A similar analysis to that in Proposition \ref{P-explicit-Fil-map-critical} yields the following characterization of the Filippov set-valued map for the supercritical regime.

\begin{proposition}\label{P-explicit-Fil-map-supercritical}
In the supercritical regime $\alpha>\frac{1}{2}$, the Filippov set-valued map $\mathcal{H}=\mathcal{H}(\Theta)$ associated with $H=H(\Theta)$ stands for the convex and unbounded polytope consisting of the points $(\omega_1,\ldots,\omega_N)\in\mathbb{R}^N$ such that 
\[
\omega_i=\Omega_i+\frac{K}{N}\sum_{j\notin \mathcal C_i(\Theta)}h(\theta_j-\theta_i)+\frac{K}{N}\sum_{j\in \mathcal C_i(\Theta)\setminus\{i\}}y_{ij},\ \mbox{ for all }i=1,\ldots,N,
\]
for some $Y=(y_{ij})_{1\leq i,j\leq N}\in\Skew_N(\mathbb{R})$.
\end{proposition}

The Filippov set-valued map enjoys similar expressions in the critical and supercritical regimes except for a ``slight'' change. In the former case, the coefficients $y_{ij}$ range in the interval $[-1,1]$ whereas in the latter case they take values in the whole $\mathbb{R}$. Indeed, the same examples for $\alpha=\frac{1}{2}$ in Example \ref{Ex-explicit-Fil-map-critical} can be considered for $\alpha>\frac{1}{2}$. For instance, similar polytopes to those in Figure \ref{fig:filippov} are obtained at the total collision phase configurations when the corresponding polygon is replaced by its affine envelope. Those similarities ensure that any Filippov solution to \eqref{E-2} with $\alpha>\frac{1}{2}$ also conserve the average frequency as in Remark \ref{R-av-frequency}. What is more, since $\mathcal{H}(\Theta)$ is apparently non-empty, then Lemma \ref{L-Fil-map-prop} shows that $\mathcal{H}$ takes values in the non-empty, closed and convex sets and it has closed graph in the set-valued sense. However, the unboundedness in $y_{ij}$ entails a severe change of behavior. Specifically, it breaks the local compactness of $m(\mathcal{H})$ and, as a consequence, the existence result in Lemma \ref{L-wp-diff-incl} fails to work. Such loss of compactness is fateful and implies that the supercritical regime $\alpha>\frac{1}{2}$ lies in the setting where all the ``classical'' assumptions ensuring global existence and one-sided uniqueness does not hold. The literature about the abstract analysis of unbounded differential inclusions is rare, see \cite{I,T}. In addition, all those results require some sort of relaxed set-valued Lipschitz condition and linear growth that do not hold in our particular problem. Nevertheless, we will show that in some cases we can still construct a Filippov solution which is unique under some conditions. 

\begin{remark}\label{R-necessary-condition-sticking-supercritical}
Notice that, despite the lack of uniqueness results in the supercritical case, the approach in Theorem \ref{T-sticking-2} may still be used to obtain a partial answer. Namely, it might give a sufficient condition on the natural frequencies to ensure that after a collision of a classical solution, we can continue a Filippov solution with sticking of the formed cluster. Since we will elaborate on this idea later, we will skip it here and will just focus on the study of a necessary condition of sticking like in \eqref{E-sticking-critical-implicit-wp}. Indeed, consider some Filippov solution $\Theta=(\theta_1,\ldots,\theta_N)$ to \eqref{E-2} with $\alpha>\frac{1}{2}$ and assume that it is defined in an interval $[0,T)$ and that $t^*\in (0,T)$ is some collision time. Then, we might fix a cluster $E_k(t^*)\equiv E_k$ and assume that the $n_k(t^*)\equiv n_k$ oscillators in such cluster stick all together at time $t^*$. Hence, a similar proof to that of Theorem \ref{T-sticking-2} would entail the existence of some bijection $\sigma:\{1,\ldots,n_k\}\longrightarrow E_k$ and some $Y\in\Skew_{n_k}(\mathbb{R})$ such that
\begin{equation}\label{E-sticking-supercritical-implicit-wp}
M_m^\sigma(\Omega)=\frac{K}{N}\left(Y\cdot \mathbf{J}_{n_k}+\mathbf{J}_{n_k}\cdot Y\right).
\end{equation}
One might want to obtain again a more explicit characterization of such condition. We can resort on similar ideas coming from Farkas' alternative, see Lemma \ref{L-farkas-application-appendix-2} in Appendix \ref{Appendix-sticking}. Such Lemma ensures that \eqref{E-sticking-supercritical-implicit-wp} is perfectly equivalent to the condition \eqref{E-sticking-supercritical-explicit}
\[
m_{ij}+m_{jk}+m_{ki}=0,
\]
for every $i,j,k=1,\ldots,n_k$, where $m_{ij}$ denotes the $(i,j)$ component of the matrix $M_m^\sigma (\Omega)$. Let us look into the particular structure of $M_{n_k}^\sigma(\Omega)$ to restate the above condition (see \eqref{D-matrices})
\[
m_{ij}+m_{jk}+m_{ki}=(\Omega_{\sigma_i}-\Omega_{\sigma_j})+(\Omega_{\sigma_j}-\Omega_{\sigma_k})+(\Omega_{\sigma_k}-\Omega_{\sigma_i}).
\]
Then, the necessary sticking condition is automatically satisfied for every given configuration of natural frequencies. This suggests that, independently on the chosen natural frequencies, any classical solution in the supercritical case that stops at a collision state might always be continued as Filippov solution with sticking of the cluster. For this, we will need some accurate control of the behavior of such classical solutions at the maximal time of existence.
\end{remark}

\begin{lemma}\label{L-wp-3}
Consider $\Theta=(\theta_1,\ldots,\theta_N)$ any classical solution to \eqref{E-2} with $\alpha\in (\frac{1}{2},1)$ that is defined in a finite maximal existence interval $[0,t^*)$. Then,
\begin{enumerate}
\item The solution does not blow up at $t^*$, i.e.,
\[
\lim_{t\rightarrow t^*}\vert \Theta(t)\vert\neq \infty,
\]
\item The solution converges towards a collision state, i.e., there exists $\Theta^*\in\mathcal{C}$ such that
\[
\lim_{t\rightarrow t^*}\Theta(t)=\Theta^*.
\]
\end{enumerate}
In addition, the trajectory $t\mapsto \Theta(t)$ remains absolutely continuous up to the collision time $t=t^*$; specifically, $\dot{\Theta}\in L^2((0,t^*),\mathbb{R}^N)$.
\end{lemma}

\begin{proof}
We split the proof into three parts. The first part is devoted to show that the classical trajectories verify the following fundamental inequalities:
\begin{align}
\frac{1}{2}\int_0^t\vert \dot{\Theta}(s)\vert^2\,ds&\leq V_{int}(\Theta_0)+\frac{C_\Omega^2}{2}t,\label{E-fund-1}\\
\vert \Theta(t)\vert&\leq \vert\Theta_0\vert +\int_0^t\vert \dot{\Theta}(s)\vert\,ds,\label{E-fund-2}
\end{align}
for every $t\in [0,t^*)$. Here, $V_{int}(\Theta)$ is the second term of the potential $V(\Theta)$ in \eqref{E-potential} and we set the constant 
\begin{equation*}\label{C-Omega}
C_\Omega:=\left(\sum_{i=1}^N\Omega_i^2\right)^{1/2}.
\end{equation*} 
We will show in the second step that such inequalities \eqref{E-fund-1} and \eqref{E-fund-2} infer the next ones
\begin{align}
\frac{1}{2}\int_0^{t^*}\vert \dot{\Theta}(s)\vert^2\,ds&\leq V_{int}(\Theta_0)+\frac{C_\Omega^2}{2}t^*<\infty,\label{E-fund-3}\\
\int_0^{t^*}\vert \dot{\Theta}(s)\vert\,ds&\leq V_{int}(\Theta_0)+\frac{1+C_\Omega^2}{2}t^*<\infty,\label{E-fund-4}\\
\vert \Theta(t)\vert& \leq \vert \Theta_0\vert+V_{int}(\Theta_0)+\frac{1+C_\Omega^2}{2}t^*,\label{E-fund-5}
\end{align}
for every $t\in [0,t^*)$. Finally, the third part will focus on proving the assertions in the statement of the Lemma via such fundamental inequalities \eqref{E-fund-1}-\eqref{E-fund-5}. 

\medskip
\noindent
$\bullet$ {\it Step 1:} Recall that in Section \ref{preliminaries-section}, the classical solution $t\longmapsto\Theta(t)$ of \eqref{E-2} equivalently solves a gradient flow system \eqref{E-gradient-flow}, i.e.,
\[
\dot{\Theta}(t)=-\nabla V(\Theta(t)),
\]
for all $t\in [0,t^*)$, where $V$ is given in \eqref{E-potential}. Hence,
\[
\frac{d}{dt}V(\Theta(t))=\nabla V(\Theta(t))\cdot\dot{\Theta}(t)=-\vert \dot{\Theta}(t)\vert^2,
\]
for every $t\in [0,t^*)$. Taking integrals in time, we obtain
\begin{equation}\label{E-11}
\int_0^t\vert \dot{\Theta}(s)\vert^2\,ds=V(\Theta_0)-V(\Theta(t))=\sum_{i=1}^N\Omega_i(\theta_{i,0}-\theta_i(t))+V_{int}(\Theta_0)-V_{int}(\Theta(t)),
\end{equation}
for every $t\in [0,t^*)$. Recall that the function $W$ in \eqref{E-W} involved in the potential  \eqref{E-potential} is a primitive function of $h$. Then, $W\geq 0$ as a consequence of the antisymmetry of $h$ and our choice $W(0)=0$ and, in particular, $V_{int}\geq 0$. This, together with the Cauchy--Schwarz inequality, yield
\begin{equation}\label{E-12}
\int_0^t\vert \dot{\Theta}(s)\vert^2\,ds\leq C_\Omega\int_0^t\vert \dot{\Theta}(s)\vert\,ds+V_{int}(\Theta_0),
\end{equation}
for every $t\in [0,t^*)$. Using Young's inequality in the first term of \eqref{E-12}, we arrive at the first fundamental inequality \eqref{E-fund-1}. The second inequality \eqref{E-fund-2} is standard, but let us sketch it for the sake of clarity
\[
\frac{d}{dt}\frac{\vert \Theta\vert^2}{2}=\Theta\cdot \dot{\Theta}\leq \vert \Theta\vert\,\vert \dot{\Theta}\vert,
\]
for all $t\in [0,t^*)$. Then,  we arrive at 
\[
\frac{d}{dt}\vert \Theta(t)\vert\leq \vert \dot{\Theta}(t)\vert,
\]
for every $t\in [0,t^*)$ and integrating with respect to time yields \eqref{E-fund-2}.

\medskip
\noindent
$\bullet$ {\it Step 2:} 
First, taking limits $t\rightarrow t^*$ in \eqref{E-fund-1}, we clearly obtain \eqref{E-fund-3}. Also, the finite length of the trajectory \eqref{E-fund-4} holds true by virtue of the Cauchy--Schwarz inequality and Young's inequality both applied to the preceding one. Finally, inequalities \eqref{E-fund-2} and \eqref{E-fund-4} entail \eqref{E-fund-5}.

\medskip
\noindent
$\bullet$ {\it Step 3:} The classical trajectory $t\mapsto\Theta(t)$ is defined up to a finite maximal time $t^*$. Hence, classical results show that either it blows up at $t=t^*$ or there exists some sequence $\{t_n\}_{n\in\mathbb{N}}\nearrow t^*$ and some $\Theta^*\in\mathcal{C}$ such that $\{\Theta(t_n)\}_{n\in\mathbb{N}}\rightarrow \Theta^*.$ Since the former option is prevented by \eqref{E-fund-5}, then the latter must hold true. Let us prove that the whole trajectory converges towards that collision state $\Theta^*$. In other case, there exists another sequence $\{s_n\}_{n\in\mathbb{N}}\nearrow t^*$ and some $\varepsilon_0>0$ such that
\begin{equation}\label{E-13}
\vert \Theta(s_n)-\Theta^*\vert\geq \varepsilon_0,
\end{equation}
for all $n\in\mathbb{N}$. Without loss of generality we can assume that the sequences $\{t_n\}_{n\in\mathbb{N}}$ and $\{s_n\}_{n\in\mathbb{N}}$ are ordered as follows
\[
t_1<s_1<t_2<s_2<\ldots
\]
and that
\begin{equation}\label{E-14}
\vert \Theta(t_n)-\Theta^*\vert\leq \frac{\varepsilon_0}{2^{n}},
\end{equation}
for every $n\in\mathbb{N}$. Thereby,
\begin{align*}
\vert \Theta(t_n)-\Theta(s_n)\vert&\geq \vert \Theta(s_n)-\Theta^*\vert-\vert \Theta(t_n)-\Theta^*\vert\geq \varepsilon_0-\frac{\varepsilon_0}{2^{n}},\\
\vert \Theta(s_n)-\Theta(t_{n+1})\vert&\geq \vert \Theta(s_n)-\Theta^*\vert-\vert \Theta(t_{n+1})-\Theta^*\vert\geq \varepsilon_0-\frac{\varepsilon_0}{2^{n+1}},
\end{align*}
for all $n\in\mathbb{N}$. Then, it is clear that
\begin{align*}
\int_0^{t^*}\vert \dot{\Theta}(t)\vert\,dt&\geq \int_{t_1}^{t^*} \vert \dot{\Theta}(t)\vert\,dt\\
&=\sum_{n=1}^\infty\int_{t_n}^{s_n}\vert \dot{\Theta}(t)\vert\,dt+\sum_{n=1}^\infty\int_{s_n}^{t_{n+1}}\vert \dot{\Theta}(t)\vert\,dt\\
&\geq \sum_{n=1}^\infty\vert \Theta(t_n)-\Theta(s_n)\vert+\sum_{n=1}^\infty\vert \Theta(s_n)-\Theta(t_{n+1})\vert\\
&\geq \sum_{n=1}^\infty \varepsilon_0\left(1-\frac{1}{2^{n}}\right)+\sum_{n=1}^\infty\varepsilon_0\left(1-\frac{1}{2^{n+1}}\right)=\infty.
\end{align*}
Thus, the trajectory would have infinite length and that contradicts \eqref{E-fund-4}. Hence, we find
\[
\lim_{t\rightarrow t^*}\Theta(t)=\Theta^*.
\]
\end{proof}

Such conclusion shows that, as expected, it is plausible to continue classical solutions by Filippov solutions (hence absolutely continuous) after a possible collision. The explicit method of continuation is exhibited in the following result.

\begin{theorem}\label{T-wp-3}
Consider $\Theta=(\theta_1,\ldots,\theta_N)$ any classical solution to \eqref{E-2} with $\alpha\in \left(\frac{1}{2},1\right)$ that is defined in a finite maximal existence interval $[0,t^*)$ and, according to Lemma \ref{L-wp-3}, let us consider the collision state $\Theta^*\in\mathcal{C}$ such that
\[
\lim_{t\rightarrow t^*}\Theta(t)=\Theta^*.
\]
Then, there exists some $\varepsilon>0$ so that the classical trajectory $t\mapsto \Theta(t)$ can be continued by a Filippov solution to \eqref{E-2} in a short interval $[t^*,t^*+\varepsilon)$ in such a way that oscillators belonging to the same cluster of the collision state $\Theta^*$ remain all stuck together.
\end{theorem}

\begin{proof}
Let $E_k$ be the $k$-th cluster of oscillators with $n_k = \# E_k$ for $k=1, \cdots, \kappa$. We consider a bijection $\sigma^k:\{1,\ldots,n_k\}\longrightarrow E_k$, for every $k=1,\ldots,\kappa$. Since the necessary condition \eqref{E-sticking-supercritical-implicit-wp} is automatically satisfied as discussed in Remark \ref{R-necessary-condition-sticking-supercritical}, then there exists some matrix $Y^k\in \Skew_{n_k}(\mathbb{R})$ such that
\begin{equation}\label{E-15}
\Omega_{\sigma^k_i}+\frac{K}{N}\sum_{\substack{l=1\\ l\neq i}}^{n_k}y_{il}^k=\Omega_{\sigma^k_j}+\frac{K}{N}\sum_{\substack{l=1\\ l\neq j}}^{n_k}y_{jl}^k=:\widehat{\Omega}_k,
\end{equation}
for every couple of indices $i,j\in\{1,\ldots,n_k\}$. Let us define the following system of $\kappa$ differential equations
\begin{equation}\label{E-reduced-system}
\dot{\vartheta}_k=\widehat{H}_k(\vartheta_1,\ldots,\vartheta_k):=\widehat{\Omega}_k+\frac{K}{N}\sum_{\substack{m=1\\ m\neq k}}^\kappa n_m h(\vartheta_m-\vartheta_k),
\end{equation}
for $k=1,\ldots,\kappa$, with initial data given by
\begin{equation}\label{E-reduced-data}
(\vartheta_1(t^*),\ldots,\vartheta_\kappa(t^*))=(\theta_{\iota_1}^*,\ldots,\theta_{\iota_\kappa}^*).
\end{equation}
Since the initial datum is a non-collision state in a lower dimension space $\mathbb{R}^\kappa$ of phase configurations, then there exists a unique classical solution to such problem that is defined in a maximal existence interval $[t^*,t^{**})$ and such that if $t^{**}<\infty$, then $(\vartheta_1,\ldots,\vartheta_\kappa)$ converges towards a new collision state by virtue of Lemma \ref{L-wp-3} (merge of clusters). The same result ensures that 
\begin{align*}
t\in[0,t^*)&\longmapsto(\theta_1(t),\ldots,\theta_N(t)),\\
t\in[t^*,t^{**})&\longmapsto(\vartheta_1(t),\ldots,\vartheta_\kappa(t)),
\end{align*}
belong to $W^{1,2}((0,t^*),\mathbb{R}^N)$ and $W^{1,2}((t^*,t^{**}),\mathbb{R}^\kappa)$, respectively. Let us set the prolongation of $t\longmapsto\Theta(t)$ in $[t^*,t^{**})$ in such a way that
\[
\theta_{\sigma^k_i}(t):=\vartheta_k(t),\ \forall t\in [t^*,t^{**}),
\]
for every $i\in E_k$ and $k=1,\ldots,\kappa$. Both trajectories glue in a $W^{1,2}$ way and it is clear, by virtue of the definition of $\widehat{H}^k$ in \eqref{E-reduced-system} and $\widehat{\Omega}_k$ in \eqref{E-15} along with the explicit expression of the Filippov map in Proposition \ref{P-explicit-Fil-map-supercritical}, that $t\in [0,t^{**})\longmapsto\Theta(t)$ becomes a Filippov solution to \eqref{E-2} in $[0,t^{**})$.
\end{proof}

\begin{remark}\label{R-wp-3}
It is clear that the above procedure can be repeated as many times as needed after each collision time of the classical solutions to the reduced systems \eqref{E-reduced-system}-\eqref{E-reduced-data}. Indeed, by Remark \ref{R-necessary-condition-sticking-supercritical} the necessary condition \eqref{E-sticking-supercritical-implicit-wp} is automatically satisfied.
Since there can only be $N-1$ collision of oscillators with sticking, we may apply Theorem \ref{T-wp-3} finitely many times to obtain global-in-time Filippov solutions to \eqref{E-2} in the supercritical case. However, one may wonder whether this global-in-time continuation procedure is unique or oscillators may also be allowed to split instantaneously after a collision. Although answering the general question for any number $N$ of oscillators and any collision state is really convoluted, let us give some particular answer for the case $N=2$:
\begin{align}
\dot{\theta}_1=\Omega_1+\frac{K}{2}h(\theta_2-\theta_1),\label{E-19}\\
\dot{\theta}_2=\Omega_2+\frac{K}{2}h(\theta_1-\theta_2).\label{E-20}
\end{align}
Consider the relative phase $\theta:=\theta_2-\theta_1$ and relative natural frequency $\Omega:=\Omega_2-\Omega_1$. Then, the associated dynamics of a classical solution is governed by the next equation
\[
\dot{\theta}=\Omega-Kh(\theta),
\]
in the maximal interval of existence $[0,t^*)$. According to Lemma \ref{L-wp-3}, we infer that $t^*=+\infty$ if $\theta(0)=\bar\theta$, whereas $t^*<+\infty$ if $\theta(0)\notin\{0,\bar\theta\}$. Here, $\bar\theta$ stands for the unique (unstable) equilibrium of the system, see Proposition \ref{P-2-stability} in the subsequent Section \ref{synchro-singular-section}. Without loss of generality, we will fix the initial relative phase so that $\theta(0)\in(0,\bar{\theta})$ (the other cases are similar). Then, we arrive at a collision of oscillators at $t=t^*$ i.e., $\lim_{t\rightarrow t^*}\theta(t)=0$.
\begin{enumerate}
\item Let us assume by contradiction that there was another Filippov solution in $[t^*,t^{**})$ consisting of two particles that instantaneously split again after $t=t^*$. Such split can arise in only two different manners:
\begin{enumerate}
\item (Sharp split) There exists some small $\varepsilon>0$ such that $\theta(t)\neq 0$, for every $t\in (t^*,t^*+\varepsilon)$. In such case, either $\theta(t)>0$, for all $t\in (t^*,t^*+\varepsilon)$, or $\theta(t)<0$, for all $t\in (t^*,t^*+\varepsilon)$.
\item (Zeno split) There exist a couple of sequences $\{t_n\}_{n\in\mathbb{N}}\searrow t^*$ and $\{s_n\}_{n\in\mathbb{N}}\searrow t^*$ such that $\theta(s_n)=0$ but $\theta(t_n)\neq 0$, for every $n\in\mathbb{N}$ (recall the left accumulations of switches or Zeno behavior in Remark \ref{R-Zeno}).
\end{enumerate}
Replacing $t^*$ by a suitable time, it is apparent that the second type of split at $t^*$ guarantees the first one at a (possibly) latter time. Let us then focus just on the fist case. Looking at the profile of $\Omega-kh(\theta)$ in Figure \ref{fig:profile-2-particles}, we then would arrive at the following conclusion: either $\dot{\theta}(t)<0$ and $\theta(t)>0$ for all $t\in(t^*,t^*+\varepsilon)$ or $\dot{\theta}(t)>0$ and $\theta(t)<0$ for all $t\in(t^*,t^*+\varepsilon)$. In any case, we obtain a contradiction.
\item Hence, the only choice for the oscillators after the collision state is to stick together. Let us define the phase of the reduced system, see \eqref{E-15}
\[
\widehat{\Omega}:=\Omega_1+y_{12}=\Omega_2+y_{21},
\]
where $Y\in\Skew_2(\mathbb{R})$ is any matrix verifying the necessary condition \eqref{E-sticking-supercritical-implicit-wp}. Indeed, there just exists one such matrix $Y$, whose items read $y_{12}=-y_{21}=\frac{\Omega_2-\Omega_1}{2}$. Then, $\widehat{\Omega}=\frac{\Omega_1+\Omega_2}{2}$ and the reduced system \eqref{E-reduced-system} looks like
\[
\dot{\vartheta}=\widehat{\Omega},\ t\in [t^*,\infty).
\]
Consequently, the only Filippov solution to \eqref{E-2} evolves through \eqref{E-19}-\eqref{E-20} up to the collision time $t^*$. After it, both oscillators stick together and they move with constant frequency equals to the average natural frequency.
\end{enumerate}

For general $N$, notice that it is not clear whether (b) in the above first item can be reduced to (a). Namely, we cannot guarantee that along a whole time interval $(t^*,t^*+\varepsilon)$ all the formed subclusters splitting from the given cluster remain at positive distance. The main reason is the possible Zeno behavior, that accumulates time events with switches of the collisional type.
\end{remark}

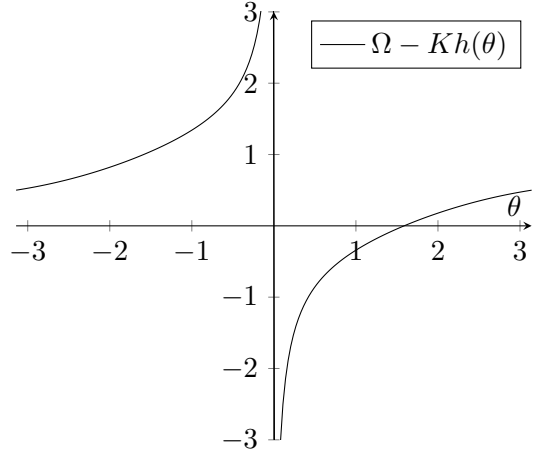
\begin{figure}[t]
\centering
\begin{tikzpicture}
\begin{axis}[
  axis x line=middle, axis y line=middle,
  xmin=-pi, xmax=pi, xtick={-3,...,3}, xlabel=$\theta$,
  ymin=-3, ymax=3, ytick={-3,-2,-1,0,1,2,3},
]
\addplot [
    domain=-pi:pi, 
    samples=200, 
    color=black,
]
{0.5-(sin(deg(x))/pow(abs(x),2*0.75))};
\addlegendentry{$\Omega-Kh(\theta)$}
\end{axis}
\end{tikzpicture}
\caption{Profile of $\Omega-Kh(\theta)$ for $\Omega=0.25$, $K=1$ and $\alpha=0.75$.}
\label{fig:profile-2-particles}
\end{figure}

\section{Rigorous limit towards singular weights}\label{singular-limit-section}

In the previous section, we studied the existence and one-sided uniqueness of absolutely continuous solutions to the singular weighted first order Kuramoto model in all the subcritical, critical and supercritical cases. Because of the continuity of the kernel for $\alpha\in\left(0,\frac{1}{2}\right)$, we can show that in that case the solutions are indeed $C^1$, although we cannot say the same neither for the critical case $\alpha=\frac{1}{2}$ nor for the supercritical case $\alpha\in\left(\frac{1}{2},1\right)$. Also, these results does not necessarily provide any extra regularity of the frequencies $\omega_i=\dot{\theta}_i$ for an augmented second order model to make sense.

Let us recall that in Subsection \ref{formal-singular-limit-subsection}, the singular Kuramoto model was formally obtained as singular limit $\varepsilon\rightarrow 0$ of the scaled regular model \eqref{KM-scaled}-\eqref{C-scaled}. Notice that if apart form heuristically, we rigorously proved the limit $\varepsilon\rightarrow 0$, then we would be led to an alternative existence result for the singular models. In this section, we will inspect to what extend such idea works and how many exponents we can obtain with such technique. We will recover the existence results in Section \ref{well-posedness-singular-section}. Indeed, this technique will yield a gain of piecewise $W^{1,1}$ regularity of the frequencies $\omega_i$ in the subcritical case and will provide an equation for them in weak sense that will be discussed and related with similar models in Subsection \ref{comparison-singular-CS-subsection}. However, such idea fails for the more singular cases, where the compactness of frequencies is very weak. While the singular limit for the subcritical case is straightforward, we need to derive new ideas to deal with the limiting set-valued Filippov map in  the critical and supercritical cases along with the loss of strong compactness of the frequencies in such cases.

\subsection{Limit in the subcritical case and augmented flocking model}\label{singular-limit-subcritical-subsection}

The following result provides a list of a priori estimates for the global-in-time classical solutions of the regularized system \eqref{KM-scaled}-\eqref{C-scaled}, for any $\varepsilon>0$:

\begin{lemma}\label{singular-limit-subcritical-lemma}
Let us consider any initial data $\Theta_0=(\theta_{1,0},\ldots,\theta_{N,0})\in \mathbb{R}^N$ and set  the unique global-in-time classical solution $\Theta^\varepsilon=(\theta_1^\varepsilon,\ldots,\theta_N^\varepsilon)$ to \eqref{KM-scaled}-\eqref{C-scaled} in the subcritical case $\alpha\in\left(0,\frac{1}{2}\right)$, for every $\varepsilon>0$. Then, there exists some non-negative constant $C$ such that
\begin{align*}
\Vert \dot{\Theta}^\varepsilon\Vert_{C^{0,1-2\alpha}([0,\infty),\mathbb{R}^N)}&\leq C,\\
\Vert \Theta^\varepsilon\Vert_{C^{1,1-2\alpha}([0,T],\mathbb{R}^N)}&\leq \vert \Theta_0\vert+CT,
\end{align*}
for every $T>0$ and $\varepsilon>0$. As a consequence, there exists some subsequence of $\{\Theta^\varepsilon\}_{\varepsilon>0}$, that we denote in the same way for simplicity, and some $\Theta\in C^1([0,+\infty),\mathbb{R}^N)$ such that $\dot{\Theta}\in C^{0,1-2\alpha}([0,\infty),\mathbb{R}^N)$, it verifies the same estimates as above and
\[
\{\Theta^\varepsilon\}_{\varepsilon>0}\rightarrow \Theta\ \mbox{ in }C^1([0,T],\mathbb{R}^N),
\]
for every $T>0$.
\end{lemma}

\begin{proof}
All the properties directly follow from the first one along with the Ascoli--Arzel\`a theorem. Recall that there is some constant $M>0$ such that
\[
 \vert h_\varepsilon(\theta)\vert\leq M\ \mbox{ and }\ \vert h_\varepsilon(\theta_1)-h_\varepsilon(\theta_2)\vert\leq M\vert \theta_1-\theta_2\vert_o^{1-2\alpha},
\]
for every $\theta,\,\theta_1,\,\theta_2\in \mathbb{R}$ and every $\varepsilon>0$. Then, the first property is also a straightforward consequence of such uniform-in-$\varepsilon$ boundedness and H\"{o}lder-continuity of the kernel.
\end{proof}

The following result holds true as a clear consequence of the uniform equicontinuity of the sequence $h_\varepsilon$ along with the compactness of the sequence $\{\Theta^\varepsilon\}_{\varepsilon>0}$.

\begin{theorem}\label{singular-limit-subcritical-theorem}
The limit function $\Theta$ of $\{\Theta^\varepsilon\}_{\varepsilon>0}$ in Lemma \ref{singular-limit-subcritical-lemma} is a classical global-in-time solution of the singular model \eqref{KM-scaled}-\eqref{C-scaled} with $\varepsilon=0$ in the subcritical case $\alpha\in \left(0,\frac{1}{2}\right)$.
\end{theorem}

Notice that we have arrived at a construction of classical global-in-time solutions of the singular problem with $0<\alpha<\frac{1}{2}$ through two different techniques: Theorems \ref{T-wp-1} and \ref{singular-limit-subcritical-theorem}. However, both techniques are actually very related since originally, the Filippov theory relies on a similar regularizing procedure.  In what follows, we will see that such procedure provides us with extra a priori estimates for the ``acceleration'' (derivatives of frequencies). Also, such procedure will allow us to derive a ``piecewise weak equation'' for them. This is the rest of the content of this subsection.

Remark that a necessary and sufficient condition for two oscillators $\theta_i$ and $\theta_j$ that collide at some time to stick together is that $\Omega_i=\Omega_j$ by virtue of Theorem \ref{T-sticking-1}. In some sense, those two oscillators are identified in an unique cluster with a bigger ``mass''. Then, we can quantify the times of ``pure collisions'' as follows. Starting with $T_{0}=0$, we define
\begin{equation}\label{collision-times-eq}
T_k:=\inf\{t>T_{k-1}:\,\exists\,i\mbox{ and }j\in S_i(T_{k-1})^c\mbox{ such that }\bar\theta_i(t)=\bar\theta_j(t)\},
\end{equation}
for every $k\in\mathbb{N}$. Recall the notation in Subsection \ref{cluster-notation-subsection}, and see \cite{P2} for related notation in the discrete Cucker--Smale model with singular influence function. Then, taking derivatives in \eqref{KM-scaled}-\eqref{C-scaled} we can obtain the next split
\begin{align}
\ddot{\theta}_i^\varepsilon&=\frac{K}{N}\sum_{j\notin \mathcal C_i(T_{k-1})}h_\varepsilon'(\theta_j^\varepsilon-\theta_i^\varepsilon)(\dot{\theta}_j^\varepsilon-\dot{\theta}_i^\varepsilon)\nonumber\\
&+\frac{K}{N}\sum_{j\in (\mathcal C_i\setminus S_i)(T_{k-1})}h_\varepsilon'(\theta_j^\varepsilon-\theta_i^\varepsilon)(\dot{\theta}_j^\varepsilon-\dot{\theta}_i^\varepsilon)\nonumber\\
&+\frac{K}{N}\sum_{j\in S_i(T_{k-1})}h_\varepsilon'(\theta_j^\varepsilon-\theta_i^\varepsilon)(\dot{\theta}_j^\varepsilon-\dot{\theta}_i^\varepsilon),\label{derivatives-split-eq}
\end{align}
where $t\in [T_{k-1},T_k)$. The idea is to show that  we can pass to the limit in the above expressions in $L^1([T_{k-1},\tau])$-weak, for every $k\in\mathbb{N}$ and for every $\tau\in (T_{k-1},T_k)$. This is the content of the next theorem. Before going on, let us discuss the possible scenarios for the sequence $\{T_k\}_{k\in\mathbb{N}}$ and how can we cover the whole interval $[0,+\infty)$ with them in any case so that our dynamics can be reduced to each of them:

\begin{enumerate}
\item It might happen that there exists some $k_0\in\mathbb{N}$ such that $T_{k_0+1}=+\infty$  (then, $T_k=+\infty$ for every $k> k_0$). This is the case either all particles have stuck together in finite time or after some finite time there is no more collision. In this case
\[
[0,+\infty)=\bigcup_{0\leq k\leq k_0-1}[T_k,T_{k+1})\cup [T_{k_0},+\infty),
\]
and at each interval there is no collision.
\item Also it might happen that the sequence $\{T_k\}_{k\in\mathbb{N}}$ is infinite and unbounded, i.e., $T_k\nearrow +\infty$. Hence,
\[
[0,+\infty)=\bigcup_{k\geq 0}[T_k,T_{k+1}),
\]
and there is no collision in each interval.
\item Finally, it might also be the ``odd'' case that the sequence $\{T_k\}_{k\in\mathbb{N}}$ is infinite but bounded. In such case, there exists some $T^\infty\in \mathbb{R}^+$ with right Zeno behavior, i.e.  $T_k\nearrow T^\infty$. Then, a straightforward argument involving the mean value theorem shows that $T^\infty$ is a sticking point. Then we can split the dynamics up to time $T^\infty$ through
\[
[0,T^\infty)=\bigcup_{k\geq 0}[T_k,T_{k-1}).
\]
Taking $T^\infty$ as our initial time, we can repeat each of the steps 1, 2 and 3 above so that we can globally recover the whole dynamics. Notice that since there just can be $N-1$ times of sticking, then there just can be $N-1$ times like $T^\infty$.
\end{enumerate}

For simplicity in our arguments, we will assume that we lie in the case 2, although the same results apply to any of the other cases. Before going to the heart of the result, let us summarize some good properties of the kernel $h_\varepsilon'$.

\begin{lemma}\label{singular-limit-W21-lemma}
Consider any value $\alpha\in\left(0,\frac{1}{2}\right)$. Then, the following properties hold true:
\begin{enumerate}
\item Formula for the derivative:
\[h_\varepsilon'(\theta)=\frac{1}{(\varepsilon^2+c\vert \theta\vert_o^2)^\alpha}\left[\cos\theta-2\alpha c\frac{\sin\vert \theta\vert_0}{\vert \theta\vert_0}\frac{\vert \theta\vert_0^2}{\varepsilon^2+c\vert\theta\vert_o^2}\right].\]
\item Upper bound by an $L^1(\mathbb{T})$-function:
\[\vert h_\varepsilon'(\theta)\vert,\vert h'(\theta)\vert \leq M\frac{1}{\vert \theta\vert_o^{2\alpha}}.
\]
\item Strong $L^1(\mathbb{T})$ convergence:
\[
h_\varepsilon'\rightarrow h'\ \mbox{ in }L^1(\mathbb{T}).
\]
\item Weighted H\"{o}lder-continuity:
\[
\vert h_\varepsilon'(\theta_1)-h_\varepsilon'(\theta_2)\vert\leq M\frac{\vert \theta_1-\theta_2\vert_o^\beta}{\min\{\vert \theta_1\vert_o,\vert \theta_2\vert_o\}^\gamma},
\]
for every couple of exponents $\beta,\gamma\in (0,1)$ such that $\gamma=2\alpha+\beta$.
\item Weighted $L^\infty(\mathbb{T})$ convergence:
\[
\vert h_\varepsilon'(\theta)-h'(\theta)\vert\leq M\frac{\varepsilon^{1-2\alpha}}{\vert \theta\vert_o}.
\]
\end{enumerate}
\end{lemma}

\begin{proof}
The first two results are straightforward and the third one is a clear consequence of the dominated convergence theorem. The fourth property follows from an obvious application of the mean value theorem and the fifth one is a standard property of mildly singular kernels (one can show that $M=\alpha/\beta$) .
\end{proof}

\begin{theorem}\label{singular-limit-W21-theorem}
For any initial datum $\Theta_0\in\mathbb{R}^N$, consider $\Theta^\varepsilon$, the classical global-in-time solution of \eqref{KM-scaled}-\eqref{C-scaled} in the subcritical case $\alpha\in\left(0,\frac{1}{2}\right)$. Also, consider the limiting $\Theta$ in Theorem \ref{singular-limit-subcritical-theorem} and the collision times $\{T_k\}_{k\in\mathbb{N}}$ in \eqref{collision-times-eq}. Then, the following properties hold true :
\begin{enumerate}
\item For every $i\in\{1,\ldots,N\}$ and $j\notin \mathcal C_i(T_{k-1})$
\[
h_\varepsilon'(\theta_j^\varepsilon-\theta_i^\varepsilon)\rightarrow h'(\theta_j-\theta_i), \quad \text{as} \quad \varepsilon \to 0, \quad \text{in}\quad C([T_{k-1},\tau]).
\]
\item For every $i\in\{1,\ldots,N\}$ and $j\in \mathcal C_i(T_{k-1})\setminus S_i(T_{k-1})$
\[
h_\varepsilon'(\theta_j^\varepsilon-\theta_i^\varepsilon)\rightharpoonup h'(\theta_j-\theta_i), \quad \text{as} \quad \varepsilon \to 0, \quad \text{in}\quad L^1([T_{k-1},\tau]).
\]
\item For every $i\in\{1,\ldots,N\}$ and $j\in S_i(T_{k-1})$
\[
\frac{d}{dt}h_\varepsilon(\theta_j^\varepsilon-\theta_i^\varepsilon)\rightarrow 0, \quad \text{as} \quad \varepsilon \to 0, \quad \text{in}\quad W^{-1,\infty}([T_{k-1},\tau]).
\]
\end{enumerate}
\end{theorem}

\begin{proof} We split the proof in three steps. 

$\,$\\
$\bullet$ \textit{Step 1:} In the first case, fix any $i\in\{1,\ldots,N\}$ and $j\notin \mathcal C_i(T_{k-1})$. There exists (by definition) some positive constant $\delta_0=\delta_0(k,\tau)<\pi$ such that
\[
\vert \theta_i(t)-\theta_j(t)\vert_o \geq \delta_0,\ \mbox{ for all }t\in [T_{k-1},\tau].
\]
Then, by the uniform convergence in Lemma \ref{singular-limit-subcritical-lemma} there exists some $\varepsilon_0>0$ such that 
\begin{equation}\label{P1-1-derivatives}
\vert \theta_i^\varepsilon(t)-\theta_j^\varepsilon(t)\vert_o \geq \frac{\delta_0}{2},\ \mbox{ for all }t\in [T_{k-1},\tau],
\end{equation}
for every $\varepsilon\in(0,\varepsilon_0)$. Consequently, by crossing terms we have
\begin{multline*}
\vert h_\varepsilon'(\theta_j^\varepsilon(t)-\theta_i^\varepsilon(t))-h'(\theta_j(t)-\theta_i(t))\vert \\
\leq \vert h_\varepsilon'(\theta_j^\varepsilon(t)-\theta_i^\varepsilon(t))-h'(\theta_j^\varepsilon(t)-\theta_i^\varepsilon(t))\vert+\vert h'(\theta_j^\varepsilon(t)-\theta_i^\varepsilon(t))-h'(\theta_j(t)-\theta_i(t))\vert,
\end{multline*}
for every $t\in[T_{k-1},\tau]$. Hence, both two terms converge to zero uniformly in $[T_{k-1},\tau]$, as $\varepsilon\rightarrow 0$. This is due to  \eqref{P1-1-derivatives}, the third property in Lemma \ref{singular-limit-W21-lemma}, the uniform continuity of $h'$ in compact sets away from $2\pi\mathbb{Z}$ and the uniform convergence of the phases in Lemma \ref{singular-limit-subcritical-lemma}. This ends the proof of the first part.

\medskip
\noindent
$\bullet$ \textit{Step 2}: In the second case, $i\in\{1,\ldots,N\}$ and $j\in \mathcal C_i(T_{k-1})\setminus S_i(T_{k-1})$. Then, 
\[
\bar\theta_j(T_{k-1})=\bar\theta_i(T_{k-1})\ \mbox{ but }\ \dot{\theta}_j(T_{k-1})\neq \dot{\theta}_i(T_{k-1}).
\]
Thus, it is clear that  we again have $\vert \theta_j(t)-\theta_i(t)\vert_o>0$, for $t\in [\tau^*, \tau]$ and for every $\tau^*\in (T_{k-1},\tau)$. This amounts to saying that the preceding argument again holds in $[\tau^*,\tau]$ and consequently,
\[
h_\varepsilon'(\theta_j^\varepsilon-\theta_i^\varepsilon)\rightarrow h'(\theta_j-\theta_i)\ \mbox{ in }C([\tau^*,\tau]),
\]
for every $\tau^*\in (T_{k-1},\tau)$. Then, we just need to prove the weak convergence in some interval $[T_{k-1},\tau^*]$. Let us set $\tau^*$. Since $\dot{\theta}_j(T_{k-1})\neq \dot{\theta}_i(T_{k-1})$, we can assume without loss of generality that $\delta_0:=\dot{\theta}_j(T_{k-1})-\dot{\theta}_i(T_{k-1})>0$. By continuity of $\dot{\theta}_j$ and $\dot{\theta}_i$, there exists some small $\tau^*\in (T_{k-1},\tau)$ such that
\begin{equation}\label{P1-2-derivatives}
\dot{\theta}_i(t)-\dot{\theta}_j(t)\geq \frac{\delta_0}{2},\ \mbox{ for all }t\in [T_{k-1},\tau^*].
\end{equation}
Then, by the uniform convergence of the frequencies (see Lemma \ref{singular-limit-subcritical-lemma}), we can take a small enough $\varepsilon_0>0$ such that if $\varepsilon\in (0,\varepsilon_0)$ then
\begin{equation}\label{P2-2-derivatives}
\dot{\theta}_i^\varepsilon(t)-\dot{\theta}_j^\varepsilon(t)\geq \frac{\delta_0}{4},\ \mbox{ for all }t\in [T_{k-1},\tau^*].
\end{equation}
In particular, we have well defined inverses of $\theta_j-\theta_i$ and $\theta_j^\varepsilon-\theta_i^\varepsilon$ in $[T_{k-1},\tau^*]$, for every $\varepsilon\in (0,\varepsilon_0)$. Indeed, the inverse function theorem states that:
\begin{equation}\label{P3-2-derivatives}
((\theta_j-\theta_i)^{-1})'=\frac{1}{(\dot{\theta}_j-\dot{\theta}_i)\circ(\theta_j-\theta_i)^{-1}},
\end{equation}
and a similar statement holds for $\theta_j^\varepsilon-\theta_i^\varepsilon$. In order to show the  weak convergence in $L^1([T_{k-1},\tau^*])$, we equivalently claim that the following assertions are true:
\begin{enumerate}
\item Uniform-in-$\varepsilon$ $L^1$ bound of $h_\varepsilon'(\theta_j^\varepsilon-\theta_i^\varepsilon)$ and $h(\theta_j-\theta_i)$ in $[T_{k-1},\tau^*]$, i.e., there exists some constant $M>0$ such that
\[
\Vert h_\varepsilon'(\theta_j^\varepsilon-\theta_i^\varepsilon)\Vert_{L^1([T_{k-1},\tau^*])},\Vert h'(\theta_j-\theta_i)\Vert_{L^1([T_{k-1},\tau^*])}\leq M,
\]
for every $\varepsilon\in (0,\varepsilon_0)$.
\item Convergence of the mean values over finite intervals, i.e., 
\[
\lim_{\varepsilon\rightarrow 0}\int_{T_{k-1}}^{\tau^{**}}(h_\varepsilon'(\theta_j^\varepsilon(t)-\theta_i^\varepsilon(t))-h'(\theta_j(t)-\theta_i(t)))\,dt=0,
\]
for every $\tau^{**}\in (T_{k-1},\tau^*)$.
\end{enumerate}
Let us then prove such claim. Regarding the first assertion, we just focus on $h_\varepsilon'(\theta_j^\varepsilon-\theta_i^\varepsilon)$ (the other case is similar). Due to a simple change of variables $\theta=(\theta_j^\varepsilon-\theta_i^\varepsilon)(t)$ and \eqref{P2-2-derivatives}-\eqref{P3-2-derivatives}
\[
\int_{T_{k-1}}^{\tau^{**}} \vert h_\varepsilon'(\theta_j^\varepsilon(t)-\theta_i^\varepsilon(t))\vert\,dt=\int_{\theta_j^\varepsilon(T_{k-1})-\theta_i^\varepsilon(T_{k-1})}^{\theta_j^\varepsilon(\tau^{**})-\theta_i^\varepsilon(\tau^{**})}\frac{\vert h_\varepsilon'(\theta)\vert d\theta}{(\dot{\theta}_j^\varepsilon-\dot{\theta}_i^\varepsilon)((\theta_j^\varepsilon-\theta_i^\varepsilon)^{-1}(\theta))}\leq \Vert h_\varepsilon'\Vert_{L^1(\mathbb{T})}\frac{4}{\delta_0}.
\]
Then the assertion under consideration follows from the second item in Lemma \ref{singular-limit-W21-lemma}. Regarding the second assertion we split into two terms
\[
\int_{T_{k-1}}^{\tau^{**}}(h_\varepsilon'(\theta_j^\varepsilon-\theta_i^\varepsilon)-h'(\theta_j-\theta_i))\,dt=I_{\varepsilon}+II_{\varepsilon},
\]
where,
\begin{align*}
I_\varepsilon&:=\int_{T_{k-1}}^{\tau^{**}}(h_\varepsilon'(\theta_j^\varepsilon-\theta_i^\varepsilon)-h'(\theta_j^\varepsilon-\theta_i^\varepsilon))\,dt,\\
II_{\varepsilon}&:=\int_{T_{k-1}}^{\tau^{**}}(h'(\theta_j^\varepsilon-\theta_i^\varepsilon)-h'(\theta_j-\theta_i))\,dt.
\end{align*}
The same change of variables as above allows us restate $I_\varepsilon$ in the following way
\[
I_{\varepsilon}=\int_{\theta_j^\varepsilon(T_{k-1})-\theta_i^\varepsilon(T_{k-1})}^{\theta_j^\varepsilon(\tau^{**})-\theta_i^\varepsilon(\tau^{**})}(h_\varepsilon'(\theta)-h'(\theta))\frac{d\theta}{(\dot{\theta}_j^\varepsilon-\dot{\theta}_i^\varepsilon)((\theta_j^\varepsilon-\theta_i^\varepsilon)^{-1}(\theta))}.
\]
Then, estimate \eqref{P2-2-derivatives} along with the strong $L^1(\mathbb{T})$ convergence of the kernels in (3) of Lemma \ref{singular-limit-W21-lemma} shows that $I_\varepsilon$ vanishes when $\varepsilon\rightarrow 0$:
\[
|I_\varepsilon|\leq \frac{4}{\delta_0} \int_{\theta_j^\varepsilon(T_{k-1})-\theta_i^\varepsilon(T_{k-1})}^{\theta_j^\varepsilon(\tau^{**})-\theta_i^\varepsilon(\tau^{**})} |h_\varepsilon^\prime(\theta) - h^\prime(\theta)| d\theta = \frac{4}{\delta_0} \|h_\varepsilon^\prime(\theta) - h^\prime(\theta) \|_{L^1(\mathbb T)} \to 0, \quad \text{as} \quad \varepsilon \to 0.
\]
For the term $II_\varepsilon$, we use the forth item in Lemma \ref{singular-limit-W21-lemma} to show
\begin{align*}
\vert II_\varepsilon\vert&\leq M\int_{T_{k-1}}^{\tau^{**}} \frac{\vert (\theta_j^\varepsilon-\theta_j)-(\theta_i^\varepsilon-\theta_i)\vert_o^\beta}{\min\{\vert \theta_j^\varepsilon-\theta_i^\varepsilon\vert_o,\vert \theta_j-\theta_i\vert_o\}^\gamma} \, dt\\
&\leq 2^\beta M\Vert \Theta^\varepsilon-\Theta\Vert_{C([T_{k-1},\tau^{**}],\mathbb{R}^N)}^\beta\int_{T_{k-1}}^{\tau^{**}}\frac{1}{\min\{\vert \theta_j^\varepsilon-\theta_i^\varepsilon\vert_o,\vert \theta_j-\theta_i\vert_o\}^\gamma}\, dt\\
&\leq 2^\beta M\Vert \Theta^\varepsilon-\Theta\Vert_{C([T_{k-1},\tau^{**}],\mathbb{R}^N)}^\beta\int_{T_{k-1}}^{\tau^{**}}\max\left\{\frac{1}{\vert \theta_j^\varepsilon-\theta_i^\varepsilon\vert_o^\gamma},\frac{1}{\vert \theta_j-\theta_i\vert_o^\gamma}\right\}\,dt\\
&\leq 2^\beta M\Vert \Theta^\varepsilon-\Theta\Vert_{C([T_{k-1},\tau^{**}],\mathbb{R}^N)}^\beta\int_{T_{k-1}}^{\tau^{**}}\left(\frac{1}{\vert \theta_j^\varepsilon-\theta_i^\varepsilon\vert_o^\gamma}+\frac{1}{\vert \theta_j-\theta_i\vert_o^\gamma}\right)\,dt.
\end{align*}
Then, a new change of variables along with the equations \eqref{P2-2-derivatives}-\eqref{P3-2-derivatives} and the local integrability in one dimension of an inverse power of order $\gamma$ entail the existence of a non-negative constant $C$ that does not depend on $\varepsilon$ such that
\[
\vert II_\varepsilon\vert\leq C\Vert \Theta^\varepsilon-\Theta\Vert_{C([T_{k-1},\tau^{**}],\mathbb{R}^N)}.
\]
Then, the second step follows from the uniform convergence of the phases in Lemma \ref{singular-limit-subcritical-lemma}.

$\,$\\
$\bullet$ \textit{Step 3:} In the third case, consider $i\in\{1,\ldots,N\}$ and $j\in S_i(T_{k-1})$. By the uniqueness in Theorem \ref{T-wp-1}, we can ensure that $\theta_j(t)=\theta_i(t)$ for all $t\geq T_{k-1}$. Then, the uniform convergence of the kernels $h_\varepsilon$ along with the uniform convergence of the phases in Lemma \ref{singular-limit-subcritical-lemma} shows that
\[
h_\varepsilon(\theta_j^\varepsilon-\theta_i^\varepsilon)\rightarrow 0\mbox{ in }C([T_{k-1},\tau]),
\]
and then, the result holds true by definition of the norm in $W^{-1,\infty}([T_{k-1},\tau])$.
\end{proof}

\begin{remark}\label{singular-limit-subcritical-remark}
The preceding results show that the unique global-in-time solution $\Theta$ to the problem \eqref{E-2}, with $\alpha\in(0,\frac{1}{2})$, that we constructed in Theorem \ref{T-wp-1}, satisfy that $\theta_i\in C^{1,1-2\alpha}([0,\infty),\mathbb{R}^N)$ and the frequencies $\dot{\theta}_i$ exhibit higher regularity. Indeed, they are piece-wise $W^{1,1}$ in the sense that  $\dot{\theta}_i\in W^{1,1}([T_{k-1},\tau])$, for every $k\in\mathbb{N}$ and every $\tau\in (T_{k-1},T_k)$. In addition, they verify the following equation in weak sense
\begin{equation}\label{singular-second-order}
\ddot{\theta}_i=\frac{K}{N}\sum_{j\notin S(i)(T_{k-1})}h'(\theta_j-\theta_i)(\dot{\theta_j}-\dot{\theta}_i),
\end{equation}
in $[T_{k-1},\tau]$. Throughout the proof of the above result we have just used the local integrability in one dimension of any inverse power of order smaller than $1$. However, one might have tried to use that such inverse powers actually belong to $L^p_{loc}$ in order to show that in Steps 2 the convergence take place in $L^p([T_{k-1},\tau])$-weak for any $1\leq p<\frac{1}{2\alpha}$. In this way, the gain of regularity is in reality higher, namely $\dot{\theta}_i\in W^{1,p}([T_{k-1},\tau])$, for every $1\leq p<\frac{1}{2\alpha}$.
\end{remark}

In the following, we will discuss the corresponding singular limit in the critical and supercritical case. Since the Filippov set-valued map is relatively simpler in that latter case, we will start with that supercritical case. Later, we will adapt the ideas therein to show a parallel result in the critical regime.

\subsection{Limit in the supercritical case}\label{singular-limit-supercritical-subsection}
Using a similar vector notation to that in \eqref{KM3-vector} for the singular weighted model, our regularized system \eqref{KM-scaled}-\eqref{C-scaled} can be restated as
\[
\left\{\begin{array}{l}
\dot{\Theta}^\varepsilon=H^\varepsilon(\Theta^\varepsilon),\\
\Theta^\varepsilon(0)=\Theta_0 ,
\end{array}\right.
\]
where the components of the vector field $H^\varepsilon$ read
\[
H^\varepsilon_i(\Theta)=\Omega_i+\frac{K}{N}\sum_{j\neq i}h_\varepsilon(\theta_j-\theta_i),
\]
for every $\Theta\in\mathbb{R}^N$ and every $i\in\{1,\ldots,N\}$. Then, one can mimic the ideas in Section \ref{preliminaries-section} to show that the regularized system can also be written as a gradient flow
\begin{equation}\label{KM-regularized-grad-flow}
\left\{\begin{array}{l}
\dot{\Theta}^\varepsilon=-\nabla V^\varepsilon(\Theta^\varepsilon),\\
\Theta^\varepsilon(0)=\Theta_0,
\end{array}\right.
\end{equation}
where the regularized potential now reads
\begin{equation}\label{KM-regularized-potential}
V^\varepsilon(\Theta):=-\sum_{i=1}^N\Omega_i\theta_i+V_{int}^\varepsilon(\Theta):=-\sum_{i=1}^N\Omega_i\theta_i+\frac{K}{2N}\sum_{i\neq j}W_\varepsilon(\theta_i-\theta_j),
\end{equation}
for every $\Theta\in\mathbb{R}^N$. Again, $W_\varepsilon$ is the anti-derivative of $h_\varepsilon$ such that $W_\varepsilon(0)=0$, i.e.,
\[
W_\varepsilon(\theta):=\int_0^\theta h_\varepsilon(\theta')\,d\theta'.
\]
Also, it is clear that $W_\varepsilon\geq 0$ in the supercritical case, for every $\varepsilon>0$. Then, the following result holds true.
\begin{lemma}\label{L-limit-supercritical}
In the supercritical case $\alpha\in (\frac{1}{2},1)$, consider the unique global-in-time classical solution $\Theta^\varepsilon$ to the regularized system \eqref{KM-regularized-grad-flow}. Then,
\[
\frac{1}{2}\int_0^t\vert \dot{\Theta}^\varepsilon(s)\vert^2\,ds\leq \frac{C_\Omega^2}{2}t+V_{int}(\Theta_0),
\]
for every $t>0$ and every $\varepsilon>0$, where $C_\Omega:=\left(\sum_{i=1}^N\Omega_i^2\right)^{1/2}$.
\end{lemma}

The above result shows that $\{\Theta^\varepsilon\}_{\varepsilon>0}$ is bounded in $H^1((0,T),\mathbb{R}^N)$, for every $T>0$. Then, there exists some subsequence that we denote in the same way so that $\{\Theta^\varepsilon\}_{\varepsilon>0}$ weakly converge to some $\Theta\in H^1_{loc}((0,\infty),\mathbb{R}^N)$ in $H^1((0,T),\mathbb{R}^N)$ for every $T>0$. The Sobolev embedding and the definition of weak convergence ensure that 
\begin{align*}
\Theta^\varepsilon\rightarrow \Theta &\mbox{ in }C([0,T],\mathbb{R}^N),\\
\dot{\Theta}^\varepsilon\rightharpoonup\dot{\Theta}&\mbox{ in }L^2((0,T),\mathbb{R}^N),
\end{align*}
for every $T>0$. Before we obtain the desired convergence result of \eqref{KM-regularized-grad-flow} towards a Filippov solution, let us introduce the following split of the frequencies:
\begin{equation}\label{KM-regularized-split-solution}
\dot{\Theta}^\varepsilon(t)=x^\varepsilon(t)+y^\varepsilon(t),
\end{equation}
where, componentwise, each term reads as follows
\begin{align*}
x^\varepsilon_i(t)&=\frac{K}{N}\sum_{j\notin \mathcal C_i(t)}(h_\varepsilon(\theta_j^\varepsilon(t)-\theta_i^\varepsilon(t))-h(\theta_j(t)-\theta_i(t))),\\
y^\varepsilon_i(t)&=\frac{K}{N}\sum_{j\notin  \mathcal C_i(t)}h(\theta_j(t)-\theta_i(t))+\frac{K}{N}\sum_{j\in  \mathcal C_i(t)}h_\varepsilon(\theta_j^\varepsilon(t)-\theta_i^\varepsilon(t)).
\end{align*}
Then, it is clear by definition that 
\begin{align*}
x^\varepsilon\rightarrow 0 &\mbox{ in }C([0,T],\mathbb{R}^N),\\
y^\varepsilon\rightharpoonup\dot{\Theta}&\mbox{ in }L^2((0,T),\mathbb{R}^N),
\end{align*}
for every $T>0$, and $y^\varepsilon(t)\in\mathcal{H}(\Theta(t))$, for every $t\geq 0$. As a consequence, we infer that $\Theta^\varepsilon$ becomes a Filippov approximate solution in the following sense:
\begin{equation}\label{KM-approximate-Fil-sol-supercritical}
\dot{\Theta}^\varepsilon(t)\in \mathcal{H}(\Theta(t))+x^\varepsilon(t).
\end{equation}

\begin{remark}\label{R-pointwise-singular-limit-supercritical}
Recall that $\mathcal{H}(\Theta(t))$ is a closed set, for every $t\geq 0$, see Proposition \ref{L-Fil-map-prop}. Consequently, in order to prove that the limiting $\Theta(t)$ yields a Filippov solution, it would be enough to show the almost everywhere convergence of the sequence $\{\dot{\Theta}^\varepsilon\}_{\varepsilon>0}$ towards $\dot{\Theta}$. Unfortunately, it is well known that weak convergence in $L^2$ is not enough for that purpose. Hence, we must deal only with such weak convergence.
\end{remark}

Before going to the heart of the matter, we need to exhibit another characterization of the Filippov set-valued map in terms of implicit equations. The next technical lemma will be used for that. For the sake of clarity, a proof has been provided in Lemma \ref{L-skew-sym-property-appendix} of Appendix \ref{Appendix-Hrep-Fil-map}.

\begin{lemma}\label{L-skew-sym-property}
Consider any $n\in\mathbb{N}$ and any vector $x\in\mathbb{R}^n$. Then, the following assertions are equivalent:
\begin{enumerate}
\item There exists some $Y\in\Skew_n(\mathbb{R})$ such that 
\[
x=Y\cdot\mathbf{j}.
\]
\item The following implicit equation holds true
\[
x\cdot \mathbf{j}=0,
\]
\end{enumerate}
where $\mathbf{j}$ stands for the vector of ones.
\end{lemma}

Hence, we are ready to obtain the above-mentioned characterization.

\begin{proposition}\label{P-explicit-Fil-map-supercritical-implicit-eq}
In the supercritical regime $\alpha>\frac{1}{2}$, the Filippov set-valued map $\mathcal{H}=\mathcal{H}(\Theta)$ associated with $H=H(\Theta)$ consists in the affine subspace of dimension $N-\kappa$ of points $(\omega_1,\ldots,\omega_N)\in\mathbb{R}^N$ obeying the following implicit equations (recall Subsection \ref{cluster-notation-subsection})
\begin{equation}\label{E-explicit-Fil-map-supercritical-implicit-eq}
\frac{1}{n_k}\sum_{i\in E_k}\omega_i=\frac{1}{n_k}\sum_{i\in E_k}\left(\Omega_i+\frac{K}{N}\sum_{j\notin \mathcal{C}_i}h(\theta_j-\theta_i)\right),
\end{equation}
for every $k=1,\ldots,\kappa.$
\end{proposition}

\begin{proof}
By Proposition \ref{P-explicit-Fil-map-supercritical}, $\mathcal{H}(\Theta)$ consists of the set of points $(\omega_1,\ldots,\omega_N)\in\mathbb{R}^N$ such that for every $k=1,\ldots,\kappa$ there exist a skew symmetric matrix $Y^k\in \Skew_{n_k}(\mathbb{R})$ and a bijection $\sigma^k:\{1,\ldots,n_k\}\longrightarrow E_k$ such that the following equations hold true
\[
\omega_{\sigma^k_i}=\Omega_{\sigma^k_i}+\frac{K}{N}\sum_{\substack{m=1\\ m\neq k}}^{\kappa}n_mh(\theta_{\iota_m}-\theta_{\iota_k})+\frac{K}{N}\sum_{j=1}^{n_k}y_{ij}^k,
\]
for every $i=1,\ldots,n_k$. Then, the result follows by applying Lemma \ref{L-skew-sym-property} to each of the above sets of $n_k$ equations to the particular vectors $x^k\in \mathbb{R}^{n_k(\Theta)}$ with components:
\[
x_i^k:=\omega_{\sigma^k_i}-\Omega_{\sigma^k_i}-\frac{K}{N}\sum_{\substack{m=1\\ m\neq k}}^{\kappa}n_mh(\theta_{\iota_m}-\theta_{\iota_k}),\ i=1,\ldots,n_k,
\]
when we equivalently restate it using the notation in Subsection \ref{cluster-notation-subsection}.
\end{proof}

\begin{remark}
Recall that in the subcritical case $\alpha\in (0,\frac{1}{2})$ in Subsection \ref{singular-limit-subcritical-subsection}, any strong limit $\Theta$ already yielded a solution to the limiting system \eqref{E-2}. Indeed, there just can be one and only one such strong limit by the one-sided uniqueness of the limiting system \eqref{E-2} Also, in that subcritical case one can find a nice split of the dynamics in a sequence of intervals where no collision happens. Thus, on every such interval, the kind of collisional state of our trajectory remains unchanged. Let us remember that the reason why that sequence fills the whole half line in the subcritical case relies on the following facts: first, by uniqueness we can characterize the sticking of oscillators and once they stick during some time they remain stuck for all times. In particular, only $N-1$ sticking times can exist. Second, when an accumulation of collisions takes place, it has to be at a sticking time. Hence, there just can be $N-1$ such accumulations of collisions, thus recovering the whole half line.

Unfortunately, at this step we are missing for the critical and supercritical cases $\alpha\in \left[\frac{1}{2},1\right)$ whether any limit $\Theta$ becomes a Filippov solution to the limiting system \eqref{E-2}. Thus, despite the fact that we have clear characterizations of sticking of such solutions we cannot apply them to any such limit $\Theta$. In addition, the behavior of any $H^1$ weak limit can be very wild. Specifically, a possible scenario of a $H^1_{loc}$ trajectory is that sticking might happen just for a short period of time and, after it, the cluster splits. Also, ``pure collisions'' might accumulate at a non-sticking time exhibiting Zeno behavior (recall Remark \ref{R-Zeno}). Thereby, a split of the dynamics into countably many intervals $(T_k,T_{k+1})$ like in the above Subsection \ref{singular-limit-subcritical-subsection}, where the collisional state remains unmodified, is not viable.
\end{remark}

Since the above Remark prevent us to achieve a split of the dynamics into countably many time intervals that fills the whole half-line and, each of them exhibiting unvaried collisional state, we will develop a new approach supported by the above explicit $H$-representation of the Filippov set-valued map at any collision state. One of our main tools will be the \textit{Kuratowski--Ryll--Nardzewski measurable selection theorem} \cite{KRN} that applies to set-valued Effros-measurable maps. For the sake of completeness we include the statement of such result that we adapt to a finite-dimensional setting.
\begin{lemma}[Kuratowski--Ryll--Nardzewski]\label{L-KRN}
Consider any $n,m\in\mathbb{N}$ and any set-valued map $\mathcal{F}:\mathbb{R}^n\longrightarrow 2^{\mathbb{R}^m}$ with values in the non-empty and closed subsets of $\mathbb{R}^m$. Assume that $\mathcal{F}$ is Effros-measurable, that is, for every open set $U\subseteq \mathbb{R}^m$, the following set is measurable
\[
\{x\in\mathbb{R}^n:\,\mathcal{F}(x)\cap U\neq \emptyset\}.
\]
Then, $\mathcal{F}$ has a measurable selection, i.e., there exists a measurable function $F:\mathbb{R}^n\longrightarrow\mathbb{R}^m$ such that 
\[
F(x)\in \mathcal{F}(x),\ \mbox{ a.e. }\ x\in\mathbb{R}^n.
\]
\end{lemma}

Sometimes, it is helpful to control how many of these single-valued measurable selections of the Effros-measurable set-valued map do we need in order to essentially have the whole set-valued map ``represented'' in some sense. This is the content of an intimately related result: the \textit{Castaing representation theorem}, see \cite[Theorem III.30]{CV}.
\begin{lemma}[Castaing]\label{L-Castaing}
Consider any $n,m\in\mathbb{N}$ and any set-valued map $\mathcal{F}:\mathbb{R}^n\longrightarrow 2^{\mathbb{R}^m}$ with values in the non-empty and closed subsets of $\mathbb{R}^m$. Assume that $\mathcal{F}$ is Effros-measurable. Then $\mathcal{F}$ has a Castaing representation, i.e., there exists a sequence $\{F^n\}_{n\in\mathbb{N}}$ of measurable maps $F^n:\mathbb{R}^n\longrightarrow\mathbb{R}^m$ such that 
\[
\mathcal{F}(x)=\overline{\{F^n(x):\,n\in\mathbb{N}\}},\ \mbox{ a.e. }x\in\mathbb{R}^n.
\]
\end{lemma}

Such results will be directly applied to the critical case in the next Subsection \ref{singular-limit-critical-subsection}. However, for the supercritical case, we will need a refinement of the above theorem to allow for integrable representations of the set-valued map. The Effros-measurability has to be improved to some integrability condition for set-valued maps. We will focus on the next result.

\begin{lemma}\label{L-Castaing-integrable}
Consider any $n,m\in\mathbb{N}$ and any set-valued map $\mathcal{F}:\mathbb{R}^n\longrightarrow 2^{\mathbb{R}^m}$ with values in the non-empty and closed subsets of $\mathbb{R}^m$. Assume that $\mathcal{F}$ is Effros-measurable and strongly integrable, that is, the single-valued map $\vert \mathcal{F}\vert$ is integrable, where $\vert \mathcal{F}\vert$ is defined by
\[
\vert \mathcal{F}\vert(x):=\sup\{\vert y\vert:\,y\in \mathcal{F}(x)\},\ \mbox{ a.e. }x\in \mathbb{R}^n.
\]
Then, every measurable selection of $\mathcal{F}$ is integrable. In particular, $\mathcal{F}$ enjoys a Castaing representation consisting of integrable selections.
\end{lemma}

\begin{proof}
Let us take any measurable selection $F$ of the set-valued $\mathcal{F}$, that exists by Lemma \ref{L-KRN}. Then, by definition of $\vert \mathcal{F}\vert$ we obtain
\[
\vert F(x)\vert\leq \vert \mathcal{F}\vert(x),\ \mbox{ a.e. } x\in\mathbb{R}^n.
\]
Since $\vert \mathcal{F}\vert$ is integrable, the first part of the result holds true. The second one is a simple consequence of the first one along with Lemma \ref{L-Castaing}.
\end{proof}

\begin{remark}\label{R-Castaing-integrable}
Notice that the same ideas as in the above result in Lemma \ref{L-Castaing-integrable} also yield similar statements for the spaces $L^1_{loc}(\mathbb{R}^n)$ and $L^\infty(\mathbb{R}^n)$. Namely,
\begin{enumerate}
\item If $\mathcal{F}$ is locally strongly integrable, i.e., $\vert\mathcal{F}\vert \in L^1_{loc}(\mathbb{R}^n)$, then every measurable selection belongs to $L^1_{loc}(\mathbb{R}^n)$.
\item If $\mathcal{F}$ is strongly essentially bounded, i.e., $\vert\mathcal{F}\vert\in L^\infty(\mathbb{R}^n)$, then each measurable selection belongs to $L^\infty(\mathbb{R}^n)$.
\end{enumerate}
\end{remark}

\begin{theorem}\label{singular-limit-supercritical-theorem}
Consider the classical solutions $\{\Theta^\varepsilon\}_{\varepsilon>0}$ to the regularized system \eqref{KM-regularized-grad-flow} with $\alpha\in (\frac{1}{2},1)$ and any weak $H^1_{loc}$ limit $\Theta$. Then,
\[
\dot{\Theta}(t)\in \mathcal{H}(\Theta(t))\mbox{ a.e. }t\geq 0.
\]
\end{theorem}

\begin{proof}

$\bullet$ \textit{Step 1: $H$-representation of the Filippov map.} By virtue of Proposition \ref{P-explicit-Fil-map-supercritical-implicit-eq}
\begin{equation}\label{E-36}
\mathcal{H}(\Theta(t))= \bigcap_{l=1}^{\kappa(t)} \mathcal{P}_l(t),
\end{equation}
where each of the $\mathcal{P}_l(t)$ stands for the hyperplane $\mathcal{P}_l(t):=\{x\in\mathbb{R}^N:\, a_l(t)\cdot x=b_l(t)\}$. Here, the above vector and scalar functions $a_l(t)$ and $b_l(t)$ read as follows
\[
a_l(t):=\frac{1}{n_l(t)}\sum_{i\in E_l(t)}e_i,\hspace{0.25cm} b_l(t):=\frac{1}{n_l(t)}\sum_{i\in E_l(t)}\left(\Omega_i+\frac{K}{N}\sum_{j\notin \mathcal C_i(t)} h(\theta_j-\theta_i)\right).
\]

\noindent
$\bullet$ \textit{Step 2: Castaing representation of coefficients.} Also, let us define $\mathcal{A}:\mathbb{R}_0^+\longrightarrow 2^{\mathbb{R}^N}$ and $\mathcal{B}:\mathbb{R}_0^+\longrightarrow 2^\mathbb{R}$ by
\[
\mathcal{A}(t):=\{a_l(t):\ l=1,\ldots,\kappa(t)\}\ \mbox{ and }\ \mathcal{B}(t):=\{b_l(t): \ l=1,\ldots,\kappa(t)\}.
\]
It is clear that both maps take closed non-empty values and they are Effros-measurable. Then, Lemma \ref{L-Castaing} allows obtaining a Castaing representation of both maps. On the one hand, $\mathcal{A}$ is strongly essentially bounded (see Remark \ref{R-Castaing-integrable}), thus there exists a sequence $\{A^n\}_{n\in\mathbb{N}}\subseteq L^\infty(0,+\infty)$ such that 
\[
\mathcal{A}(t)=\overline{\{A^n(t):\,n\in\mathbb{N}\}},
\]
for almost every $t\geq 0$. By the finiteness of $\mathcal{A}(t)$ we equivalently have
\begin{equation}\label{E-37}
\{a_l(t):\,l=1,\ldots,\kappa(t)\}=\{A^n(t):\,n\in\mathbb{N}\},
\end{equation}
for almost every $t\geq 0$. However, it is not clear whether $\mathcal{B}$ is strongly locally integrable since we expect eventual switches of the collisional type of the limiting $\Theta(t)$, thus on its coefficients $b_l(t)$.

\medskip
\noindent
$\bullet$ \textit{Step 3: Strong local integrability of $\mathcal{B}$.} Let us show that the above wild behavior still does not prevent us from our goal. Consider the regularized coefficients
\[
b_l^\varepsilon(t):=\frac{1}{n_l(t)}\sum_{i\in E_l(t)}\left(\Omega_i+\frac{K}{N}\sum_{j\notin \mathcal C_i(t)} h_\varepsilon(\theta_j^\varepsilon-\theta_i^\varepsilon)\right),\ l=1,\ldots,\kappa(t).
\]
We can associate a similar set-valued map $\mathcal{B}^\varepsilon:\mathbb{R}_0^+\longrightarrow 2^{\mathbb{R}}$ defined by
\[
\mathcal{B}^\varepsilon(t)=\{b_l^\varepsilon(t):\,l=1,\ldots,\kappa(t)\}.
\]
Notice that, by definition it is clear that
\[
\lim_{\varepsilon\rightarrow 0}b_l^\varepsilon(t)=b_l(t),
\]
for every $l=1,\ldots,\kappa(t)$ since $j\notin \mathcal C_i(t)$ in their definitions and, at those $\theta_j(t)-\theta_i(t)$, the limiting kernel $h$ is continuous. Since both $\mathcal{B}(t)$ and $\mathcal{B}^\varepsilon(t)$ consist of finitely many terms, we deduce that
\begin{equation}\label{E-38}
\vert\mathcal{B}^\varepsilon(t)\vert\longrightarrow\vert\mathcal{B}(t)\vert,\ \mbox{ a.e. }t\in\mathbb{R}_0^+.
\end{equation}
Then, Fatou's lemma on any finite time interval $[0,T]\subseteq \mathbb{R}_0^+$ with $T>0$ entails
\begin{equation}\label{E-39}
\int_{0}^{T}\vert \mathcal{B}\vert(t)\,dt\leq \liminf_{\varepsilon\rightarrow 0}\int_{0}^{T}\vert \mathcal{B}^\varepsilon\vert(t)\,dt.
\end{equation}
By definition, it is clear that
\[
\dot{\Theta}^\varepsilon(t)\cdot a_l(t)=\frac{1}{n_l(t)}\sum_{i\in E_l(t)}\left(\Omega_i+\frac{K}{N}\sum_{j=1}^Nh_\varepsilon(\theta_j^\varepsilon-\theta_i^\varepsilon)\right)=b_l^\varepsilon(t),
\]
where we have cancelled the terms with $j\in E_l(t)$ in the last step by the antisymmetry of $h_\varepsilon$. Then, our set-valued maps are strongly dominated as follows
\begin{equation}\label{E-40}
\vert \mathcal{B}^\varepsilon\vert(t)\leq \vert \dot{\Theta}^\varepsilon(t)\vert\ \mbox{ a.e. }t\geq 0. 
\end{equation}
Putting \eqref{E-40} into \ref{E-39} we obtain
\begin{equation*}
\begin{aligned}
\int_{0}^{T}\vert \mathcal{B}\vert(t)\,dt&=\int_{0}^{T}\liminf_{\varepsilon\rightarrow 0}\vert \mathcal{B}^\varepsilon\vert(t)\,dt\leq \liminf_{\varepsilon\rightarrow 0}\int_{0}^{T}\vert \dot{\Theta}^\varepsilon(t)\vert\,dt\\
&\leq T^{1/2}\liminf_{\varepsilon\rightarrow 0}\left(\int_0^T\vert \dot{\Theta}^\varepsilon(t)\vert^2\,dt\right)^{1/2}\leq T^{1/2}\left(C_\Omega^2T+2V_{int}(\Theta_0)\right)^{1/2}<\infty.
\end{aligned}
\end{equation*}
Here, we have used the Cauchy--Schwarz inequality in the second step and the a priori bound in Lemma \ref{L-limit-supercritical} in the last one. Then, Remark \ref{R-Castaing-integrable} yields the existence of a Castaing representation $\{B^n\}_{n\in\mathbb{N}}\subseteq L^1_{loc}(0,+\infty)$ of the map $\mathcal{B}$. Again, we conclude that
\begin{equation}\label{E-41}
\{b_l(t):\,l=1,\ldots,\kappa(t)\}=\{B^n(t):\,n\in\mathbb{N}\},
\end{equation}
for almost every $t\geq 0$.

\medskip
\noindent
$\bullet$ \textit{Step 4: Conclusion.} Since $y^\varepsilon(t)\in \mathcal{H}(\Theta(t))$, for every $\varepsilon>0$ and every $t\geq 0$, then the H-representation \eqref{E-36} along with the essentially bounded and locally integrable representations \eqref{E-37} and \eqref{E-41} yield the equations
\[
A^n(t)\cdot y^\varepsilon(t)=B^n(t),\ n\in\mathbb{N},
\]
for almost every $t\geq 0$. In particular,
\[
\int_0^{+\infty}A^n(t)\cdot y^\varepsilon(t)\,\varphi(t)\,dt=\int_0^{+\infty}B^n(t)\,\varphi(t)\,dt,
\]
for every $\varepsilon>0$, each $\varphi\in C_c(\mathbb{R}^+)$ and any $n\in\mathbb{N}$. Notice that the boundedness and local integrability of our selectors allows such expression to make sense. We can now use the weak convergence in $L^2$ of $y^\varepsilon$ towards $\dot{\Theta}$ to obtain
\[
\int_0^{+\infty}A^n(t)\cdot \dot{\Theta}(t)\,\varphi(t)\,dt=\int_0^{+\infty}B^n(t)\,\varphi(t)\,dt,
\]
for every $\varphi\in C_c(\mathbb{R}^+)$ and each $n\in\mathbb{N}$. The fundamental lemma of calculus of variations along with the Castaing representations in \eqref{E-37} and \eqref{E-41} and the H-representation in \eqref{E-36} allow us to conclude the desired result.
\end{proof}

\subsection{Limit in the critical case}\label{singular-limit-critical-subsection}
In this Subsection, we will address the singular limit of the regularized system \eqref{KM-scaled}-\eqref{C-scaled} towards a Filippov solution to \eqref{E-2} in the critical regime $\alpha=\frac{1}{2}$. We will mostly apply a similar approach to that in the supercritical regime. Nevertheless, there are several novelties to be considered, that make the study slightly different. First, we will show that we actually enjoy a better $W^{1,\infty}$ a priori estimate, apart from the above $H^1$ bound in Lemma \ref{L-limit-supercritical}. Second, the explicit expression of the Filippov map in Proposition \ref{P-explicit-Fil-map-supercritical-implicit-eq} in terms of intersection of hyperplanes will be adapted to this case.

\begin{lemma}\label{L-limit-critical}
In the critical regime $\alpha=\frac{1}{2}$, consider the unique global-in-time solution $\Theta^\varepsilon$ to the regularized system \eqref{KM-regularized-grad-flow}. Then,
\[
\Vert \dot{\Theta}^\varepsilon\Vert_{L^\infty((0,\infty),\mathbb{R}^N)}\leq C_\Omega+K,
\]
for every $\varepsilon>0$, where $C_\Omega:=\left(\sum_{i=1}^N\Omega_i^2\right)^{1/2}$.
\end{lemma}

We omit the proof since it is a clear consequence of the boundedness of $h$ in the critical case. As a consequence of the above Lemma \ref{L-limit-critical}, we infer the existence of a subsequence of $\{\Theta^\varepsilon\}_{\varepsilon>0}$ that we denote in the same way so that it weakly-* converges to some $\Theta\in W^{1,\infty}_{loc}((0,\infty);\mathbb{R}^N)$ in $W^{1,\infty}((0,T),\mathbb{R}^N)$, for every $T>0$. In particular
\begin{align*}
\Theta^\varepsilon\rightarrow \Theta &\mbox{ in }C([0,T],\mathbb{R}^N),\\
\dot{\Theta}^\varepsilon\overset{*}{\rightharpoonup}\dot{\Theta}&\mbox{ in }L^\infty((0,T),\mathbb{R}^N),
\end{align*}
for every $T>0$. In addition, the same split as in \eqref{KM-regularized-split-solution} can be considered and we obtain
\begin{align*}
x^\varepsilon\rightarrow 0 &\mbox{ in }C([0,T],\mathbb{R}^N),\\
y^\varepsilon\overset{*}{\rightharpoonup}\dot{\Theta}&\mbox{ in }L^\infty((0,T),\mathbb{R}^N),
\end{align*}
and  $y^\varepsilon(t)\in\mathcal{H}(\Theta(t))$, for every $t\geq 0$ and $\varepsilon>0$. Hence, $\Theta^\varepsilon$ becomes an approximate solution in the same sense as in \eqref{KM-approximate-Fil-sol-supercritical}. What is more, the same Remark \ref{R-pointwise-singular-limit-supercritical} is in order. Then, again we cannot ensure pointwise convergence of $\dot{\Theta}^\varepsilon$. In order to obtain an analogue characterization of the Filippov map, we will need the next technical lemma.

\begin{lemma}\label{L-skew-sym-property-bounded-coefficients}
Consider any $n\in\mathbb{N}$ and any vector $x\in \mathbb{R}^n$. Then, the following two assertions are equivalent:
\begin{enumerate}
\item There exists some $Y\in\Skew_n([-1,1])$ such that
\[
x=Y\cdot \mathbf{j}.
\]
\item We have
\[
\sum_{i=1}^k x_{\sigma_i}\in [-k(n-k),k(n-k)],
\]
for every permutation $\sigma$ of $\{1,\ldots,n\}$ and any $k\in\mathbb{N}$.
\end{enumerate}
\end{lemma}

A complete proof is provided in Appendix \ref{Appendix-Hrep-Fil-map}. The following result is a consequence of Lemma \ref{L-skew-sym-property-bounded-coefficients} along with the explicit formula in Proposition \ref{P-explicit-Fil-map-critical}.

\begin{proposition}\label{P-explicit-Fil-map-critical-implicit-eq}
In the critical regime $\alpha=\frac{1}{2}$, the Filippov set-valued map $\mathcal{H}=\mathcal{H}(\Theta)$ associated with $H=H(\Theta)$ is the compact and convex polytope of points $(\omega_1,\ldots,\omega_N)\in\mathbb{R}^N$ whose H-representation consist of the affine inequalities (recall Subsection \ref{cluster-notation-subsection})
\begin{equation}\label{E-explicit-Fil-map-critical-implicit-eq}
\frac{1}{m}\sum_{i\in I}\omega_i\in \frac{1}{m}\sum_{i\in I}\left(\Omega_i+\frac{K}{N}\sum_{j\notin \mathcal{C}_i}h(\theta_j-\theta_i)\right)+\left[-\frac{K}{N}(n_k-m),\frac{K}{N}(n_k-m)\right],
\end{equation}
for every $k=1,\ldots,\kappa$, and $I\subseteq E_k$ with $\#I=m$.
\end{proposition}

Then, we move to the main result, i.e., the convergence of the singular limit towards a Filippov solution to the critical system.

\begin{theorem}\label{singular-limit-critical-theorem}
Consider the classical solutions $\{\Theta^\varepsilon\}_{\varepsilon>0}$ to the regularized system \eqref{KM-regularized-grad-flow} with $\alpha=\frac{1}{2}$ and any weak-* limit $\Theta$ in $W^{1,\infty}_{loc}$. Then,
\[
\dot{\Theta}(t)\in \mathcal{H}(\Theta(t))\mbox{ a.e. }t\geq 0.
\]
\end{theorem}

\begin{proof}
We mimic the proof of Theorem \ref{singular-limit-supercritical-theorem}. Recall that by the above Proposition \ref{P-explicit-Fil-map-critical-implicit-eq}, an analogue H-representation to that in \eqref{E-36} holds. Specifically,
\begin{equation}\label{E-42}
\mathcal{H}(\Theta(t))= \bigcap_{l=1}^{\kappa(t)}\bigcap_{I\subseteq E_l}(\mathcal{S}_{l,I}^+(t)\cap \mathcal{S}_{l,I}^-(t)),
\end{equation}
where the semi-spaces read
\begin{align*}
\mathcal{S}_{l,I}^{+}(t)&:=\{x\in \mathbb{R}^N:\,a_{l,I}(t)\cdot x\leq b_{l,I}^{+}(t)\},\\
\mathcal{S}_{l,I}^{-}(t)&:=\{x\in \mathbb{R}^N:\,a_{l,I}(t)\cdot x\geq b_{l,I}^{-}(t)\},
\end{align*} 
for every $I\subseteq E_l(t)$. We set
\begin{align*}
a_{l,I}(t)&:=\frac{1}{m}\sum_{i\in I}e_i,\\
b_{l,I}^{\pm}(t)&:=\frac{1}{m}\sum_{i\in I}\left(\Omega_i+\frac{K}{N}\sum_{j\notin\mathcal{C}_i(t)}h(\theta_j(t)-\theta_i(t))\right)\pm (n_l(t)-m),
\end{align*}
where $m = \# I$. Now, the coefficients are clearly uniformly bounded. Then, a straightforward application of Remark \ref{R-Castaing-integrable} leads to the existence of essentially bounded selectors for the coefficients. Namely, we can give an ordering such as
\begin{align}
\{a_{l,I}(t):\,l=1,\ldots,\kappa(t),\,I\subseteq E_l(t)\} &= \{A^n(t):\,n\in \mathbb{N}\}\label{E-43}\\
\{b_{l,I}^\pm(t):\,l=1,\ldots,\kappa(t),\,I\subseteq E_l(t)\} &= \{B^{\pm,n}(t):\,n\in \mathbb{N}\},\label{E-44}
\end{align}
for almost every $t\geq 0$. 
Recall that $y^\varepsilon(t)\in \mathcal{H}(\Theta(t))$, for every $\varepsilon>0$ and every $t\geq 0$. Then, by virtue of \eqref{E-42}, \eqref{E-43} and \eqref{E-44}, we equivalently have
\[
A^n(t)\cdot y^\varepsilon(t)\leq B^{+,n}(t)\ \mbox{ and }\ A^n(t)\cdot y^\varepsilon(t)\geq B^{-,n}(t),
\]
for all $n\in \mathbb{N}$, each $\varepsilon>0$ and almost every $t \geq 0$. In particular,
\begin{align*}
\int_0^{+\infty} A^n(t)\cdot y^\varepsilon(t)\,\varphi(t)\,dt&\leq \int_0^{+\infty}B^{+,n}\,dt,\\
\int_0^{+\infty} A^n(t)\cdot y^\varepsilon(t)\,\varphi(t)\,dt&\geq \int_0^{+\infty}B^{-,n}\,dt,
\end{align*}
for all $n\in \mathbb{N}$, each $\varepsilon>0$ and any non-negative $\varphi\in C_c(\mathbb{R}^+)$. Then, using the  weak-* convergence in $L^\infty$ we obtain that
\begin{align*}
\int_0^{+\infty} A^n(t)\cdot \dot{\Theta}(t)\,\varphi(t)\,dt&\leq \int_0^{+\infty}B^{+,n}\,dt,\\
\int_0^{+\infty} A^n(t)\cdot \dot{\Theta}(t)\,\varphi(t)\,dt&\geq \int_0^{+\infty}B^{-,n}\,dt,
\end{align*}
for all $n\in \mathbb{N}$ and any non-negative $\varphi\in C_c(\mathbb{R}^+)$. Hence, the result follows from the fundamental lemma of calculus of variations along with the Castaing representations \eqref{E-43}-\eqref{E-44} and the H-representation \eqref{E-42}.
\end{proof}

\subsection{Comparison with previous results about singular weighted systems}\label{comparison-singular-CS-subsection}
In the previous parts, we studied the existence and one-sided uniqueness for the singular weighted first order Kuramoto model in all the subcritical, critical and supercritical regimes. We now compare our result with previous research on the singular weighted Cucker--Smale model which is a second order system describing the flocking behavior of interacting particles. In order to set these relations, let us recall Section \ref{preliminaries-section}, where the first order Kuramoto model \eqref{KM3} was shown to be equivalent to its second order augmentation \eqref{KM4}. On the one hand, this is clear for regular weights as studied in Theorem \ref{T-equiv}, see \cite{H-K-P-Z,H-P-Z}. What is more, it remains true in our case, which is characterized by singular weights. However, we must be specially careful with the time regularity in order for such heuristic arguments to become true. Let us focus on the subcritical regime, where the rigorous equivalence between \eqref{KM3} and \eqref{KM4} follows from Remark \ref{R-wp-3} by virtue of the one-sided uniqueness in both models. Indeed, in such subcritical case, the ``influence function'' of the augmented flocking-type model reads
\begin{equation}\label{E-10}
h^\prime(\theta) = \frac{1}{| \theta|_o^{2\alpha}}\left[\cos\theta-2\alpha\frac{\sin| \theta|_o}{| \theta|_o}\right]\sim \frac{1-2\alpha}{| \theta|^{2\alpha}_o} \quad \text{near} \quad \theta\in 2\pi\mathbb{Z},
\end{equation}
which enjoys mild singularities of order $2\alpha<1$ in the subcritical case. Such singular second order model \eqref{KM4}-\eqref{E-10} shares some similarities with the Cucker--Smale model with singular weights,
\begin{equation}\label{C-S-S}
\begin{dcases}
\ \dot x_i = v_i,\\
\ \dot v_i = \frac{K}{N} \sum_{j=1}^N \psi(|x_j - x_i|) (v_j - v_i),
\end{dcases}
\end{equation}
where the communication weight $\psi$ is given by
\begin{equation}\label{E-singularweight}
\psi(r):=\frac{1}{r^\beta},
\end{equation}
for $r>0$ and $\beta > 0$. Although some results regarding the asymptotic behavior of such system have been established \cite{H-L}, the well--posedness theory has not been addressed until very recently in \cite{P1,P2} for the microscopic model and \cite{C-C-H,M-P,P-S,S-T} for some first and second order kinetic and macroscopic versions of the model. Regarding the microscopic system \eqref{C-S-S}-\eqref{E-singularweight}, the existence of global $C^1$ piece-wise weak $W^{2,1}$ solutions $(x_1,\ldots,x_N)$ has been established in \cite{P1} for $\beta\in (0,1)$, which corresponds to $\alpha\in (0,\frac{1}{2})$ in our setting (see Theorem \ref{T-wp-1}, Theorem \ref{singular-limit-subcritical-theorem} and Remark \ref{singular-limit-subcritical-remark}). Also, in the weakly singular regime $\beta\in (0,\frac{1}{2})$ (i.e., $\alpha\in (0,\frac{1}{4})$), the same author proved in \cite{P2} that the velocities $(v_1,\ldots,v_N)$ are indeed absolutely continuous. Consequently, the $C^1$ weak solutions $(x_1,\ldots,x_N)$ are actually $W^{2,1}_{loc}$ in such latter case. This latter property was proved through a differential inequality.

The method of proof is similar to ours in Section \ref{singular-limit-section} and relies on a regularization process of the second order model near the collision times. In our case, we have obtained a similar regularization process of the first order model, entailing the corresponding regularization of the augmented second order model. Indeed, such method has not only proved succeed in our subcritical case, but also in the critical and supercritical case. Also, we have obtained the well-posedness results in an alternative way based on the gain of continuity of the kernel in the first order model along with its particular structure near the points of loss of Lipschitz-continuity. Indeed, we have succeeded in introducing an analogue well--posedness theory in Filippov sense for the endpoint case $\alpha=\frac{1}{2}$ and the supercritical case $\alpha>\frac{1}{2}$.

Regarding the more singular cases $\beta\geq 1$ (i.e., $\alpha\geq \frac{1}{2}$), one can show that there exists some class of initial data for \eqref{C-S-S}-\eqref{E-singularweight} such that one can avoid collisions and the solutions remain smooth for all times. Indeed, such solutions exhibit asymptotic flocking dynamics, see \cite{A-C-H-L}. Very recently, it was shown in \cite{C-C-M-P} that the loss of integrability of the kernel when $\beta\geq 1$ actually ensures the avoidance of collisions for general initial data. In such regime, the asymptotic flocking behavior is not guaranteed for any initial data. However, such ideas for \eqref{C-S-S}-\eqref{E-singularweight} fails in our model \eqref{KM4}-\eqref{E-10} because the kernel $h'$ with $\alpha\geq \frac{1}{2}$ does no longer behave like the communication weight $\psi$ with $\beta\geq 1$. Specifically, $\psi$ is always a positive and decreasing function whereas $h'$ is negative and increasing (see Figure \ref{fig:kuramoto-CS-075}). Then, we do expect our solutions to exhibit finite time collisions as depicted in the results in next Section \ref{synchro-singular-section}. This is the reason for the generalized theory in Filippov sense to come into play in the critical and supercritical cases.

\begin{figure}[h]
\centering
\begin{subfigure}[b]{0.45\textwidth}
\begin{tikzpicture}
\begin{axis}[
  axis x line=middle, axis y line=middle,
  xmin=-pi, xmax=pi, xtick={-3,...,3}, xlabel=$\theta$,
  ymin=-3, ymax=3, ytick={-3,-2,-1,0,1,2,3},
]
\addplot [
    domain=-pi:pi, 
    samples=200, 
    color=black,
]
{(1/pow(abs(x),2*0.75))*(cos(deg(x))-2*0.75*sin(deg(abs(x)))/abs(x))};
\addlegendentry{$h'(\theta)$}
\end{axis}
\end{tikzpicture}
\caption{$h'(\theta)$}
\label{fig:kuramoto-CS-075-a}
\end{subfigure}
\begin{subfigure}[b]{0.45\textwidth}
\begin{tikzpicture}
\begin{axis}[
  axis x line=middle, axis y line=middle,
  xmin=-pi, xmax=pi, xtick={-3,...,3}, xlabel=$\theta$,
  ymin=-3, ymax=3, ytick={-3,-2,-1,0,1,2,3},
]
\addplot [
    domain=-pi:pi, 
    samples=200, 
    color=black,
]
{1/pow(abs(x),2*0.75)};
\addlegendentry{$\psi(\theta)$}
\end{axis}
\end{tikzpicture}
\caption{$\psi(\theta)$}
\label{fig:kuramoto-CS-075-b}
\end{subfigure}
\caption{Comparison of the functions $h'(\theta)$ and $\psi(\theta)$ with $\alpha=0.75$.}\label{fig:kuramoto-CS-075}
\end{figure}
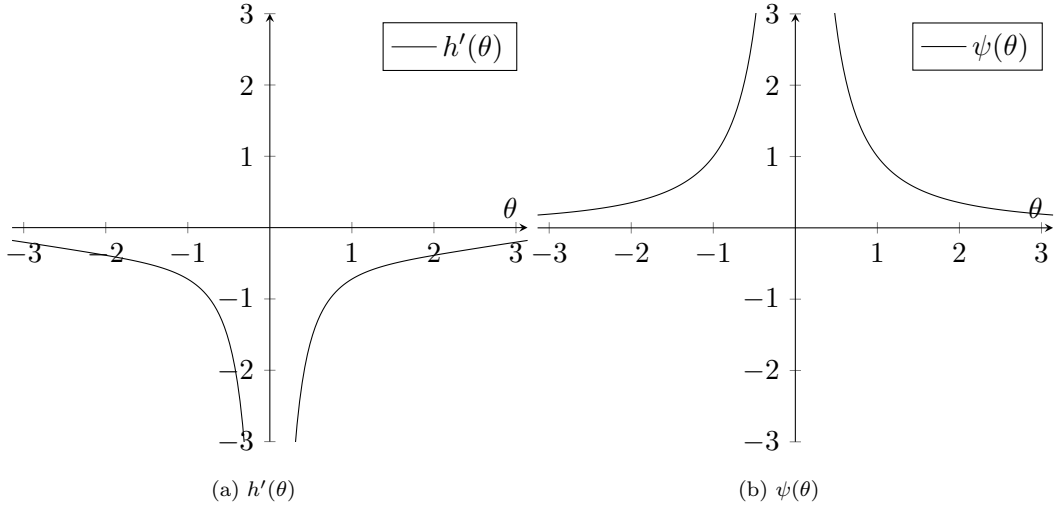

\section{Synchronization of the singular weighted system}\label{synchro-singular-section}
\setcounter{equation}{0}
We now analyze the collective behavior in the system \eqref{E-2}. We first consider the system of two interacting oscillators. We extend the argument to the N-oscillator system in succession.

\subsection{Two oscillator case} In this part, we consider the dynamics of two oscillator. The system \eqref{E-1} for two oscillator becomes
\begin{equation}
\begin{aligned}\label{D-13}
\dot\theta_1 &= \Omega_1 + \frac{K}{2}\frac{\sin(\theta_2 - \theta_1)}{|\theta_2 - \theta_1|_o^{2\alpha}} , \\
\dot\theta_2 &= \Omega_2 + \frac{K}{2}\frac{\sin(\theta_1 - \theta_2)}{|\theta_1 - \theta_2|_o^{2\alpha}} .
\end{aligned}
\end{equation}
Recall that in the critical and supercritical cases we do expect collisions, see Subsections \ref{well-posedness-critical-subsection} an \ref{well-posedness-supercritical-subsection}. Then, the above representation of the system is only valid before the first collision. After that, the right-hand side has to be replaced with the corresponding Filippov set-valued map. At this step, we shall focus on the dynamics before the first collision.
Let us define the relative phase and natural frequency by $\theta := \theta_2 - \theta_1$ and $\Omega := \Omega_2 - \Omega_1$. Then, the system \eqref{D-13} can be rewritten into the following form:
\begin{equation}\label{D-14}
\dot\theta = \Omega - K\frac{\sin\theta}{|\theta|_o^{2\alpha}}.
\end{equation}

\begin{proposition}\label{P-2-id}
Let $\theta:[0,T)\rightarrow \mathbb{R}$ be a maximal classical solution to the differential equation \eqref{D-14} with $\alpha\in (0,1)$ such that the oscillators are identical, i.e., $\Omega=0$, and initial datum $0<\vert \theta_0\vert<\pi$. Then, the maximal time of existence $T$ lies in the interval $[t_{min},t_{max}]$, where
\[
t_{min}=\frac{\vert \theta_0\vert^{2\alpha}}{2K\alpha}\ \mbox{ and }\ t_{max}=\frac{\vert \theta_0\vert^{2\alpha+1}}{2K\alpha\sin\vert \theta_0\vert}.
\]
In addition, the following lower and upper estimates 
\[
\vert \theta_0\vert^{2\alpha}-2K\alpha t\leq \vert \theta\vert^{2\alpha}\leq \vert \theta_0\vert^{2\alpha}-2K\alpha t\frac{sin\vert \theta_0\vert}{\vert \theta_0\vert}t
\]
hold, for all $t\in [0,T)$ and $\lim_{t\rightarrow T}\theta(t)=0$. Hence, two identical oscillators confined to the half-circle exhibit finite-time phase synchronization.

\end{proposition}
\begin{proof}
First of all, let us note that in the identical case $\pi+2\pi\mathbb{Z}$ are equilibria of \eqref{D-14} where the interaction kernel is Lipschitz-continuous. Hence, the maximal solution $\theta$ cannot touch such values if initially started at $\theta_0$. Thereby, $\theta(t)\in (-\pi,\pi)$ for every $t\in [0,T)$ and consequently, $|\theta(t)|_o = |\theta(t)|$ for $t\in [0,T)$. Let us now multiply by $(2\alpha + 1)|\theta|^{2\alpha} \mathrm{sgn}(\theta)$ on both side to obtain
\[
\frac{d}{dt}|\theta|^{2\alpha+1} = (2\alpha+1)|\theta|^{2\alpha} \mathrm{sgn}(\theta)\frac{d}{dt}\theta = - K(2\alpha+1) \sin \theta \, \mathrm{sgn}(\theta) = -K(2\alpha + 1)\sin|\theta|.
\]
Denote $y = |\theta|^{2\alpha + 1}$, then the equation becomes
\begin{equation}\label{D-15}
\frac{d}{dt} y = - K(2\alpha + 1) \sin y^{\frac{1}{2\alpha + 1}}.
\end{equation}
We now consider upper and lower estimates for \eqref{D-15} separately.

\medskip
\noindent$\bullet$ {\it Lower estimate:} Since $|y| \geq \sin|y|$, we have
\[
\frac{d}{dt} y \geq -K(2\alpha + 1) y^{\frac{1}{2\alpha+1}}.
\]
By multiplying by $\frac{2\alpha}{2\alpha + 1} y^{-\frac{1}{2\alpha+1}}$ on both sides, we obtain
\[
\frac{d}{dt} y^{\frac{2\alpha}{2\alpha+1}} \geq -2K\alpha.
\]
This yields
\[
y^\frac{2\alpha}{2\alpha + 1} \geq y_0^\frac{2\alpha}{2\alpha + 1} - 2K\alpha t .
\]
Thus, we have a lower estimate
\[
|\theta|^{2\alpha} \geq |\theta_0|^{2\alpha} - 2K\alpha t \quad \text{for} \quad 0\leq t<T.
\]
In particular, the above lower estimate shows that 
\[T\geq\frac{\vert \theta_0\vert^{2\alpha}}{2K\alpha}\equiv t_{min}.\]

\noindent$\bullet$ {\it Upper estimate:} As long as $0\leq y < \pi^{2\alpha + 1}$, the solution $y$ is non-increasing, i.e., $\frac{d}{dt}y \leq 0$. Since the initial data $\theta_0$ satisfies $|\theta_0|<\pi$, we have $y_0 < \pi^{2\alpha + 1}$, thus $y(t) \leq y_0$, for $t>0$. Hence, we have the following inequality
\begin{equation}\label{D-16}
\sin y^\frac{1}{2\alpha+1} \geq \frac{\sin y_0^\frac{1}{2\alpha+1}}{y_0^\frac{1}{2\alpha+1}}y^\frac{1}{2\alpha+1}.
\end{equation}
Applying \eqref{D-16} to \eqref{D-15}, we find
\[
\frac{d}{dt} y \leq -K(2\alpha + 1)\frac{\sin y_0^\frac{1}{2\alpha+1}}{y_0^\frac{1}{2\alpha+1}}y^\frac{1}{2\alpha+1}.
\]
Multiplying by $\frac{2\alpha}{2\alpha+1}y^{-\frac{1}{2\alpha+1}} $ on both sides, we obtain
\[
\frac{d}{dt} y^\frac{2\alpha}{2\alpha+1} \leq -2K\alpha\frac{\sin y_0^\frac{1}{2\alpha+1}}{y_0^\frac{1}{2\alpha+1}},
\]
which yields
\[
y^\frac{2\alpha}{2\alpha+1} \leq y_0^\frac{2\alpha}{2\alpha+1}  -2K\alpha\frac{\sin y_0^\frac{1}{2\alpha+1}}{y_0^\frac{1}{2\alpha+1}} t.
\]
This is equivalent to
\[
|\theta|^{2\alpha} \leq |\theta_0|^{2\alpha} - 2K\alpha \frac{\sin |\theta_0|}{|\theta_0|}t \quad \text{for} \quad 0\leq t< T.
\]
Again, the upper estimate shows that 
\[
T\leq\frac{|\theta_0|^{2\alpha+1}}{2K\alpha \sin|\theta_0|}\equiv t_{max}.
\]
\end{proof}

Assume that the oscillators are non-identical $\Omega = \Omega_2 - \Omega_1 > 0$ and the system \eqref{D-13} has a phase-locked state $(\bar\theta_1, \bar\theta_2)$ satisfying $0< \bar\theta_2 - \bar\theta_1<\pi$. Then, the equation \eqref{D-14} has an equilibrium $\bar\theta = \bar\theta_2 - \bar\theta_1 \in (0, \pi) $ such that
\begin{equation}\label{D-17}
\Omega - K\frac{\sin\bar\theta}{|\bar\theta|_o^{2\alpha}} = 0.
\end{equation}
To guarantee the existence of such equilibrium, we need the following conditions for the coupling strength $K$:
\begin{align*}
&\text{if} \quad \alpha < \frac{1}{2}, \quad \text{choose} \quad  K\geq \frac{\Omega}{\overline{h}}, \\
&\text{if} \quad \alpha = \frac{1}{2}, \quad \text{choose} \quad  K >  \Omega ,
\end{align*}
where $\overline{h} := \max_{0<r <\pi} h(r)$. Note that the equilibrium exists for the case of $\alpha > \frac{1}{2}$ without any condition on the coupling $K>0$.
We now investigate the stabilities of the equilibria in each cases.

\begin{proposition}\label{P-2-stability}
Let $\theta$ be a solution of \eqref{D-14}. We have the following stability results.
\begin{enumerate}
\item For $\alpha \geq \frac{1}{2}$, the equilibrium $\bar\theta$ is unstable. Furthermore, if the initial datum $\theta_0$ satisfies
\[
\theta_0 \neq 0 \quad\text{and}\quad \theta_0 \neq \bar\theta,
\]
then the solution $\theta$ reaches $0$ or $2\pi$ in finite time.
\item For $\alpha < \frac{1}{2}$, there are a stable equilibrium $\bar\theta \in (0, \tilde \theta)$ and an unstable equilibrium $\bar\theta^*\in (\tilde\theta, \pi)$, where $\tilde\theta \in (0, \frac{\pi}{2})$ is the solution to $\tilde\theta = 2\alpha\tan\tilde\theta$. Moreover, if the initial datum $\theta_0$ is located in $(-2\pi + \bar\theta^*, \bar\theta^*)$, the solution $\theta$ converges to $\bar\theta$ asymptotically.
\end{enumerate}
\end{proposition}
\begin{proof}
We linearize the equation \eqref{D-14} near $\bar\theta$ as
\[
\dot\theta = -k h^\prime(\bar\theta) (\theta- \bar\theta) + R(\bar\theta).
\]
When $\alpha \geq \frac{1}{2}$, we have $h^\prime(\bar\theta) < 0$, for $\theta \in (0, \pi)$. Thus, the equilibrium $\bar\theta$ is unstable. For $\alpha < \frac{1}{2}$, if the equilibrium $\bar\theta$ is located in $(0, \tilde\theta)$, we have $h^\prime(\bar\theta) > 0$, i.e., it is stable. By similar argument, due to $h^\prime(\bar\theta^*) < 0$, the equilibrium $\bar\theta^*$ located in $(\tilde\theta, \pi)$ is unstable. We now investigate the convergence of the solution.

\medskip

\noindent (1) For the case of $\alpha \geq \frac{1}{2}$ we consider two possibilities:

$\bullet$ Case 1 ($\theta_0 >\bar\theta$):
Since the function $h$ is decreasing in $(0, 2\pi)$, we have $h(\theta) < h(\bar\theta)$, for $\theta \in (\bar\theta, 2\pi)$. Thus, we find
\[
\dot \theta  = \Omega - Kh(\theta) > \Omega - Kh(\bar\theta) = 0, \quad \text{for} \quad \theta \in (\bar\theta, 2\pi).
\]
Moreover, due to the monotonic increase of $\theta$, we obtain the lower estimate for the frequency:
\[
\dot \theta  = \Omega - Kh(\theta) > \Omega - Kh(\theta_0) > 0, \quad \text{for} \quad \theta \in (\bar\theta, 2\pi).
\]
Hence, there exists a finite time $t_1 < \frac{2\pi - \theta_0}{\Omega - Kh(\theta_0)}$, for which the solution converges to $2\pi$.

$\bullet$ Case 2 ($\theta_0 < \bar\theta$):
We can apply an analogous argument for this case. Since the function $h$ is decreasing, we deduce $h(\theta) > h(\bar\theta)$ for $\theta \in (0, \bar\theta)$. Thus, we have
\[
\dot\theta = \Omega - Kh(\theta) < \Omega - Kh(\bar\theta) = 0, \quad \text{for} \quad \theta \in (0, \bar\theta).
\]
This monotonic decrease of phase yields the upper estimate for the frequency:
\[
\dot\theta = \Omega - Kh(\theta) < \Omega - Kh(\theta_0) < 0, \quad \text{for} \quad \theta \in (0, \bar\theta).
\]
So, there exists a finite time $t_2 < \frac{\theta_0}{|\Omega - Kh(\theta_0)|}$, for which the solution converge to zero.

\medskip
\noindent (2) For the case of $\alpha < \frac{1}{2}$, we consider two steps for the aymptotic convergence to the equilibrium:

$\bullet$ {\it Step 1:} We first show the solution moves into the interval $(0, \tilde\theta)$ in finite time when the initial datum $\theta_0$ is located in $(-2\pi + \bar\theta^*, 0] \cup [\tilde\theta, \bar\theta^*)$. As long as the solution $\theta$ located in $[\tilde\theta, \bar\theta^*)$, we have $h(\theta) > h(\bar\theta)$. Thus, the solution is non-increasing:
\[
\dot\theta = \Omega - Kh(\theta) < \Omega - Kh(\bar\theta) = 0, \quad \text{for} \quad \theta \in [\tilde\theta, \bar\theta^*).
\]
Moreover, the non-increase of solution $\theta(t) \leq \theta_0$ gives an upper bound of frequency:
\[
\dot\theta = \Omega - Kh(\theta) < \Omega - Kh(\theta_0) < 0,
\]
while $\theta$ is in $[\tilde\theta, \bar\theta^*)$. So, there exists a finite time $t_3 := \frac{\theta_0 - \tilde\theta}{|\Omega- Kh(\theta_0)|}$ such that the solution verifies $\theta(t) < \tilde\theta$ for $t>t_3$. Analogously, if the initial datum $\theta_0$ is given in $(-2\pi + \bar\theta^*, 0]$, then we have $h(\theta) < h(\bar\theta)$, the solution is non-decreasing:
\[
\dot\theta = \Omega - Kh(\theta) > \Omega - Kh(\bar\theta) = 0,
\]
and the frequency has a lower bound
\[
\dot\theta = \Omega - Kh(\theta) > \Omega - Kh(\theta_0) > 0,
\]
as long as $\theta\in(-2\pi + \bar\theta^*, 0]$. Thus, there exists a finite time $t_4:=\frac{|\theta_0|}{|\Omega - Kh(\theta_0)|}$ such that the solution verifies $\theta(t) > 0$, for $t>t_4$.

$\bullet$ {\it Step 2:} We will show that the solution converges to the stable equilibrium $\bar\theta$ asymptotically, when the initial datum is in $(0, \tilde\theta)$. Suppose the initial data is located in $(0, \bar\theta)$. Then, the following inequality 
\[
\frac{h(\bar\theta)}{\bar\theta}\theta < h(\theta) < h^\prime(\bar\theta)(\theta - \bar\theta) + h(\bar\theta),
\]
holds for the function $h$. Thus, the solution satisfies the differential inequality
\[
\Omega - K\big(h^\prime(\bar\theta)(\theta - \bar\theta) + h(\bar\theta) \big) < \dot\theta < \Omega - \frac{Kh(\bar\theta)}{\bar\theta}\theta .
\]
By Gr\"{o}nwall's lemma, we obtain
\[
\bar\theta  -(\bar\theta - \theta_0) e^{-kh^\prime(\bar\theta) t} < \theta(t) < \bar\theta - (\bar\theta - \theta_0)e^{-\frac{Kh(\bar\theta)}{\bar\theta}t}.
\]
Similarly, if the initial datum $\theta_0$ is in $(\bar\theta, \tilde\theta)$, the function $h$ satisfies
\[
\frac{h(\tilde\theta) - h(\bar\theta)}{\tilde\theta - \bar\theta}(\theta - \bar\theta) + h(\bar\theta) <  h(\theta)  < h^\prime(\bar\theta)(\theta - \bar\theta) + h(\bar\theta).
\]
Then, we have the following differential inequality:
\[
\Omega - K\big(h^\prime(\bar\theta)(\theta - \bar\theta) + h(\bar\theta) \big) < \dot\theta < \Omega - K\bigg( \frac{h(\tilde\theta) - h(\bar\theta)}{\tilde\theta - \bar\theta}(\theta - \bar\theta) + h(\bar\theta) \bigg).
\]
Hence, by Gr\"{o}nwall's lemma, we find
\[
\bar\theta - (\theta_0 - \bar\theta) e^{-Kh^\prime(\bar\theta)t}< \theta(t) < \bar\theta - (\theta_0 - \bar\theta)e^{-K\frac{h(\tilde\theta) - h(\bar\theta)}{\tilde\theta - \bar\theta}t}.
\]
\end{proof}

\begin{remark}\label{R-summary-synchro-2-oscillators}
In the subcritical case $\alpha\in\left(0,\frac{1}{2}\right)$, the emergence of phase-locked state for two non-identical oscillators occurs asymptotically (see Proposition \ref{P-2-stability}), whereas the phase synchronization for two identical oscillators appears in finite time (see Proposition \ref{P-2-id}). However, in the critical and supercritical cases $\alpha\in \left[\frac{1}{2},1\right)$, phase synchronization always appears in finite time as depicted in the above-mentioned Propositions \ref{P-2-stability} and \ref{P-2-id} as long as the initial phase configuration does not agree with the unstable phase-locked state $\overline{\theta}$. Namely, in the supercritical case both oscillators stick together into a unique cluster moving at constant frequency $\widehat{\Omega}=\frac{\Omega_1+\Omega_2}{2}$, independently on the chosen natural frequencies. However, in the critical case, the same only happens under the assumption $\vert \Omega_1-\Omega_2\vert\leq K$. In other case, the formed cluster will instantaneously split.
\end{remark}

\subsection{$N$-oscillator case}

In this subsection, we consider the system of $N$ interacting oscillators. We will first focus on the dynamics in the simpler subcritical case $\alpha\in \left(0,\frac{1}{2}\right)$, where solutions have proved to be classical, see Theorem \ref{T-wp-1}. The reason to start with this case is that the right hand side of \eqref{E-2} can be considered in the single-valued sense for that case. The dynamics in the critical case $\alpha=\frac{1}{2}$ and some intuition about the dynamics in the supercritical regime $\alpha\in (\frac{1}{2},1)$ will be provided at the end of this Subsection.

Let $\Theta = (\theta_1, \ldots, \theta_N)$ be the solution to the system \eqref{E-2}. We first study the phase synchronization for identical oscillators. Fist, let us set the indices $M$ and $m$ to satisfy
\begin{equation}\label{E-phase-M-m}
\theta_M(t):=\max\{\theta_1(t),\ldots,\theta_N(t)\}\ \mbox{ and }\ \theta_m(t):=\min\{\theta_1(t),\ldots,\theta_N(t)\},
\end{equation}
for each time $t\geq 0$. Then, we can define the diameter of phase to be
\begin{equation}\label{E-phase-diameter}
D(\Theta):=\theta_M-\theta_m .
\end{equation}

\begin{theorem}\label{T-N-id}
Let $\Theta=(\theta_1, \cdots, \theta_N)$ be the solution to \eqref{E-2} with $\alpha\in \left(0,\frac{1}{2}\right)$ for identical oscillators $(\Omega_i = 0)$, for $i=1, \ldots, N$. Assume that the initial configuration $\Theta_0$ is confined in a half circle, i.e., $0<D(\Theta_0) < \pi$. Then, there is complete phase synchronization at a finite time not larger than $T_c$ where
$$T_c=\frac{D(\theta_0)^{1-2\alpha}}{2\alpha Kh(D(\Theta_0))}.$$ 
\end{theorem}
\begin{proof}
We consider the dynamics of phase diameter:
\[
\frac{d}{dt} D(\Theta) = \frac{K}{N} \sum_{j=1}^N \Big( h(\theta_j - \theta_M) - h(\theta_j - \theta_m) \Big).
\]
Since $h(\theta_j - \theta_M)<0$ and $h(\theta_j - \theta_m)>0$ as long as $D(\Theta) < \pi $, we have
\[
\frac{d}{dt} D(\Theta) \leq 0 \quad \text{and} \quad D(\Theta(t)) \leq D(\Theta_0) < \pi , \quad \text{for} \quad t>0.
\]
Due to the contraction of phase, and the fact that $\theta\in (0,\pi)\mapsto\frac{h(\theta)}{\theta}$ is decreasing, we have
\[
h(\theta_j - \theta_M) \leq  \frac{h\big(D(\Theta_0)\big)}{D(\Theta_0)}(\theta_j - \theta_M) \quad \text{and} \quad h(\theta_j - \theta_m) \geq  \frac{h\big(D(\Theta_0)\big)}{D(\Theta_0)}(\theta_j - \theta_m) .
\]
Thus, we attain the following differential inequality:
\begin{align*}
\frac{d}{dt} D(\Theta) &\leq \frac{K}{N} \sum_{j=1}^N \Big( \frac{h\big(D(\Theta_0)\big)}{D(\Theta_0)}(\theta_j - \theta_M) - \frac{h\big(D(\Theta_0)\big)}{D(\Theta_0)}(\theta_j - \theta_m) \Big) \\
&=  \frac{K}{N}\frac{h\big(D(\Theta_0)\big)}{D(\Theta_0)} \sum_{j=1}^N \Big( (\theta_j - \theta_M) - (\theta_j - \theta_m) \Big)\\
& = -K \frac{h\big(D(\Theta_0)\big)}{D(\Theta_0)} D(\Theta).
\end{align*}
By Gr\"{o}nwall's lemma, we obtain
\[
D(\Theta) \leq D(\Theta_0) e^{-K\frac{h(D(\Theta_0))}{D(\Theta_0)} t} \quad \text{for} \quad t \geq 0.
\]
Notice that $h(\theta)$ behaves like $\theta^{1-2\alpha}$ near the origin. Indeed, it is easy to prove that for every $\theta_*\in (0,\pi)$
\[
h(\theta)\geq \frac{h(\theta_*)}{\theta_*^{1-2\alpha}}\,\theta^{1-2\alpha},\ \forall\,\theta\in [0,\theta_*].
\]
The main idea is to show that the mapping
$$\theta\longmapsto\frac{h(\theta)}{\theta^{1-2\alpha}},$$
is nonincreasing in $[0,\pi]$. Since the phase diameter $D(\Theta)$ is bounded above by $D(\Theta_0)$ we can take $\theta_*=D(\Theta_0)$ and apply the above lower estimate for $h$ to attain the following estimate of the phase diameter
\begin{align*}
\frac{d}{dt} D(\Theta) &= \frac{K}{N} \sum_{j=1}^N \Big( h(\theta_j - \theta_M) - h(\theta_j - \theta_m) \Big) \\
&\leq \frac{K}{N} \frac{h(D(\Theta_0))}{D(\Theta_0)}\sum_{j=1}^N \Big( - (\theta_M - \theta_j)^{1-2\alpha} - (\theta_j - \theta_m)^{1-2\alpha}\Big)\\
&\leq -\frac{K}{N}\frac{h(D(\Theta_0))}{D(\Theta_0)} \sum_{j=1}^N ((\theta_M-\theta_j)+(\theta_j-\theta_m))^{1-2\alpha}\\
&=-\frac{K h(D(\Theta_0))}{D(\Theta_0)} D(\Theta)^{1-2\alpha},
\end{align*}
for every $t \geq 0$. In the last inequality we have used that $1-2\alpha\in (0,1)$ and, consequently,
\[
(a+b)^{1-2\alpha}\leq a^{1-2\alpha}+b^{1-2\alpha},
\]
for every couple of nonnegative numbers $a,b\in\mathbb{R}$. Then, integrating the above differential inequality implies
\[
D(\Theta(t))\leq \left(D(\Theta_0)^{2\alpha}-2\alpha K\frac{D(\Theta_0)}{D(\Theta_0)}t\right)^{\frac{1}{2\alpha}},
\]
for all $t\geq 0$. This implies the convergence to zero at a finite time not larger than $T_c$.
\end{proof}

We now consider the system for non-identical oscillators. The next proposition yields the structure of phase-locked state of \eqref{E-2} for non-identical oscillators with mutually distinct natural frequencies in the subcritical regime.

\begin{proposition}\label{P-N-order}
Let $\alpha \in\left(0,\frac{1}{2}\right)$ and $\bar\Theta = (\bar\theta_1, \cdots, \bar\theta_N)$ be an equilibrium of the system \eqref{E-2} such that $\max_{i,j} |\bar\theta_i - \bar\theta_j| < \tilde \theta$ where $\tilde \theta \in (0, \frac{\pi}{2})$ is the solution to $\tilde\theta = 2\alpha \tan \tilde\theta$. Assume the natural frequencies satisfy the ordering $\Omega_1 < \ldots < \Omega_N$. Then, the phase-locked state $\bar\Theta$ verifies the ordering $\bar\theta_1 < \ldots < \bar\theta_N$.
\end{proposition}

\begin{proof}
First, we show that the equilibria $\bar\theta_i$'s are mutually distinct, i.e.,
\[
\bar\theta_i \neq \bar\theta_j \quad \text{for} \quad i\neq j.
\]
Since $\bar\Theta$ is an equilibrium, it satisfies
\begin{equation}\label{D-17-0}
\Omega_i + \frac{K}{N} \sum_{k\ne i} h(\bar\theta_k - \bar\theta_i) = 0,
\end{equation}
for every $i=1,\ldots,N$. If there existed two oscillators having the same equilibria $\bar\theta_i = \bar\theta_j$, then we would have
\[
\frac{K}{N} \sum_{k\ne i} h(\bar\theta_k - \bar\theta_i) = \frac{K}{N} \sum_{k\ne j} h(\bar\theta_k - \bar\theta_j),
\]
which contradicts with $\Omega_i \neq \Omega_j$. We now show the ordering property. From \eqref{D-17-0}, we have
\begin{align*}
\Omega_{i+1} - \Omega_i &= - \frac{K}{N} \sum_{j\ne i+1} h(\bar\theta_{j} - \bar\theta_{i+1}) + \frac{K}{N} \sum_{j\ne i} h(\bar\theta_j - \bar\theta_i)\\
& = \frac{K}{N} \Big(\sum_{j\ne i, i+1} h(\bar\theta_{i+1} - \bar\theta_j) - h(\bar\theta_i - \bar\theta_j)\Big)  - \frac{K}{N} \big(h(\bar\theta_i - \bar\theta_{i+1}) - h(\bar\theta_{i+1} - \bar\theta_i)\big) \\
& = \frac{K}{N} \Big(\sum_{j\ne i, i+1} h(\bar\theta_{i+1} - \bar\theta_j) - h(\bar\theta_i - \bar\theta_j)\Big)  + \frac{2K}{N}  h(\bar\theta_{i+1} - \bar\theta_i)\\
&=\frac{K}{N}\sum_{j\ne i, i+1} c_j(\bar \theta_{i+1}-\bar\theta_i)  + \frac{2K}{N}  h(\bar\theta_{i+1} - \bar\theta_i),
\end{align*}
where the coefficients $c_j$ read
\[c_j:=\frac{h(\bar\theta_{i+1}-\bar\theta_j)-h(\bar\theta_i-\bar\theta_j)}{\bar\theta_{i+1}-\bar\theta_i}.\]
They are properly defined because all the equilibria are mutually distinct and they are positive because $h$ is strictly increasing in $(-\tilde \theta,\tilde \theta)$. Thus, the order $\Omega_{i+1} > \Omega_i$ yields the order of equilibria $\bar\theta_{i+1} > \bar\theta_i$.
\end{proof}

In the subcritical case, we can attain the uniform boundedness of phase differences under sufficiently large coupling strength.

\begin{lemma}\label{L-N-bd}
Let $\Theta$ be the solution to \eqref{E-2} for $\alpha\in\left(0,\frac{1}{2}\right)$ and non-identical oscillators with initial data $\Theta_0$,  satisfying $D(\Theta_0) < D^\infty < \tilde\theta$. If the coupling strength $K$ is sufficiently large such that
\[
K > \frac{D(\dot\Theta_0)}{h^\prime(D^\infty)(D^\infty - D(\Theta_0))},
\]
then, the phase diameter $D(\Theta)$ is uniformly bounded by $D^\infty$:
\[
D(\Theta(t)) < D^\infty , \quad \text{for} \quad t\geq 0.
\]
\end{lemma}

\begin{proof}
Suppose there exists a finite time $t^* >0$ such that
\[
t^*:= \sup\{t: D(\Theta(s)) < D^\infty \quad \text{for} \quad 0\leq s\leq t\} \quad \text{and} \quad D(\Theta(t^*)) = D^\infty.
\]
We set indices $F$ and $S$ so that
\[
\dot\theta_F:= \max\{\dot{\theta}_1, \ldots, \dot{\theta}_N\} \quad \text{and} \quad \dot\theta_S:=\min\{\dot{\theta}_1, \ldots, \dot{\theta}_N\},
\]
for each time $t$. We define the frequency difference so that
\[
D(\dot\Theta(t)):= \dot\theta_F - \dot\theta_S.
\]
We note that
\begin{equation}\label{D-17-1}
D(\dot\Theta(t)) - D(\dot\Theta_0) = \int_0^t \frac{d}{dt} D(\dot\Theta(s)) \, ds. 
\end{equation}
By taking time derivative on $D(\dot\Theta)$, we obtain
\[
\frac{d}{dt} D(\dot\Theta) = \frac{K}{N} \sum_{j=1}^N \Big( h^\prime(\theta_j - \theta_F) (\dot\theta_j - \dot\theta_F) - h^\prime(\theta_j - \theta_S)(\dot\theta_j - \dot\theta_m)\Big).
\]
As long as $D(\Theta) < D^\infty$, we have
\[
h^\prime(\theta_j - \theta_i) \geq h^\prime(D^\infty) > 0.
\]
Thus, we get
\begin{equation}\label{D-17-2}
\frac{d}{dt}D(\dot\Theta) \leq \frac{K}{N} \sum_{j=1}^N h^\prime(D^\infty)\Big((\dot\theta_j - \dot\theta_F) -(\dot\theta_j - \dot\theta_S)\Big) = -Kh^\prime(D^\infty) D(\dot\Theta).
\end{equation}
We combine \eqref{D-17-1} and \eqref{D-17-2} to obtain
\begin{equation}\label{D-17-3}
D(\dot\Theta(t)) \leq D(\dot\Theta_0) - Kh^\prime(D^\infty) \int_0^t D(\dot\Theta(s)) \, ds.
\end{equation}
Setting $y(s):= \int_0^t D(\dot\Theta(s))\, ds$, we can rewrite \eqref{D-17-3} into
\[
y^\prime(t) \leq y^\prime(0) - Kh^\prime(D^\infty) y(t).
\]
Hence, we have
\[
y(t) \leq \frac{y^\prime(0)}{Kh^\prime(D^\infty)} (1- e^{-K h^\prime(D^\infty) t}) \leq \frac{y^\prime(0)}{Kh^\prime(D^\infty)} .
\]
Since $D(\Theta(t^*)) = D^\infty$ and $K>\frac{D(\dot\Theta_0)}{h^\prime(D^\infty)(D^\infty - D(\Theta_0))}$, we get
\begin{align*}
D^\infty &= D(\Theta_0) + \int_0^{t^*} \frac{d}{ds} D(\Theta(s)) \, ds\\
&\leq D(\Theta_0) + \int_0^{t^*} D(\dot\Theta(s)) \, ds\\
&\leq D(\Theta_0) + \frac{D(\dot\Theta_0)}{Kh^\prime(D^\infty)} < D^\infty,
\end{align*}
which is a contradiction. Thus, we have the desired uniform bound for phase difference
\[
D(\Theta(t)) < D^\infty , \quad \text{for} \quad t\geq 0.
\]
\end{proof}

\begin{remark}
Note that, in the preceding proof, the solution $\Theta=\Theta(t)$ is $C^1$ but not necessarily $C^2$ because of the essential discontinuity of $h'$. Then, one cannot directly argue with two time derivatives in the computation of $\frac{d}{dt}D(\dot{\Theta})$. However, the preceding arguments can be made rigorous because the $C^1$ solution of \eqref{E-2} is a piece-wise $W^{2,1}$ solution of the augmented model \eqref{KM4}-\eqref{E-10} as discussed in Remark \ref{singular-limit-subcritical-remark} in the preceding Section 4. Other possible approach is to directly show the Gr\"{o}wall inequality \eqref{D-17-3} in integral form.
\end{remark}

In the following result, we show the collision avoidance when the oscillators are initially well-ordered.

\begin{lemma}\label{L-N-col}
Let $\Theta$ be the solution to \eqref{E-2}, with $\alpha\in\left(0,\frac{1}{2}\right)$, and initial data $\Theta_0$ satisfying $D(\Theta_0) < D^\infty < \tilde\theta$. Assume the natural frequencies and the initial configuration satisfy the ordering $\Omega_1 < \cdots < \Omega_N$ and $\theta_{1,0} < \cdots < \theta_{N,0}$, respectively. We assume the coupling strength $K$ is sufficiently large such that 
\[K > \frac{D(\dot\Theta_0)}{h^\prime(D^\infty)(D^\infty - D(\Theta_0))}.\]
Then, there is no collision between oscillators, i.e., 
\[
\theta_i (t) \neq \theta_j(t) \quad \text{for} \quad i\neq j, \quad t>0.
\] 
\end{lemma}
\begin{proof}
From Lemma 4.6, we have an uniform bound of the phase diameter $D(\Theta(t)) < D^\infty$, for $t\geq0$.
Let $\ell$ be an index such that
\[
\theta_{\ell+1} (t)- \theta_\ell (t) = \min_{j=1, \ldots, N-1} \theta_{j+1} (t)- \theta_j (t) ,
\]
for each time $t\geq 0$. For notationally simplicity, we set $\Delta:= \theta_{\ell+1} - \theta_\ell$. Then, we have
\begin{equation}\label{D-18}
\begin{aligned}
\dot\Delta &= \Omega_{\ell+1} - \Omega_{\ell} + \frac{K}{N} \sum_{j=1}^N \Big(h(\theta_j - \theta_{\ell+1}) - h(\theta_j - \theta_{\ell}) \Big)\\
&\geq \Omega_\delta + \frac{K}{N} \sum_{j=1}^N \Big(h(\theta_j - \theta_{\ell+1}) - h(\theta_j - \theta_{\ell}) \Big) ,
\end{aligned}
\end{equation}
where $\Omega_\delta:= \min_{j=1, \ldots, N-1} \Omega_{j+1} - \Omega_j > 0$.
We define the sets of indices such that
\[
\mathcal S_1(\ell) :=\{j: j< \ell\} \quad \text{and} \quad \mathcal S_2(\ell) :=\{j: j> \ell+1 \}.
\]
Note that $h(\theta)$ is convex increasing for $\theta\in (-\tilde\theta, 0)$ and is concave increasing for $\theta \in (0, \tilde\theta)$. Thus, we have
\begin{equation}\label{D-19}
\begin{aligned}
&0< h^\prime(b) \leq \frac{h(b) - h(a)}{b-a} \leq h^\prime(a) \quad \text{for} \quad 0 \leq a < b \leq \tilde\theta, \\
&0<h^\prime(c) \leq \frac{h(d) - h(c)}{d-c} \leq h^\prime(d) \quad \text{for} \quad -\tilde\theta \leq c < d \leq 0.
\end{aligned}
\end{equation}
From \eqref{D-18} and \eqref{D-19}, we obtain
\begin{align*}
\dot\Delta &\geq \Omega_\delta +\frac{K}{N} \sum_{j\in\mathcal S_1(\ell)} \Big( h(\theta_j - \theta_{\ell+1}) - h(\theta_j - \theta_\ell)\Big) + \frac{K}{N} h(\theta_\ell - \theta_{\ell+1}) \\
& \quad - \frac{K}{N} h(\theta_{\ell+1} - \theta_\ell)  + \frac{K}{N} \sum_{j\in \mathcal S_2(\ell)} \Big( h(\theta_j - \theta_{\ell+1}) - h(\theta_j - \theta_\ell)\Big) \\
&\geq \Omega_\delta - \frac{K}{N} \sum_{j\in\mathcal S_1(\ell)} h^\prime(\theta_j - \theta_\ell) \, \Delta  - \frac{K}{N} \sum_{j\in\mathcal S_2(\ell)} h^\prime(\theta_j - \theta_{\ell+1}) \, \Delta - \frac{2K}{N}h(\Delta)\\
&\geq \Omega_\delta - \frac{K|\mathcal S_1(\ell)|}{N} h^\prime(\Delta) \,  \Delta - \frac{K|\mathcal S_2(\ell)|}{N} h^\prime(\Delta) \, \Delta - \frac{2K}{N}h(\Delta) \\
&\geq \Omega_\delta -K h^\prime(\Delta) \,  \Delta  - \frac{2K}{N}h(\Delta)\\
&\geq  \Omega_\delta -C \Delta^\gamma =: q(\Delta),
\end{align*}
where we have used
\[
h(\theta) \leq C_1 \theta^\gamma \quad \text{and} \quad h^\prime(\theta)\theta \leq C_2 \theta^\gamma ,
\]
for $\theta \geq 0$ and $0 < \gamma < 1-2\alpha$ in the last inequality. Since $\lim_{\theta\to 0+}q(\theta) = \Omega_\delta >0$ and $q(\theta)$ is continuous for $\theta> 0$, there exists a positive $\varepsilon >0$ such that $ q(\theta) > 0$, for $\theta \in (0, \varepsilon)$. Hence, the distance $\Delta$ has a positive lower bound.
\end{proof}

In the sequel, we study the stability of the phase-locked state for the system of non-identical oscillators. We use the center manifold theorem to investigate the stability of linearized system.

\begin{lemma}[Center Manifold Theorem \cite{C}]\label{L-cmt}
Consider the system
\begin{equation}\label{center-1}
\begin{aligned}
\dot x &= Ax + f_A(x,y)\\
\dot y &= By + f_B(x,y)
\end{aligned}
\end{equation}
where $x \in \mathbb R^n , y\in \mathbb R^m$ and $A$ and $B$ are constant matrices such that all the eigenvalues of $A$ have zero real parts while all the eigenvalues of $B$ have negative real parts. Assume that the functions $f_A$ and $f_B$ are $C^2$ with $f_A(0,0) = 0, \nabla f_A(0,0) = 0, f_B(0,0) = 0, \nabla f_B(0,0)=0$. Then, we have the following results:\\
\begin{enumerate}
\item There exists a center manifold for \eqref{center-1}, $y = \phi(x), |x|<\delta$, where $\phi=\phi(x)$ is $C^2$. The flow on the center manifold is governed by the $n$-dimensional system:
\begin{equation}\label{center-2}
\dot u = Au + f_A(u, \phi(u))
\end{equation} 
\item Assume the zero solution of \eqref{center-2} is stable (respectively asymptotically stable/unstable). Then, the zero solution of \eqref{center-1} is stable (respectively asymptotically stable/unstable).
\end{enumerate}
\end{lemma}

\begin{theorem}\label{T-N-sta}
Let $\bar\Theta:=(\bar\theta_1, \cdots, \bar\theta_N)\notin\mathcal{C}$ be a collision-less equilibrium of \eqref{E-2}.
\begin{enumerate}
\item If $\alpha \geq \frac{1}{2}$, then the phase-locked state $\bar\Theta$ is unstable.
\item If $\alpha < \frac{1}{2}$, then the phase-locked state $\bar\Theta$ is stable.
\end{enumerate}
\end{theorem}
\begin{proof}
\noindent $(1)$ We first linearize the system \eqref{E-2}:
\begin{equation}\label{D-20}
\dot\Theta = A(\Theta- \bar\Theta) +  R (\bar\Theta) ,
\end{equation}
where the elements of matrix $A=[a_{ij}]$ are determined by
\begin{equation}\label{D-21}
\begin{aligned}
a_{ij} &= \frac{\cos(\bar\theta_j - \bar\theta_i)}{|\bar\theta_j - \bar\theta_i|_o^{2\alpha}} - 2\alpha \frac{\sin |\bar\theta_j - \bar\theta_i|_o}{|\bar\theta_j - \bar\theta_i|_o^{2\alpha + 1}} \quad \text{for} \quad i\ne j, \\
a_{ii} &= - \sum_{j\ne i} a_{ij}.
\end{aligned}
\end{equation}
If $\alpha \geq \frac{1}{2}$, we find $a_{ij} < 0$, for $i\neq j$, and hence $a_{ii} > 0$, for $i=1, \ldots, N$. This leads the matrix $A$ is a Laplacian type matrix of which all eigenvalues are non-negative. Since the matrix $A$ represents all-to-all connected network, there exists a zero eigenvalue for which the multiplicity is one and all the other eigenvalues are positive which implies the unstability of the equilibrium.

\medskip
\noindent $(2)$ We now assume $\alpha < \frac{1}{2}$. Since the equilibrium satisfies $\max_{i,j} |\bar\theta_i - \bar\theta_j| < \tilde \theta$ and $\bar\theta_i \ne \bar\theta_j$ for $i\ne j$, the elements of the matrix have signs so that $a_{ij}>0$ for $i\neq j$ and $a_{ii}<0$, for $i=1, \ldots, N$. By similar argument as above, we can obtain that the eigenvalues of $A$ are non-positive and there is a zero eigenvalue with multiplicity 1. Let $\lambda_1 = 0$ and $\lambda_2, \ldots, \lambda_N < 0$ be the eigenvalues for matrix $A$ and let $v_1, \ldots, v_N $ be the corresponding left eigenvectors such that
\[
v_i A = \lambda_i v_i \quad \text{for} \quad i=1, \ldots, N.
\]
We note that $v_1 = (1, \cdots, 1)$. We set the matrices $P$ and $D$ so that
\[
P^{-1} := \begin{pmatrix}1 & \cdots & 1  \\ & v_2 &  \\ & \vdots & \\ & v_N & \end{pmatrix} \quad \text{and} \quad   D:= \begin{pmatrix} 0 & 0 & \cdots & 0\\ 0 & \lambda_2 & \cdots & 0 \\ \vdots & \vdots & \ddots & \vdots \\ 0 & 0 & \cdots & \lambda_N \end{pmatrix}.
\]
Then, we can diagonalize the matrix $A$:
\begin{equation}
P^{-1} A P = D .
\end{equation}
We change the variables from $\Theta = (\theta_1, \ldots, \theta_N)$ to $X = (x_1, \ldots, x_N)$ such that
\begin{equation}\label{D-22}
X := P^{-1} \Theta.
\end{equation}
Then, the system \eqref{D-20} can be transformed into the following form:
\begin{equation}\label{D-23}
\dot X = D (X - \bar X) + \tilde R(X)
\end{equation}
Let $\hat x_1 := (x_2, \ldots, x_N)$ and $\hat D$ be a minor matrix of $D$ such that
\[
\hat D:= \begin{pmatrix}   \lambda_2 & \cdots & 0 \\ \vdots & \ddots & \vdots \\ 0 & \cdots & \lambda_N \end{pmatrix}.
\]
Then, we can rewrite the system \eqref{D-23} in the following form:
\begin{equation}\label{D-24}
\begin{pmatrix} x_1 \\ \hat x_1 \end{pmatrix}^\prime = \begin{pmatrix} 0 & 0 \\ 0 & \hat D \end{pmatrix}\begin{pmatrix} x_1 - \bar x_1 \\ \hat x_1 - \hat{\bar{x}}_1 \end{pmatrix} + \begin{pmatrix} \tilde R_1(x_1, \hat x_1) \\ \hat{\tilde{R}}_1 (x_1, \hat x_1) \end{pmatrix} .
\end{equation}
Consider the center manifold in Lemma \ref{L-cmt}, that can be written as follows
\[
W_c:= \{ (x,y) \in \mathbb R \times \mathbb R^{N-1} : y = c(x) \quad \text{for} \quad |x|< \varepsilon, \quad \phi(\bar x_1) = 0, \quad D\phi(\bar x_1)=0 \} ,
\]
and consider the equation
\begin{equation}\label{D-25}
\dot x_1 = \tilde R_1 (x_1, \phi(x_1)).
\end{equation}
By the Center Manifold Theorem, the stability of \eqref{D-25} implies the stability of the system \eqref{D-24}. Since the equality \eqref{D-22} yields $x_1 = \theta_1 + \cdots + \theta_N$ and we have 
\[
\dot x_1 = \sum_{i=1}^N \dot\theta_i = \sum_{i=1}^N \Omega_i = 0.
\]
Thus, the right hand side $\tilde R_1 \equiv 0$ and the dynamics of \eqref{D-25} is stable. Therefore, the phase-locked state $\bar\Theta$ is stable for $\alpha < \frac{1}{2}$.
\end{proof}

Finally, we are ready to show the emergence of phase locked state for non-identical oscillators.

\begin{theorem}\label{T-N-pl}
Let $\Theta$ be a solution to \eqref{E-2} with initial data $\Theta_0$ satisfying $D(\Theta_0) < D^\infty < \tilde \theta$ for $\alpha\in\left(0,\frac{1}{2}\right)$.  If the coupling strength is sufficiently large such that 
\[K>\frac{D(\dot\Theta_0)}{h^\prime(D^\infty)(D^\infty - D(\Theta_0))},\]
then we can show the emergence of phase-locked state. Moreover, if each oscillator has distinct natural frequency, i.e., $\Omega_i \neq \Omega_j$ for $i\neq j$, then, the synchronization occurs asymptotically.
\end{theorem}
\begin{proof}
By applying Gronwall's lemma on \eqref{D-17-2}, we have an exponential decay of upper estimate on the frequency diameter:
\[
D(\dot\Theta(t)) \leq D(\dot\Theta_0) e^{-Kh^\prime(D^\infty) t}.
\]
This exponential decay implies the emergence of phase-locked state.

Assume the oscillators have mutually distinct natural frequencies. Since Proposition 4.6 gives the structure of phase-locked state, the oscillators draw in descending order of natural frequencies in finite time. After this time, by Lemma 4.7, we have a positive lower bound $\varepsilon_\Delta > 0$ of distance between oscillators. Then, we have
\begin{align*}
\frac{d}{dt} D(\dot\Theta) &= \frac{K}{N} \sum_{j=1}^N \Big( h^\prime(\theta_j - \theta_F)(\dot\theta_j - \dot\theta_F) - h^\prime(\theta_j - \theta_S)(\dot\theta_j - \dot\theta_S)\Big)\\
&\geq \frac{K}{N} \sum_{j=1}^N \Big( h^\prime(\varepsilon_\Delta)(\dot\theta_j - \dot\theta_F) - h^\prime(\varepsilon_\Delta)(\dot\theta_j - \dot\theta_S)\Big)\\
&= -Kh^\prime(\varepsilon_\Delta)D(\dot\Theta).
\end{align*}
By Gr\"{o}nwall's lemma, we have an lower estimate on the frequency diameter:
\[
D(\dot\Theta(t))\geq D(\dot\Theta_0) e^{-Kh^\prime(\varepsilon_\Delta) t}.
\]

\end{proof}

Let us now get some insight into the behavior of the Filippov solutions to \eqref{E-2} (see Theorems \ref{T-wp-2} and \ref{T-wp-3}) in the most singular cases $\alpha=\frac{1}{2}$ and $\alpha\in (\frac{1}{2},1)$. Looking at Remark \ref{R-summary-synchro-2-oscillators} for the dynamics of $2$ oscillators, we expect global synchronization in finite time for $N$ oscillators. Specifically, in the supercritical case, the emerged global cluster is hoped to stay stuck independently on the chosen natural frequencies. In the critical case, the sticking conditions \eqref{E-sticking-critical-explicit-wp} are required for the cluster to remain stuck. To start with, let us prove the finite-time global phase synchronization of identical oscillators in the critical and supercritical cases. To that end, we need the following technical results.

\begin{lemma}\label{L-N-id-kernel-bound}
Consider $\alpha\in [\frac{1}{2},1)$, $\beta\in (0,2\alpha]$ and $\theta_*\in (0,\pi)$ and define the number
$$c({\alpha,\beta})=\left(\frac{2\alpha-\beta}{\beta}\right)^{1/2}.$$
Then, the following lower bound for $h_\varepsilon$ holds true
$$h_\varepsilon(\theta)\geq \frac{h_\varepsilon(\theta_*)}{\theta_*^\beta}\theta^\beta,\  \forall\,\theta\in [c(\alpha,\beta)\varepsilon,\theta_*],$$
for every $0<\varepsilon<c(\alpha,\beta)^{-1}\theta_*$. 
\end{lemma}

\begin{proof}
Define a scalar function
$$g_\varepsilon(\theta):=\frac{h_\varepsilon(\theta)}{\theta^\beta}=\frac{\frac{\sin\theta}{(\varepsilon^2+\theta^2)^\alpha}}{\theta^\beta},\ \theta\in (0,\pi).$$
We claim that $g_\varepsilon$ is nonincreasing in the interval $(c(\alpha,\beta)\varepsilon,\pi)$ for every $\varepsilon\in (0,c(\alpha,\beta)^{-1}\theta_*)$. Then, the result is apparent once monotonicity of $g_\varepsilon$ is proved. Indeed, taking derivatives we have
\begin{align*}
g_\varepsilon'(\theta)&=\frac{1}{\theta^{\beta+1}(\varepsilon^2+\theta^2)^\alpha}\left[\theta\cos\theta-\left(2\alpha\frac{\theta^2}{\varepsilon^2+\theta^2}+\beta\right)\sin\theta\right]\\
&=\frac{1}{\theta^{\beta+1}(\varepsilon^2+\theta^2)^\alpha}\left[\theta\cos\theta-\left(2\alpha+\frac{\beta\theta^2-(2\alpha-\beta)\varepsilon^2}{\theta^2+\varepsilon^2}\right)\sin\theta\right],
\end{align*}
for every $\theta\in (0,\frac{\pi}{2})$. Notice that $2\alpha\geq 1$ and $\beta\leq 2\alpha$. Then, by virtue of the definition of $c(\alpha,\beta)$ one checks that
$$\theta\cos\theta-\left(2\alpha+\frac{\beta\theta^2-(2\alpha-\beta)\varepsilon^2}{\theta^2+\varepsilon^2}\right)\sin\theta\leq \theta\cos\theta-\sin\theta\leq 0,$$
for every $\theta\in (c(\alpha,\beta)\varepsilon,\pi)$ and the monotonicity of $g_\varepsilon$ becomes clear.
\end{proof}

\begin{lemma}\label{L-N-id-critical-supercritical-asymptotic}
Let $\Theta=(\theta_1, \cdots, \theta_N)$ be the solution to \eqref{E-2} with $\alpha\in [\frac{1}{2},1)$ for identical oscillators, $\Omega_i = 0$, for $i=1, \ldots, N$ obtained in Theorems \ref{singular-limit-supercritical-theorem} and \ref{singular-limit-critical-theorem} as singular limits. Suppose the initial configuration $\Theta_0$ is confined in a half circle, i.e., $0<D(\Theta_0)<\pi$. Then, 
$$
\begin{array}{ll}
\displaystyle D(\Theta(t))\leq D(\Theta_0)e^{-K\frac{h(D(\Theta_0))}{D(\Theta_0)}}t, & \mbox{if } \alpha=\frac{1}{2},\\
\displaystyle D(\Theta(t))\leq\left(D(\Theta_0)^{1-2\alpha}+(2\alpha-1)2^{2\alpha-1}K\frac{h(D(\Theta_0))}{D(\Theta_0)^{2\alpha}}t\right)^{-\frac{1}{2\alpha-1}}, & \mbox{if }\alpha\in (\frac{1}{2},1),
\end{array}
$$
for every $t\geq 0$.
\end{lemma}

\begin{proof}
The main idea is to handle the approximate sequence $\{\Theta^\varepsilon\}_{\varepsilon>0}$ obtained as solutions to the regularized system \eqref{KM-regularized-grad-flow} and to take limits $\varepsilon\rightarrow 0$ in the phase diameter estimates. First, notice that by virtue of the assumed initial condition on the diameter one has that
$$\frac{d}{dt}D(\Theta^\varepsilon)\leq 0\ \mbox{ and }\ D(\Theta^\varepsilon(t))\leq D(\Theta_0)<\pi,\ \mbox{ for }\ t>0.$$
Indeed, note that we can obtain an explicit decay rate for the diameter by mimicking the ideas in Theorem \ref{T-N-id}. Namely, choosing $\theta_*=D(\Theta_0)$ and $\beta=2\alpha$ in Lemma \ref{L-N-id-kernel-bound}, we notice that $c(\alpha,\beta)=0$. Consequently, the lower bound of the kernel $h_\varepsilon$ is valid in the whole interval $[0,D(\Theta_0)]$. Then,
\begin{align*}
\frac{d}{dt}D(\Theta^\varepsilon)&=\frac{K}{N}\sum_{j=1}^N(h_\varepsilon(\theta_j^\varepsilon-\theta_M^\varepsilon)-h_\varepsilon(\theta_j^\varepsilon-\theta_m^\varepsilon))\\
&\quad -\frac{K}{N}\sum_{j=1}^N(h_\varepsilon(\theta_M^\varepsilon-\theta_j^\varepsilon)+h_\varepsilon(\theta_j^\varepsilon-\theta_m^\varepsilon))\\
&\leq -\frac{K}{N}\frac{h_\varepsilon(D(\Theta_0))}{D(\Theta_0)^{2\alpha}}\sum_{j=1}^N((\theta_M^\varepsilon-\theta_j^\varepsilon)^{2\alpha}+(\theta_j^\varepsilon-\theta_m^\varepsilon)^{2\alpha})\\
&\leq -\frac{K}{N}\frac{h_\varepsilon(D(\Theta_0))}{D(\Theta_0)^{2\alpha}}2^{2\alpha-1}\sum_{j=1}^N((\theta_M^\varepsilon-\theta_j^\varepsilon)+(\theta_j^\varepsilon-\theta_m^\varepsilon))^{2\alpha}\\
&=-K\frac{h_\varepsilon(D(\Theta_0))}{D(\Theta_0)^{2\alpha}}2^{2\alpha-1}D(\Theta)^{2\alpha}.
\end{align*}
Let us integrate the above differential inequality. We need to distinguish the cases $\alpha=\frac{1}{2}$ and $\alpha\in (\frac{1}{2},1)$:
$$
\begin{array}{ll}
\displaystyle D(\Theta^\varepsilon(t))\leq D(\Theta_0)e^{-K\frac{h_\varepsilon(D(\Theta_0))}{D(\Theta_0)}}t, & \mbox{if } \alpha=\frac{1}{2},\\
\displaystyle D(\Theta^\varepsilon(t))\leq\left(D(\Theta_0)^{1-2\alpha}+(2\alpha-1)2^{2\alpha-1}K\frac{h_\varepsilon(D(\Theta_0))}{D(\Theta_0)^{2\alpha}}t\right)^{-\frac{1}{2\alpha-1}}, & \mbox{if }\alpha\in (\frac{1}{2},1),
\end{array}
$$
for every $t\geq 0$. Recall that by virtue of Lemmas \ref{L-limit-supercritical} and \ref{L-limit-critical}, we obtained $\Theta^\varepsilon\overset{*}{\rightharpoonup} \Theta$ in $H^1((0,T);\mathbb{R}^N)$. In particular, $\Theta^\varepsilon\rightarrow \Theta$ in $C([0,T],\mathbb{R}^N)$. Then, we can take the limit $\varepsilon\rightarrow 0$ in the above estimates to attain the desired result.
\end{proof}

Under the assumptions in the preceding Lemma \ref{L-N-id-critical-supercritical-asymptotic} one obtains exponential decay of the diameter in the critical case and algebraic decay in the supercritical regime. However, a finite-time global synchronization is expected. This is the content of the following result.

\begin{theorem}\label{T-N-id-critical-supercritical-finite-time-1}
Let $\Theta=(\theta_1, \cdots, \theta_N)$ be the solution to \eqref{E-2} with $\alpha\in [\frac{1}{2},1)$ for identical oscillators, $\Omega_i = 0$, for $i=1, \ldots, N$ obtained in Theorems \ref{singular-limit-supercritical-theorem} and \ref{singular-limit-critical-theorem} as singular limits of the regularized solutions $\Theta^\varepsilon$ to \eqref{KM-regularized-grad-flow}. Assume that the initial configuration $\Theta_0$ is confined in a half circle, i.e., $0<D(\Theta_0)<\pi$. Then, for every $\beta\in (0,1)$ there exist two oscillators that collide at some time not larger than $T^1_c$, where
$$T_c^1=\frac{D(\Theta_0)}{(1-\beta)Kh(D(\Theta_0))}.$$
\end{theorem}

\begin{proof}
Let us assume the contrary. Then, by continuity there exists some $T>T^1_c$ so that there is no collision between oscillators along the time interval $[0,T]$. Again, by continuity there exists $\delta_T\in (0,D(\Theta_0))$ so that
$$\vert \theta_i(t)-\theta_j(t)\vert\geq \frac{\delta_T}{2},$$
for all $t\in [0,T]$ and every $i\neq j$. Since $\Theta^\varepsilon\rightarrow \Theta$ in $C([0,T],\mathbb{R}^N)$, then there exists $\varepsilon_0>0$ so that
$$\vert \theta_i^\varepsilon(t)-\theta_j^\varepsilon(t)\vert\geq \delta_T,$$
for all $t\in [0,T]$ and every $i\neq j$ and every $\varepsilon\in (0,\varepsilon_0)$. Take $\theta_*=D(\Theta_0)$ and consider a nonnegative 
$$\varepsilon_1<\min\{\varepsilon_0,c(\alpha,\beta)\theta_*^{-1}),c(\alpha,\beta)^{-1}\delta_T\}.$$
Then, it is clear that
$$\vert \theta_i^\varepsilon(t)-\theta_j^\varepsilon(t)\vert\in [c(\alpha,\beta)\varepsilon,\theta_*],$$
for every $t\in [0,T]$ any $\varepsilon\in (0,\varepsilon_1)$ and each $i\neq j$. Applying Lemma \ref{L-N-id-kernel-bound} we obtain
\begin{align*}
\frac{d}{dt}D(\Theta^\varepsilon)&=\frac{K}{N}\sum_{j=1}^N(h_\varepsilon(\theta_j^\varepsilon-\theta_M^\varepsilon)-h_\varepsilon(\theta_j^\varepsilon-\theta_m^\varepsilon))\\
&\quad -\frac{K}{N}\sum_{j=1}^N(h_\varepsilon(\theta_M^\varepsilon-\theta_j^\varepsilon)+h_\varepsilon(\theta_j^\varepsilon-\theta_m^\varepsilon))\\
&\leq -\frac{K}{N}\frac{h_\varepsilon(D(\Theta_0))}{D(\Theta_0)^{\beta}}\sum_{j=1}^N((\theta_M^\varepsilon-\theta_j^\varepsilon)^{\beta}+(\theta_j^\varepsilon-\theta_m^\varepsilon)^{\beta})\\
&\leq -\frac{K}{N}\frac{h_\varepsilon(D(\Theta_0))}{D(\Theta_0)^{\beta}}\sum_{j=1}^N((\theta_M^\varepsilon-\theta_j^\varepsilon)+(\theta_j^\varepsilon-\theta_m^\varepsilon))^{\beta}\\
&=-K\frac{h_\varepsilon(D(\Theta_0))}{D(\Theta_0)^{\beta}}D(\Theta)^{\beta},
\end{align*}
for every $t\in [0,T]$ and $\varepsilon\in (0,\varepsilon_1)$. Integrating the differential inequality yields
$$D(\Theta(t)^\varepsilon)\leq \left(D(\Theta_0)^{1-\beta}-(1-\beta)K\frac{h_\varepsilon(D(\Theta_0))}{D(\Theta_0)^\beta}t\right)^\frac{1}{1-\beta},$$
for every $t\in [0,T]$ and $\varepsilon\in (0,\varepsilon_1)$. Taking limits when $\varepsilon\rightarrow 0$ amounts to
$$D(\Theta(t))\leq \left(D(\Theta_0)^{1-\beta}-(1-\beta)K\frac{h(D(\Theta_0))}{D(\Theta_0)^\beta}t\right)^\frac{1}{1-\beta},$$
for each $t\in [0,T]$. However, it clearly yields a contradiction with the fact that $T>T^1_c$ due to the definition of $T^1_c$.
\end{proof}

The above result leads to a time estimate for the first collision between a couple of oscillators in the critical and supercritical cases. However, such idea can be repeated and improved in the critical case to give a total collision in finite time. The key ideas will be the uniqueness in Theorem \ref{T-wp-2} or, more specifically, the characterization of sticking of oscillators in Corollary \ref{C-sticking-2}. 

\begin{theorem}\label{T-N-id-critical-finite-time-global}
Let $\Theta=(\theta_1, \cdots, \theta_N)$ be the solution to \eqref{E-2} with $\alpha=\frac{1}{2}$ for identical oscillators, $\Omega_i = 0$, for $i=1, \ldots, N$. Assume that the initial configuration $\Theta_0$ is confined in a half circle, i.e., $0<D(\Theta_0)<\pi$. Then, there is complete phase synchronization in a finite time not larger than $T_c$, where
$$T_c=\frac{D(\Theta_0)}{Kh(D(\Theta_0))}.$$
\end{theorem}

\begin{proof}
Let us assume the contrary, i.e., complete synchronization does not arises along $[0,T_c]$. By continuity there exists some $T>T_c$ so that it does not happen along $[0,T]$ neither. Recall that by virtue of Corollary \ref{C-sticking-2}, sticking of oscillators takes place in the critical case after any collision. Then, the collision classes $\mathcal{C}_i(t)$ and sticking classes $S_i(t)$ in Subsection \ref{cluster-notation-subsection} agree each other. Let us list the family of collision (sticking) classes, i.e., the different clusters at time $t$
$$\mathcal{E}(t)=\{\mathcal{C}_1(t),\ldots,\mathcal{C}_N(t)\}=\{E_1(t),\ldots,E_{\kappa(t)}\}.$$
As a consequence of the assumed hypothesis $\kappa(t)$, is nonincreasing with respect to $t$ and bounded below by $2$. Coming back to the initial configuration, we define $i_M$ and $i_m$ in such a way that
$$
\max_{1\leq j\leq N}\theta_{j,0}=\theta_{i_M,0}\ \mbox{ and }\ \min_{1\leq j\leq N}\theta_{j,0}=\theta_{i_m,0}.
$$
Since the regularized system \eqref{KM-regularized-grad-flow} enjoys uniqueness in full sense, the oscillators $\theta_i^\varepsilon$ and $\theta_j^\varepsilon$ cannot cross. Similarly, by the Corollary \ref{C-sticking-2}, the oscillators $\theta_i$ and $\theta_j$ cannot cross neither unless they keep stuck together after that time. In any case, it is clear that
\begin{align*}
\max_{1\leq j\leq N}\theta_{j}(t)&=\theta_{i_M}(t), & \min_{1\leq j\leq N}\theta_{j}(t)&=\theta_{i_m}(t),\\
\max_{1\leq j\leq N}\theta_{j}^\varepsilon(t)&=\theta_{i_M}^\varepsilon(t), & \min_{1\leq j\leq N}\theta_{j}^\varepsilon(t)&=\theta_{i_m}^\varepsilon(t),
\end{align*}
for every $t\geq 0$ and any $\varepsilon>0$. Then, we have
$$D(\Theta^\varepsilon(t))=\theta_{i_M}^\varepsilon(t)-\theta_{i_m}^\varepsilon(t)\ \mbox{ and }\ D(\Theta(t))=\theta_{i_M}(t)-\theta_{i_m}(t),$$
for every $t\geq 0$ and any $\varepsilon>0$. Notice that all the above remarks ensure that for every $t\in [0,T]$
$$
\begin{array}{ll}
\theta_j(t)-\theta_{i_m}(t)>0, & \mbox{for all }j\in \mathcal{C}_{i_M}(t),\\
\theta_{i_M}(t)-\theta_j(t)>0, & \mbox{for all }j\in \mathcal{C}_{i_m}(t),\\
\theta_{i_M}(t)-\theta_j(t)>0, & \mbox{for all }j\notin \mathcal{C}_{i_M}\cup \mathcal{C}_{i_m}(t),\\
\theta_j(t)-\theta_{i_m}(t)>0, & \mbox{for all }j\notin \mathcal{C}_{i_M}\cup \mathcal{C}_{i_m}(t).
\end{array}
$$
Since $\Theta^\varepsilon\rightarrow\Theta$ in $C([0,T],\mathbb{R}^N)$, by continuity we can obtain $\varepsilon_0>0$ and $\delta_T>0$ so that
\begin{equation}\label{E-45}
\begin{array}{ll}
\theta_j^\varepsilon(t)-\theta_{i_m}^\varepsilon(t)>\delta_T, & \mbox{for all }j\in \mathcal{C}_{i_M}(t),\\
\theta_{i_M}^\varepsilon(t)-\theta_j^\varepsilon(t)>\delta_T, & \mbox{for all }j\in \mathcal{C}_{i_m}(t),\\
\theta_{i_M}^\varepsilon(t)-\theta_j^\varepsilon(t)>\delta_T, & \mbox{for all }j\notin \mathcal{C}_{i_M}\cup \mathcal{C}_{i_m}(t),\\
\theta_j^\varepsilon(t)-\theta_{i_m}^\varepsilon(t)>\delta_T, & \mbox{for all }j\notin \mathcal{C}_{i_M}\cup \mathcal{C}_{i_m}(t).
\end{array}
\end{equation}
for every $t\in [0,T]$ and every $\varepsilon\in (0,\varepsilon_0)$. Take $\theta_*=D(\Theta_0)$, fix $\beta\in (0,1)$ and consider a nonnegative 
$$\varepsilon_1<\min\{\varepsilon_0,c(\alpha,\beta)\theta_*^{-1}),c(\alpha,\beta)^{-1}\delta_T\}.$$
Then, it is clear that
\begin{equation}\label{E-46}
\begin{array}{ll}
\theta_j^\varepsilon(t)-\theta_{i_m}^\varepsilon(t)\in [c(\alpha,\beta)\varepsilon,\theta_*], & \mbox{for all }j\in \mathcal{C}_{i_M}(t),\\
\theta_{i_M}^\varepsilon(t)-\theta_j^\varepsilon(t)\in [c(\alpha,\beta)\varepsilon,\theta_*], & \mbox{for all }j\in \mathcal{C}_{i_m}(t),\\
\theta_{i_M}^\varepsilon(t)-\theta_j^\varepsilon(t)\in [c(\alpha,\beta)\varepsilon,\theta_*], & \mbox{for all }j\notin \mathcal{C}_{i_M}\cup \mathcal{C}_{i_m}(t),\\
\theta_j^\varepsilon(t)-\theta_{i_m}^\varepsilon(t)\in [c(\alpha,\beta)\varepsilon,\theta_*], & \mbox{for all }j\notin \mathcal{C}_{i_M}\cup \mathcal{C}_{i_m}(t),
\end{array}
\end{equation}
for every $t\in [0,T]$  and any $\varepsilon\in (0,\varepsilon_1)$. Now, let us split as follows
\begin{align*}
\frac{d}{dt}D(\Theta^\varepsilon)=&-\frac{K}{N}\sum_{j\in \mathcal{C}_{i_M}(t)}(h_\varepsilon(\theta_{i_M}^\varepsilon-\theta_j^\varepsilon)+h_\varepsilon(\theta_j^\varepsilon-\theta_{i_m}^\varepsilon))\\
&-\frac{K}{N}\sum_{j\in \mathcal{C}_{i_m}(t)}(h_\varepsilon(\theta_{i_M}^\varepsilon-\theta_j^\varepsilon)+h_\varepsilon(\theta_j^\varepsilon-\theta_{i_m}^\varepsilon))\\
&-\frac{K}{N}\sum_{j\notin \mathcal{C}_{i_M}(t)\cup \mathcal{C}_{i_m}(t)}(h_\varepsilon(\theta_{i_M}^\varepsilon-\theta_j^\varepsilon)+h_\varepsilon(\theta_j^\varepsilon-\theta_m^\varepsilon))\\
\leq &-\frac{K}{N}\sum_{j\in \mathcal{C}_{i_M}(t)}h_\varepsilon(\theta_j^\varepsilon-\theta_{i_m}^\varepsilon)-\frac{K}{N}\sum_{j\in \mathcal{C}_{i_m}(t)}h_\varepsilon(\theta_{i_M}^\varepsilon-\theta_j^\varepsilon)\\
&-\frac{K}{N}\sum_{j\notin \mathcal{C}_{i_M}(t)\cup \mathcal{C}_{i_m}(t)}(h_\varepsilon(\theta_{i_M}^\varepsilon-\theta_j^\varepsilon)+h_\varepsilon(\theta_j^\varepsilon-\theta_{i_m}^\varepsilon)),
\end{align*}
for every $t\in [0,T]$ and every $\varepsilon\in (0,\varepsilon_1)$. By virtue of Lemma \ref{L-N-id-kernel-bound} and the estimates in \eqref{E-46}, the above chain of inequalities implies
\begin{align*}
\frac{d}{dt}D(\Theta^\varepsilon)\leq &-\frac{K}{N}\frac{h_\varepsilon(D(\Theta_0))}{D(\Theta_0)^\beta}\sum_{j\in \mathcal{C}_{i_M}(t)}(\theta_j^\varepsilon-\theta_{i_m}^\varepsilon)^\beta-\frac{K}{N}\frac{h_\varepsilon(D(\Theta_0))}{D(\Theta_0)^\beta}\sum_{j\in \mathcal{C}_{i_m}(t)}(\theta_{i_M}^\varepsilon-\theta_j^\varepsilon)^\beta\\
&-\frac{K}{N}\frac{h_\varepsilon(D(\Theta_0))}{D(\Theta_0)^\beta}\sum_{j\notin \mathcal{C}_{i_M}(t)\cup \mathcal{C}_{i_m}(t)}((\theta_{i_M}^\varepsilon-\theta_j^\varepsilon)^\beta+(\theta_j^\varepsilon-\theta_{i_m}^\varepsilon)^\beta).
\end{align*}
Let us integrate such differential inequality to obtain
\begin{align*}
D(\Theta^\varepsilon(t))\leq D(\Theta^0)&-\frac{K}{N}\frac{h_\varepsilon(D(\Theta_0))}{D(\Theta_0)^\beta}\int_0^t\sum_{j\in \mathcal{C}_{i_M}(s)}(\theta_j^\varepsilon(s)-\theta_{i_m}^\varepsilon(s))^\beta\,ds\\
&-\frac{K}{N}\frac{h_\varepsilon(D(\Theta_0))}{D(\Theta_0)^\beta}\int_0^t\sum_{j\in \mathcal{C}_{i_m}(s)}(\theta_{i_M}^\varepsilon(s)-\theta_j^\varepsilon(s))^\beta\,ds\\
&-\frac{K}{N}\frac{h_\varepsilon(D(\Theta_0))}{D(\Theta_0)^\beta}\int_0^t\sum_{j\notin \mathcal{C}_{i_M}(s)\cup \mathcal{C}_{i_m}(s)}((\theta_{i_M}^\varepsilon(s)-\theta_j^\varepsilon(s))^\beta+(\theta_j^\varepsilon(s)-\theta_{i_m}^\varepsilon(s))^\beta)\,ds,
\end{align*}
for every $t\in [0,T]$ and every $\varepsilon\in (0,\varepsilon_1)$. Taking limits as $\varepsilon\rightarrow 0$ we obtain
\begin{align*}
D(\Theta(t))\leq D(\Theta^0)&-\frac{K}{N}\frac{h(D(\Theta_0))}{D(\Theta_0)^\beta}\int_0^t\sum_{j\in \mathcal{C}_{i_M}(s)}(\theta_{i_M}(s)-\theta_{i_m}(s))^\beta\,ds\\
&-\frac{K}{N}\frac{h(D(\Theta_0))}{D(\Theta_0)^\beta}\int_0^t\sum_{j\in \mathcal{C}_{i_m}(s)}(\theta_{i_M}(s)-\theta_{i_m}(s))^\beta\,ds\\
&-\frac{K}{N}\frac{h(D(\Theta_0))}{D(\Theta_0)^\beta}\int_0^t\sum_{j\notin \mathcal{C}_{i_M}(s)\cup \mathcal{C}_{i_m}(s)}((\theta_{i_M}(s)-\theta_j(s))^\beta+(\theta_j(s)-\theta_{i_m}(s))^\beta)\,ds.
\end{align*}
for every $t\in [0,T]$. To sum up, we obtain,
$$D(\Theta(t))\leq D(\Theta_0)-K\frac{h(D(\Theta_0))}{D(\Theta_0)^\beta}\int_0^t D(\theta(s))^\beta\,ds.$$
Hence, we find
$$D(\Theta(t))\leq \left(D(\Theta_0)^{1-\beta}-(1-\beta)K\frac{h(D(\Theta_0))}{D(\Theta_0)^\beta}t\right)^{\frac{1}{1-\beta}},$$
for all $t\in [0,T]$. Then, it is clear that
$$T<\frac{D(\Theta_0)}{(1-\beta)Kh(D(\Theta_0))},$$
for all $\beta\in (0,1)$. Taking limits when $\beta\rightarrow 0$ shows that $T\leq T_c$ and this yields the contradiction.
\end{proof}

\begin{remark}\label{R-N-id-supercritical-finite-time-global}
Notice that Theorem \ref{T-N-id-critical-supercritical-finite-time-1} also works in the supercritical case. However, the same proof as in Theorem \ref{T-N-id-critical-finite-time-global} is not valid to show finite-time complete phase synchronization of identical oscillators for $\alpha\in (\frac{1}{2},1)$. The reason is that at this point we cannot guarantee whether the Filippov solutions in $\Theta$ obtained as singular limit of the regularized solutions $\Theta^\varepsilon$ to system \eqref{KM-regularized-grad-flow} in Theorem \ref{singular-limit-supercritical-theorem} agrees with the solution obtained in Remark \ref{R-wp-3} via the ``sticking after collision'' continuation procedure of classical solutions. However, if the limiting $\Theta$ obtained in Theorem \ref{singular-limit-supercritical-theorem} satisfies such ``sticking after collision'' property, we can mimic Theorem \ref{T-N-id-critical-finite-time-global} to show that it exhibits complete phase synchronization at a finite time not larger than
$$T_c=\frac{D(\Theta_0)}{Kh(D(\Theta_0))}.$$
\end{remark}

\appendix

\section{Regular interactions}\label{regular-appendix}
\setcounter{equation}{0}
In this Appendix, we study the Kuramoto model with regular coupling weights:
\begin{equation}\label{D-1}
\dot\theta_i = \Omega_i + \frac{K}{N} \sum_{j=1}^N \frac{\sigma^{2\alpha}}{(\sigma^2 + c|\theta_j - \theta_i|^2_o)^\alpha}\sin(\theta_j - \theta_i) \quad \text{for} \quad i=1, \ldots, N,
\end{equation}
where we denote $c\equiv c_{\alpha,\zeta}=1-\zeta^{-1/\alpha}$ for simplicity. Recall that such model comes from the choice \eqref{C-1} of $\Gamma$ as the Hebbian plasticity function in \eqref{KM2}. Since the right hand side of \eqref{D-1} is Lipschitz continuous, then the system \eqref{D-1} has a unique solution by Cauchy--Lipschitz theory in this case.

For positive $\sigma$, we get the following bounds for $\Gamma$:
\[
\varepsilon_\sigma:= \frac{\sigma^{2\alpha}}{(\sigma^2 + c\pi^2)^\alpha} \leq \Gamma(\theta) \leq 1, \qquad \Gamma(0) = \Gamma(2\pi) = 1.
\]
Note that $\varepsilon_\sigma$ converges to zero as $\sigma \to 0$. We will study the emergence of synchronization for identical and non-identical oscillators and, we will use the idea of \cite{H-H-K} for the proof of synchronization.

\subsection{Identical oscillators}
Consider the Kuramoto model \eqref{D-1} for identical oscillators, which have the same natural frequency. Without loss of generality, we may assume $\Omega_i = 0$ for all $i=1, \ldots, N$. The system \eqref{D-1} becomes as follows:
\begin{equation}\label{D-2}
\dot\theta_i =  \frac{K}{N} \sum_{j=1}^N \frac{\sigma^{2\alpha}}{(\sigma^2 + c|\theta_j - \theta_i|^2_o)^\alpha}\sin(\theta_j - \theta_i), \quad i=1, \ldots, N.
\end{equation}

We can show the complete phase synchronization asymptotically for \eqref{D-2} with a constraint on initial configuration. Let us recall the notation $\theta_M(t)$ and $\theta_m(t)$ in \eqref{E-phase-M-m} for the indices of largest and shortest phases and $D(\Theta)$ for the phase diameter defined in \eqref{E-phase-diameter}.

\begin{theorem}\label{T-reg-id}
Let $\Theta=(\theta_1, \ldots, \theta_N)$ be the solution to \eqref{D-2}. Assume that the initial configuration is confined in a half circle, i.e. $D(\Theta_0)<\pi$, and the coupling strength $K$ is positive. Then, the solution $\Theta$ shows the complete phase synchronization asymptotically:
\[
D(\Theta_0) e^{-Kt} \leq D(\Theta) \leq D(\Theta_0) e^{-\frac{K \Gamma(D(\Theta_0)) \sin D(\Theta_0)}{D(\Theta_0)} t}.
\]
\end{theorem}
\begin{proof}
We consider the dynamics of phase diameter
\begin{equation}\label{D-3}
\frac{d}{dt} D(\Theta) = \frac{K}{N}\sum_{j=1}^N \Big( \Gamma(\theta_j - \theta_M)\sin(\theta_j - \theta_M) - \Gamma(\theta_j - \theta_m)\sin(\theta_j - \theta_m) \Big) .
\end{equation}
Since $\sin(\theta_j - \theta_M) \leq 0$ and $\sin(\theta_j - \theta_m) \geq 0$, as long as $D(\Theta)\leq \pi$, we have
\[
\frac{d}{dt} D(\Theta) \leq 0 \quad \text{and} \quad D(\Theta(t)) \leq D(\Theta_0) < \pi \quad \text{for} \quad t>0.
\]
By this contraction of phase difference, we have
\begin{equation}\label{D-4}
\sin(\theta_j - \theta_M) \leq \frac{\sin D(\Theta_0)}{D(\Theta_0)}(\theta_j - \theta_M) \quad \text{and} \quad \sin(\theta_j - \theta_m) \geq \frac{\sin D(\Theta_0)}{D(\Theta_0)}(\theta_j - \theta_m).
\end{equation}
On the other hand, we get
\begin{equation}\label{D-5}
\varepsilon_\sigma < \Gamma(D(\Theta_0)) \leq  \Gamma(D(\Theta)) \leq 1.
\end{equation}

By applying \eqref{D-4} and \eqref{D-5} to \eqref{D-3}, we attain the following differential inequality:
\begin{align*}
\frac{d}{dt} D(\Theta) &\leq \frac{K}{N}\sum_{j=1}^N \Big( \Gamma(\theta_j - \theta_M)\frac{\sin D(\Theta_0)}{D(\Theta_0)}(\theta_j - \theta_M) - \Gamma(\theta_j - \theta_m)\frac{\sin D(\Theta_0)}{D(\Theta_0)}(\theta_j - \theta_m)\Big)\\
&=-\frac{K}{N} \frac{\sin D(\Theta_0)}{D(\Theta_0)} \sum_{j=1}^N \Big( \Gamma(\theta_j - \theta_M)(\theta_M - \theta_j) + \Gamma(\theta_j - \theta_m)(\theta_j - \theta_m)\Big)\\
&\leq -\frac{K}{N}\frac{\Gamma(D(\Theta_0)) \sin D(\Theta_0)}{D(\Theta_0)}\sum_{j=1}^N \Big( (\theta_M - \theta_j) + (\theta_j - \theta_m)\Big)\\
& = -\frac{K \Gamma(D(\Theta_0)) \sin D(\Theta_0)}{D(\Theta_0)} D(\Theta).
\end{align*}
Gr\"{o}nwall's lemma yields the desired upper estimate. Similarly, from \eqref{D-5} and $\sin x \leq x$ for $0\leq x \leq \pi$, we have
\begin{align*}
\frac{d}{dt}D(\Theta) \geq \frac{K}{N} \sum_{j=1}^N \Big( (\theta_j - \theta_M) - (\theta_j - \theta_m)\Big)
= -KD(\Theta),
\end{align*}
which gives the lower estimate.
\end{proof}

\subsection{Non--identical oscillators}
We assume that the diameter of initial configuration is less than $D^\infty < \frac{\pi}{2}$. We first show that the diameter of phase is less than $D^\infty$ for all time $t\geq 0$ for sufficiently large coupling strength $K$. Let us recall that for $\theta\in (-\pi,\pi)$ the plasticity function reads $\Gamma(\theta)= \frac{\sigma^{2\alpha}}{(\sigma^2 + c\theta^2)^\alpha}$. Then, we have 
\begin{align*}
\Gamma^\prime(\theta) &= - \frac{2\sigma^{2\alpha} \alpha c \theta}{(\sigma^2 + c\theta^2)^{\alpha + 1}},\\
\Gamma^{\prime\prime}(\theta) &= - \frac{2\sigma^{2\alpha} \alpha c \big[\sigma^2 - (2\alpha + 1)c\theta^2\big]}{(\sigma^2 + c\theta^2)^{\alpha + 2}}.
\end{align*}
If we set 
\[\theta_{\pm} := \pm\frac{\sigma}{\sqrt{c(2\alpha + 1)}},\]
then $\Gamma'$ attains its global extrema on such points, namely
\[
\Gamma^\prime(\theta_-) = \max_{\theta \in (-\pi, \pi)} \Gamma^\prime(\theta) > 0 \quad \text{and}\quad \Gamma^\prime(\theta_+) = \min_{\theta \in (-\pi, \pi)} \Gamma^\prime(\theta) < 0.
\]
Indeed, we get
\[
\Gamma^\prime(\theta_-) = - \Gamma^\prime(\theta_+) = \frac{2\alpha \sqrt{c}}{\sigma \sqrt{2\alpha + 1} (1 + \frac{1}{2\alpha + 1})^{\alpha + 1}}.
\]
We first show the boundedness of phase differences.
\begin{lemma}\label{L-reg-bd}
Assume that $D(\Theta_0) < D^\infty$, for some small $D^\infty < \frac{\pi}{2}$, and that  the coupling strength is sufficiently large so that 
\[
-\Gamma^\prime(\theta_+) < \frac{\Gamma(D^\infty)}{\tan D^\infty} \quad \text{and} \quad K > \frac{D(\dot\Theta_0)}{\Big[ \Gamma^\prime(\theta_+) \sin D^\infty + \Gamma(D^\infty)\cos D^\infty\Big](D^\infty - D(\Theta_0))} .
\]
Then, we have
\[
D(\Theta(t)) < D^\infty \quad \text{for} \quad t\geq 0.
\]
\end{lemma}
\begin{proof}
Assume that there exists a time for which  $D(\Theta(t)) \geq D^\infty$. Then, due to the continuity
\[
t^* := \sup\{t>0: D(\Theta(s)) < D^\infty ~ \text{for} ~ 0\leq s \leq t \},
\]
is positive and finite and $D(\Theta(t^*)) = D^\infty$. We set indices $F$ and $S$ so that
\[
\dot\theta_F(t) := \max\{\dot\theta_1(t), \ldots,\dot\theta_N(t) \} \quad \text{and}\quad \dot\theta_S(t):= \min\{\dot\theta_1(t), \ldots, \dot\theta_N(t)\},
\]
for each time $t$ and define the diameter of frequency so that
\[
D(\dot\Theta(t)) := \dot\theta_F(t) - \dot\theta_S(t).
\]
Then, we have
\begin{equation}\label{D-6}
D(\dot\Theta(t)) - D(\dot\Theta_0) = \int_0^t \frac{d}{ds} D(\dot\Theta(s)) \, ds.
\end{equation}
By taking time derivative on $D(\dot\Theta)$, we get
\begin{equation}\label{D-7}
\begin{aligned}
\frac{d}{dt} D(\dot\Theta) &= \frac{K}{N} \sum_{j=1}^N \Big[ \Gamma^\prime(\theta_j - \theta_F)\sin(\theta_j - \theta_F) + \Gamma(\theta_j - \theta_F)\cos(\theta_j - \theta_F)\Big] (\dot\theta_j - \dot\theta_F) \\
&- \frac{K}{N}\sum_{j=1}^N \Big[ \Gamma^\prime(\theta_j - \theta_S)\sin(\theta_j - \theta_S) + \Gamma(\theta_j - \theta_S)\cos(\theta_j - \theta_S)\Big] (\dot\theta_j - \dot\theta_S).
\end{aligned}
\end{equation}

Then, we get the following couple of upper and lower bounds
\begin{align}
\displaystyle\Gamma'(\theta_+) \sin D^\infty  &\leq \Gamma^\prime(\theta_j - \theta_i)\sin(\theta_j - \theta_i) \leq 0,\label{D-8}\\
\displaystyle\Gamma(D^\infty)\cos D^\infty &\leq \Gamma(\theta_j - \theta_i)\cos(\theta_j - \theta_i) \leq 1.\label{D-9}
\end{align}
By applying \eqref{D-8} and \eqref{D-9} into \eqref{D-7}, we deduce
\begin{equation}\label{D-10}
\begin{aligned}
\frac{d}{dt}D(\dot\Theta) &\leq \frac{K}{N} \sum_{j=1}^N \Big[ \Gamma'(\theta_+)\sin D^\infty + \Gamma(D^\infty)\cos D^\infty\Big] \Big( (\dot\theta_j - \dot\theta_F) - (\dot\theta_j - \dot \theta_S)\Big) \\
&= -K \Big[\Gamma'(\theta_+)\sin D^\infty + \Gamma(D^\infty)\cos D^\infty\Big] (\dot\theta_F - \dot\theta_S)\\
&\leq - K \underbrace{\Big[ \Gamma^\prime(\theta_+) \sin D^\infty + \Gamma(D^\infty)\cos D^\infty\Big]}_{>0} D(\dot\Theta),
\end{aligned}
\end{equation}
for every $t\in [0,t^*]$. Combining \eqref{D-6} and \eqref{D-10}, we obtain
\begin{equation}\label{D-11}
D(\dot\Theta(t)) \leq  D(\dot\Theta_0)  - K \Big[ \Gamma^\prime(\theta_+) \sin D^\infty + \Gamma(D^\infty)\cos D^\infty\Big] \int_0^t D(\dot\Theta(s)) ds,
\end{equation}
for every $t\in [0,t^*]$. Let us define $y(t) := \int_0^t D(\dot\Theta(s)) ds$. Thus, the inequality \eqref{D-11} can be rewritten into
\[
y^\prime(t)\leq y^\prime(0) - Cy(t).
\]
Here, $C := K \Big[ \Gamma^\prime(\theta_+) \sin D^\infty + \Gamma(D^\infty)\cos D^\infty\Big]$ and $t\in [0,t^*]$. Then, we find
\[
y(t) \leq \frac{y^\prime(0)}{C}(1- e^{-Ct}) \leq \frac{y^\prime(0)}{C},
\]
for all $t\in [0,t^*]$. However, since $D(\Theta(t^*)) = D^\infty$, we get
\begin{align*}
D^\infty &= D(\Theta_0) + \int_0^{t^*} \frac{d}{ds} D(\Theta(s)) \, ds \leq D(\Theta_0) + \int_0^{t^*} D(\dot\Theta(s)) \, ds\\
&\leq D(\Theta_0)+y(t^*)\leq D(\Theta_0) + \frac{y^\prime(0)}{C} < D^\infty,
\end{align*}
when 
\[K > \frac{D(\dot\Theta_0)}{\Big[ \Gamma^\prime(\theta_+) \sin D^\infty + \Gamma(D^\infty)\cos D^\infty\Big](D^\infty - D(\Theta_0))},\]
which yields to a contradiction. Thus, $D(\Theta(t))<D^\infty$, for all $t\geq 0$.
\end{proof}
We are ready to prove the frequency synchronization for non-identical oscillators.
\begin{theorem}\label{T-reg-pl}
Assume that $D(\Theta_0) < D^\infty$, for some small $D^\infty < \frac{\pi}{2}$, and that the coupling strength is sufficiently large so that 
\[
-\Gamma^\prime(\theta_+) < \frac{\Gamma(D^\infty)}{\tan D^\infty} \quad \text{and} \quad K > \frac{D(\dot\Theta_0)}{\Big[ \Gamma^\prime(\theta_+) \sin D^\infty + \Gamma(D^\infty)\cos D^\infty\Big](D^\infty - D(\Theta_0))} .
\]
Then, we deduce a complete frequency synchronization
\[
D(\dot\Theta(0)) e^{-Kt} \leq D(\dot\Theta(t)) \leq D(\dot\Theta(0)) e^{-K\Big[ \Gamma^\prime(\theta_+) \sin D^\infty + \Gamma(D^\infty)\cos D^\infty\Big]t}.
\]
\end{theorem}
\begin{proof}
From \eqref{D-7}-\eqref{D-10}, we obtain
\[
\frac{d}{dt}D(\dot\Theta) \leq - K \Big[ \Gamma^\prime(\theta_+) \sin D^\infty + \Gamma(D^\infty)\cos D^\infty\Big] D(\dot\Theta).
\]
On the other hand, from \eqref{D-7}-\eqref{D-9}, we have
\[
\frac{d}{dt} D(\dot\Theta) \geq -KD(\dot\Theta).
\]
By Gronwall's lemma, we achieve the exponential estimates for the frequency synchronization.
\end{proof}
Since the decay rate of the asymptotic frequency synchronization is exponential, then the solution $\Theta$ shows the emergence of a phase-locked state.

\section{H-representation of the Filippov set-valued maps}\label{Appendix-Hrep-Fil-map}

In this appendix, we exhibit the proofs of the technical Lemmas \ref{L-skew-sym-property} and \ref{L-skew-sym-property-bounded-coefficients}. Recall that such results were respectively applied in Propositions \ref{P-explicit-Fil-map-supercritical-implicit-eq} and \ref{P-explicit-Fil-map-critical-implicit-eq} in order to characterize explicitly some H-representation of the Filippov set-valued map in the supercritical and critical cases. We introduce some notation that will be used here on.

\begin{definition}\label{D-appendix}
Consider $n\in\mathbb{N}$. For every $i,j\in \{1,\ldots,n\}$ we define the linear operator
\[\begin{array}{cccc}
L_{ij}: & \Skew_n(\mathbb{R}) & \longrightarrow & \mathbb{R},\\
 & Y & \longmapsto & y_{ij},\\
L_i: & \Skew_n(\mathbb{R}) & \longrightarrow & \mathbb{R},\\
 & Y & \longmapsto & \sum_{k=1}^ny_{ik},\\
\mathcal{L}: & \Skew_n(\mathbb{R}) & \longrightarrow & \mathbb{R}^n,\\
 & Y & \longmapsto & Y\cdot \mathbf{j}.
\end{array}
\]
By definition, the following relations hold true
\[L_i=\sum_{k=1}^nL_{ik}\ \mbox{ and }\ \mathcal{L}=(L_1,\ldots,L_n).\]
\end{definition}

First, we give the simpler proof of Lemma \ref{L-skew-sym-property}:
\begin{lemma}\label{L-skew-sym-property-appendix}
Consider any $n\in\mathbb{N}$ and any vector $x\in\mathbb{R}^n$. Then, the following assertions are equivalent:
\begin{enumerate}
\item There exists some $Y\in\Skew_n(\mathbb{R})$ such that 
\[x=Y\cdot\mathbf{j}.\]
\item The following implicit equation holds true
\[x\cdot \mathbf{j}=0,\]
\end{enumerate}
where $\mathbf{j}$ stand for the vector of ones.
\end{lemma}

\begin{proof}
Let us define the following linear operator
\[
\begin{array}{cccc}
\mathcal{L}: & \Skew_n(\mathbb{R}) & \longrightarrow & \mathbb{R}^n,\\
 & Y & \longmapsto & Y\cdot \mathbf{j}.
\end{array}
\]
Then, the thesis of this lemma is equivalent to 
\begin{equation}\label{E-21}
\mathcal{L}(\Skew_n(\mathbb{R}))=\mathbf{j}^\perp.
\end{equation}
On the one hand, it is clear that the inclusion $\subseteq$ in \eqref{E-21} fulfils by virtue of the properties of the skew symmetric matrices. On the other hand, let us define the matrices
\begin{equation}\label{E-skew-symmetric-basis}
E_{ij}:=\frac{1}{2}(e_i\otimes e_j-e_j\otimes e_i),
\end{equation}
for every $i\neq j$, where $\{e_i:\,i=1,\ldots,N\}$ is the standard basis of $\mathbb{R}^n$ and $\otimes$ denotes the Kronecker product. Notice that
\[\mathcal{L}(E_{ij})=\frac{1}{2}(e_i\otimes e_j-e_j\otimes e_i)\cdot \mathbf{j}=e_i-e_j.\]
Hence, $\{\mathcal{L}(E_{i,i+1}):i=1,\ldots,n-1\}=\{e_i-e_{i+1}:i=1,\ldots,n-1\}$ consist of $n-1$ independent vectors. Consequently, $\mathcal{L}$ has rank larger or equal to $n-1$. Since $\mathbf{j}^\perp$ has rank equal to $n-1$ the full identity in \eqref{E-21} holds true.
\end{proof}

Now, we focus on the proof of Lemma \ref{L-skew-sym-property-bounded-coefficients}. Our main tool in this part will be the \textit{Farkas alternative} from convex analysis that we recall in the subsequent result.

\begin{lemma}[Farkas alternative]\label{L-farkas-appendix}
Consider any finite-dimensional vector space $V$, some finite family of linear operators $T_1,\ldots,T_k:V\longrightarrow \mathbb{R}$ and $b=(b_1,\ldots,b_k)\in\mathbb{R}^k$. Then, exactly one of the following statements holds true:
\begin{enumerate}
\item There exists $v\in V$ such that 
\[T_i(v)\leq b_i,\ \ i=1,\ldots,k.\]
\item There exists $q\in\mathbb{R}^n$ with $q_i\geq 0$ for all $i=1,\ldots k$ such that
\[\sum_{i=1}^kq_iT_i\equiv 0\ \mbox{ and }\ q\cdot b<0.\]
\end{enumerate}
\end{lemma}

This result has several equivalent representations in the literature and it is sometimes called the \textit{Theorem of Alternatives}. One clear reference where we can infer our version from can be found in \cite[Lemma 2.54]{Tao}. We are now ready to give a proof of Lemma \ref{L-skew-sym-property-bounded-coefficients}.

\begin{lemma}\label{L-farkas-application-appendix}
Consider any $n\in\mathbb{N}$ and any vector $x\in\mathbb{R}^n$. Then, the following two assertions are equivalent:
\begin{enumerate}
\item There exists some $Y\in\Skew_n([-1,1])$ such that
\[x=Y\cdot \mathbf{j}.\]
\item There exists some $Y\in\Skew_n(\mathbb{R})$ such that
\[L_{ij}(Y)\leq 1,\ L_i(Y)\leq x_i\mbox{ and }-L_i(Y)\leq -x_i.\]
\item The following inequality 
\begin{equation}\label{E-24}
\sum_{i,j=1}^nq_{ij}+\lambda_ix_i\geq 0,
\end{equation}
 holds, for every $Q\in \mathcal{M}_n(\mathbb{R}_0^+)$ and $\lambda\in\mathbb{R}^n$ such that $q_{ij}+\lambda_i=q_{ji}+\lambda_j$.
\item We have that
\[\sum_{i=1}^k x_{\sigma_i}\in [-k(n-k),k(n-k)],\]
for every permutation $\sigma$ of $\{1,\ldots,n\}$ and any $k\in\mathbb{N}$.
\end{enumerate}
\end{lemma}

\begin{proof}
For the shake of simplicity in our arguments, we will split the proof into two parts. In the first part, we establish the equivalence of the first three assertions in the statement. The main tool to be used in such part is the above Lemma \ref{L-farkas-appendix}. In the second part, we will focus on the more convoluted equivalence between the first group of equivalent assertions in the above-mentioned step and the last assertion.

\medskip
\noindent
$\bullet$ \textit{Step 1: Equivalence of the first three assertions}. On the one hand, the first two assertion are perfectly equivalent by virtue of Definition \ref{D-appendix}. Then, our problem is a system of affine inequalities in the vector space $\Skew_n(\mathbb{R})$ of skew symmetric matrices. Hence, by Farkas alternative (see Lemma \ref{L-farkas-appendix}) such assertions amounts to saying that whenever $q_{ij},q_i^+,q_i^-$ are non-negative coefficients verifying
\[
\sum_{i,j=1}^nq_{ij}L_{ij}+\sum_{i=1}^nq_i^+L_i-\sum_{i=1}^nq_i^-L_i\equiv 0\ \mbox{ in }\Skew_n(\mathbb{R}),
\]
then
\[
\sum_{i,j=1}^nq_{ij}+\sum_{i=1}^nq_i^+x_i-\sum_{i=1}^nq_i^-x_i\geq 0.
\]
Defining $\lambda_i=q_i^+-q_i^-$, we can simplify an equivalent assertion: for every $Q\in\mathcal{M}_n(\mathbb{R}_0^+)$ and $\lambda\in\mathbb{R}^n$ such that
\begin{equation}\label{E-22}
\sum_{i,j=1}^nq_{ij}L_{ij}+\sum_{i=1}^n\lambda_iL_i\equiv 0\ \mbox{ in }\Skew_n(\mathbb{R}),
\end{equation}
then
\[
\sum_{i,j=1}^nq_{ij}+\sum_{i=1}^n\lambda_ix_i\geq 0.
\]
Thus, the equivalence with the third assertion follows by evaluating the identity \eqref{E-22} on every matrix in the canonical basis of $\Skew_n(\mathbb{R})$, i.e.,
\[\left\{e_i\otimes e_j-e_j\otimes e_i:1\leq i<j\leq n\right\},\]
and noticing that we obtain the condition $q_{ij}+\lambda_i=q_{ji}+\lambda_j$ in such third assertion.

\medskip
\noindent
$\bullet$ \textit{Step 2: Equivalence with the last assertion}. On the one hand, let us assume that the first assertion is satisfied, i.e., $x=Y\cdot\mathbf{j}$ for some $Y\in\Skew_n([-1,1])$. Taking any permutation $\sigma$ of $\{1,\ldots,n\}$ and any $1\leq k\leq n$ we obtain
\[
\sum_{i=1}^kx_{\sigma_i}=\sum_{i=1}^k\sum_{j=1}^ny_{\sigma_i\sigma_j}=\sum_{i=1}^k\sum_{j=1}^ky_{\sigma_i\sigma_j}+\sum_{i=1}^k\sum_{j=k+1}^ny_{\sigma_i\sigma_j}.
\]
Since the first term becomes zero (by anti-symmetry) and the second term consists of $n(n-k)$ terms with values in $[-1,1]$, then
\[\sum_{i=1}^kx_{\sigma_i}\in [-k(n-k),k(n-k)].\]
Conversely, assume that the last assertion is true and let us prove \eqref{E-24} in the third assertion. Consider $Q\in \mathcal{M}_n(\mathbb{R}_0^+)$ and $\lambda\in \mathbb{R}^n$ such that
\begin{equation}\label{E-23}
q_{ij}-q_{ji}=\lambda_j-\lambda_i.
\end{equation}
Without loss of generality we will assume that $q_{ii}=0$, for every $i=1,\ldots,n$ (notice that in other case, \eqref{E-24} is even larger), and let us split
\[I:=\sum_{i\neq j}q_{ij}+\sum_{i=1}^n\lambda_ix_i=:I_1+I_2.\]
On the one hand, let us rewrite $I_2$ and notice that 
\[I_2=\sum_{i=1}^n\lambda_ix_i=\sum_{i=1}^n(\lambda_i-\lambda_j)x_i+\lambda_j\sum_{i=1}x_i,\]
for every $j=1,\ldots,n$. Since the sum of all the $x_i$ becomes zero by hypothesis, taking averages with respect to all the indices $j=1,\ldots,n$ we obtain that
\[I_2=\frac{1}{n}\sum_{i=1}^n\sum_{j=1}^n(\lambda_i-\lambda_j)x_i.\]
Finally, changing the indices $i$ with $j$ and taking the average of both expressions we can equivalently write
\[I_2=\frac{1}{2n}\sum_{i=1}^n\sum_{j=1}^n(\lambda_i-\lambda_j)(x_i-x_j)=\frac{1}{n}\sum_{i<j}(\lambda_j-\lambda_i)(x_j-x_i).\]
Thus, substituting \eqref{E-23} into $I_2$ and putting it together with $I_1$ we can rewrite
\begin{equation}\label{E-25}
I=\sum_{i\neq j}q_{ij}+\frac{1}{2n}\sum_{i\neq j}(q_{ij}-q_{ji})(x_j-x_i)=\sum_{i\neq j}q_{ij}\left(1+\frac{1}{n}(x_j-x_i)\right).
\end{equation}
Let us consider a permutation $\sigma$ of $\{1,\ldots,n\}$ so that we can order the coefficients $\lambda_i$ in increasing way, i.e.,
\begin{equation}\label{E-27}
\lambda_{\sigma_1}\leq \lambda_{\sigma_2}\leq \cdots\leq \lambda_{\sigma_n}.
\end{equation}
Then,
\begin{align*}
I&=\sum_{i\neq j}q_{\sigma_i\sigma_j}\left(1+\frac{1}{n}(x_{\sigma_j}-x_{\sigma_i})\right)\\
&=\sum_{i<j}(q_{\sigma_i\sigma_j}-q_{\sigma_j\sigma_i})\left(1+\frac{1}{n}(x_{\sigma_j}-x_{\sigma_i})\right)+2\sum_{i<j}q_{\sigma_j\sigma_i}=:I_3+I_4.
\end{align*}
It is clear that $I_4$ is non-negative. Hence, we will focus on showing that so is $I_3$ too. By virtue of \eqref{E-23}, it is easy to show that 
\[q_{\sigma_i\sigma_j}-q_{\sigma_j}\sigma_i=\sum_{k=i}^{j-1}(q_{\sigma_k\sigma_{k+1}}-q_{\sigma_{k+1}\sigma_k}),\]
for every $i<j$. Thereby,
\begin{align}
I_3&=\sum_{i<j}\sum_{k=i}^{j-1}(q_{\sigma_k\sigma_{k+1}}-q_{\sigma_{k+1}\sigma_k})\left(1+\frac{1}{n}(x_{\sigma_j}-x_{\sigma_i})\right)\nonumber\\
&=\sum_{k=1}^{n-1}a_k(q_{\sigma_k\sigma_{k+1}}-q_{\sigma_{k+1}\sigma_k})=\sum_{k=1}^{n-1}a_k(\lambda_{\sigma_{k+1}}-\lambda_{\sigma_k})\label{E-26},
\end{align}
where in the last step we have used \eqref{E-23} again and the coefficients $a_k$ read
\[a_k:=\sum_{\substack{i\leq k\\ j\geq k+1}}\left(1+\frac{1}{n}(x_{\sigma_j}-x_{\sigma_i})\right)=k(n-k)+\frac{k}{n}\sum_{j=k+1}^nx_{\sigma_j}-\frac{n-k}{n}\sum_{i=1}^kx_{\sigma_i}.\]
Bearing in mind that the sum of all the $x_i$ vanishes by hypothesis, then
\[
a_k=k(n-k)-\sum_{i=1}^kx_{\sigma_i}.
\]
Then, $a_k\geq 0$ by hypothesis. Since we have chosen $\sigma$ so that \eqref{E-27} takes place, then the result follows from the above expression \eqref{E-26} for $I_3$.
\end{proof}

\section{Characterizing the sticking conditions}\label{Appendix-sticking}
Our purpose in this appendix is to characterize explicit conditions for the weights specifying the necessary and sufficient conditions for sticking of particles \eqref{E-sticking-critical-explicit-wp} and \eqref{E-sticking-supercritical-implicit-wp} in the Subsections \eqref{singular-limit-critical-subsection} and \eqref{singular-limit-subcritical-subsection} respectively. The first part is devoted to the latter condition for the supercritical case and the second part will focus on the former critical case.

Apart form the linear operators in Definition \ref{D-appendix} we will need the following ones.
\begin{definition}\label{D-appendix-2}
Consider $n\in\mathbb{N}$. For every $i,j\in\{1,\ldots,n\}$ we define the linear operator 
\[\begin{array}{cccc}
T_{ij}: & \Skew_n(\mathbb{R}) & \longrightarrow & \mathbb{R},\\
 & Y & \longmapsto & \sum_{k=1}^n(y_{ik}-y_{jk}).
\end{array}
\]
Notice that by definition we get the following relation with the operators in Definition \ref{D-appendix}
\[T_{ij}=\sum_{k=1}^n(L_{ik}-L_{jk}).\]
\end{definition}

Then, the next result yields a characterization for the sticking condition \eqref{E-sticking-supercritical-implicit-wp} to hold.

\begin{lemma}\label{L-farkas-application-appendix-2}
Consider any $n\in\mathbb{N}$ and any matrix $M\in\Skew_n(\mathbb{R})$. Then, the following assertions are equivalent:
\begin{enumerate}
\item There exists some $Y\in\Skew_n(\mathbb{R})$ such that
\[M=Y\cdot \mathbf{J}+\mathbf{J}\cdot Y.\]
\item There exits some $Y\in Skew_n(\mathbb{R})$ such that
\[T_{ij}(Y)\leq m_{ij}\ \mbox{ and }\ -T_{ij}(Y)\leq -m_{ij}.\]
\item We have 
\begin{equation}\label{E-28}
m_{1i}+m_{ij}+m_{j1}=0,
\end{equation}
for every $2\leq i<j\leq n$.
\item The equality
\begin{equation}\label{E-sticking-supercritical-explicit}
m_{ij}+m_{jk}+m_{ki}=0,
\end{equation}
holds, for every $1\leq i<j<k\leq n$.
\end{enumerate}
\end{lemma}

\begin{proof}
First, it is clear that the first two assertions are equivalent. Second, let us briefly show that \eqref{E-28} and \eqref{E-sticking-supercritical-explicit} are equivalent. On the one hand, it is clear that \eqref{E-28} is a particular case of \eqref{E-sticking-supercritical-explicit}. On the other hand, let us assume that \eqref{E-28} fulfills. Then, we have in particular the next three equations for $1\leq i<j<k\leq n$
\begin{align*}
m_{1i}+m_{ij}+m_{j1}&=0,\\
m_{1j}+m_{jk}+m_{k1}&=0,\\
m_{1k}+m_{ki}+m_{i1}&=0.
\end{align*}
Taking the sum of such equations we obtain \eqref{E-sticking-supercritical-explicit} by virtue of the skew-symmetry of $M$. Hence, let us just concentrate on proving the equivalence between the second and third assertions. By Lemma \ref{L-farkas-appendix}, the second assertion amounts to saying that whenever $\Lambda\in \mathcal{M}_n(\mathbb{R})$ verifies
\[
\sum_{i,j=1}^n\lambda_{ij}T_{ij}\equiv 0,
\]
then the following condition fulfills
\[
\sum_{i,j=1}^n\lambda_{ij}m_{ij}\geq 0.
\]
Evaluating along the basis $e_i\otimes e_j-e_j\otimes e_i$ we equivalently write the former condition as
\[\sum_{k=1}^n\left[(\lambda_{ik}-\lambda_{ki})-(\lambda_{jk}-\lambda_{kj})\right]=0.\]
Hence, if we define $p_{ij}=\lambda_{ij}-\lambda_{ji}$ we can conclude that the second assertion of this Lemma is completely equivalent to the fact that whenever $P\in\Skew_n(\mathbb{R})$ verifies
\begin{equation}\label{E-29}
\sum_{k=1}^n(p_{ik}-p_{jk})=0,
\end{equation}
for every $i,j\in\{1,\ldots,n\}$, then
\begin{equation}\label{E-30}
\sum_{i,j=1}^n p_{ij}m_{ij}\geq 0.
\end{equation}

\medskip
\noindent
$\bullet$ \textit{Step 1:} Here, we characterizing the condition \eqref{E-29}. Taking
\[x=\left(\sum_{k=1}^n p_{jk}\right)\mathbf{j},\]
in Lemma \ref{L-skew-sym-property} shows that those matrices $P\in\Skew_n(\mathbb{R})$ fulfilling \eqref{E-29} agree with the matrices that lie in the kernel of the operator $\mathcal{L}=(L_1,\ldots,L_n)$. Recall that by virtue of such result, $\mathcal{L}$ has rank equal to $n-1$. Since $\Skew_n(\mathbb{R})$ is a vector space with dimension $d_1:=n(n-1)/2$, then we know that
\[d_2:=\dim(\ker \mathcal{L})=\frac{n(n-1)}{2}-(n-1)=\frac{(n-1)(n-2)}{2}.\]
Consider the following matrices
\begin{equation}\label{E-31}
P_{ij}:=E_{1i}+E_{ij}+E_{j1}=E_{1i}-E_{ij}+E_{ij},
\end{equation}
where $E_{ij}$ are the skew symmetric matrices in \eqref{E-skew-symmetric-basis}. Then,
\[\mathcal{L}(P_{ij})=\mathcal{L}(E_{1i})+\mathcal{L}(E_{ij})+\mathcal{L}(E_{j1})=(e_1-e_j)+(e_i-e_j)+(e_j-e_1)=0.\]
Hence, the following subset
\[\mathcal{P}:=\{P_{ij}:2\leq i<j\leq n\}\subseteq \ker \mathcal{L},\]
consists of $(n-1)(n-2)/2$ different elements, which we can be classified  via the lexicographic order of multi-indices $(i,j)$. Let us show that all of them are linearly independent, thus generating the whole kernel. We first consider the basis of skew-symmetric matrices
\[\mathcal{B}:=\{E_{ij}:\,1\leq i<j\leq n\},\]
and, again, we can list them ordered with respect to the lexicographic order. Let us consider the matrix $\mathcal{M}\in \mathcal{M}_{d_2\times d_1}(\mathbb{R})$ of coordinates of the elements in $\mathcal{P}$ with respect to the basis $\mathcal{B}$. Then, by the definition \eqref{E-31} one infers that the $d_2\times d_2$ identity matrix appears as the submatrix of $\mathcal{M}$ consisting of all the $d_2$ rows but just the last $d_2$ columns. Hence, $\rank \mathcal{M}=d_2$ and, consequently,
\[\ker \mathcal{L}=\mbox{span}(\mathcal{P}).\]

\medskip
\noindent
$\bullet$ \textit{Step 2:} Here, we characterize the condition \eqref{E-30}, that clearly amounts to show that
\[\sum_{i,j=1}^np_{ij}m_{ij}=0,\]
for every $P\in\mathcal{P}$. Taking $P=P_{ij}$ for $2\leq i<j\leq n$ we get
\[\sum_{i,j=1}^np_{ij}m_{ij}=\frac{1}{2}\left(m_{1i}-m_{i1}+m_{ij}-m_{ji}+m_{j1}-m_{1j}\right)=m_{1i}+m_{ij}+m_{j1},\]
and this concludes the full proof of our result.
\end{proof}

Finally, we focus on the sticking condition \eqref{E-sticking-critical-implicit-wp} in the critical case. The next result exhibits an explicit characterization that follows similar techniques to those in Lemma \ref{L-farkas-application-appendix}.

\begin{lemma}\label{L-farkas-application-appendix-3}
Consider any $n\in\mathbb{N}$ and any matrix $M\in \Skew_n(\mathbb{R})$. Then, the following assertions are equivalent:
\begin{enumerate}
\item There exists some $Y\in \Skew_n([-1,1])$ such that
\[M=Y\cdot \mathbf{J}+\mathbf{J}\cdot Y.\]
\item There exists some $Y\in \Skew_n(\mathbb{R})$ such that
\[T_{ij}(Y)\leq m_{ij},\ -T_{ij}(Y)\leq -m_{ij}\ \mbox{ and }\ L_{ij}(Y)\leq 1.\]
\item The following inequality
\[\sum_{i,j=1}^nq_{ij}+\frac{1}{2}\sum_{i,j=1}^n p_{ij}m_{ij}\geq 0,\]
holds, for any $i,j=1,\ldots,n$, and for every $P\in\Skew_n(\mathbb{R})$ and $Q\in \mathcal{M}_n(\mathbb{R}_0^+)$ such that $\sum_{k=1}^n (p_{ik}-p_{jk})+q_{ij}-q_{ji}=0$.
\item The following two conditions fulfill
\begin{enumerate}
\item Condition \eqref{E-sticking-supercritical-explicit} holds true.
\item  We have that
\begin{equation}\label{E-sticking-critical-explicit}
\sum_{i=1}^m\sum_{j=m+1}^nm_{\sigma_i\sigma_j}\in [-nm(n-m),nm(n-m)],
\end{equation}
for every permutation $\sigma$ of $\{1,\ldots,n\}$ and any $1\leq m\leq n$.
\end{enumerate}
\end{enumerate}
\end{lemma}

\begin{proof}
The assertions $1$ and $2$ are apparently equivalent due to the definition of the involved linear operators. Also, both properties $2$ and $3$ are equivalent by virtue of an application of Lemma \ref{L-farkas-appendix} that is analogue to that in the proof of Lemma \ref{L-farkas-application-appendix}; hence, we skip the proof for simplicity. Thereby, we will only focus on the equivalence with the former assertion. First, let us assume that for some $Y\in\Skew_n([-1,1])$ the first assertion holds true, i.e.,
\[m_{ij}=\sum_{k=1}^n(y_{ik}-y_{jk}).\]
By Lemma \ref{L-farkas-application-appendix-2} we arrive at \eqref{E-sticking-supercritical-explicit}. Moreover,
\begin{align*}
\sum_{i=1}^m\sum_{j=m+1}^nm_{\sigma_i\sigma_j}&=\sum_{i=1}^m\sum_{j=m+1}^n\sum_{k=1}^n(y_{\sigma_i\sigma_k}-y_{\sigma_j\sigma_k})\\
&=(n-m)\sum_{i=1}^m\sum_{k=m+1}^ny_{\sigma_i\sigma_k}-m\sum_{j=m+1}^n\sum_{k=1}^my_{jk}\\
&=n\sum_{i=1}^m\sum_{k=m+1}^ny_{ik}.
\end{align*}
Since it is $n$ times the sum of $m(n-m)$ numbers in $[-1,1]$, then the condition \eqref{E-sticking-critical-explicit} is also satisfied. Conversely, let us assume that both \eqref{E-sticking-supercritical-explicit} and \eqref{E-sticking-critical-explicit} fulfill and take any $P\in\Skew_n(\mathbb{R})$ and $Q\in\mathcal{M}_n(\mathbb{R}_0^+)$ such that
\begin{equation}\label{E-32}
\sum_{k=1}^n (p_{ik}-p_{jk})+q_{ij}-q_{ji}=0,
\end{equation}
for any couple of indices $i,j=1,\ldots,n$. Without loss of generality we can assume that $q_{ii}=0$, for every $i=1,\ldots,n$. Also, let us define the coefficients $\lambda_i:=\sum_{k=1}^np_{ik}$ and consider a permutation $\sigma$ of $\{1,\ldots,n\}$ so that $\lambda_{\sigma_i}$ are ordered in a non-decreasing way, i.e.,
\begin{equation}\label{E-35}
\lambda_{\sigma_1}\leq \lambda_{\sigma_2}\leq \cdots\leq \lambda_{\sigma_n}.
\end{equation}
Let us split
\[I:=\sum_{i,j=1}^nq_{\sigma_i\sigma_j}+\frac{1}{2}\sum_{i,j=1}^n p_{\sigma_i\sigma_j}m_{\sigma_i\sigma_j}=:I_1+I_2.\]
Using \eqref{E-sticking-supercritical-explicit} in the second term we can write
\[I_2=\frac{1}{2}\sum_{i,j=1}^np_{\sigma_i\sigma_j}(m_{\sigma_i\sigma_k}-m_{\sigma_j\sigma_k}),\]
for any $k=1,\ldots,n$. Let us take the average with respect to $k$ in the above expression
\begin{align*}
I_2&=\frac{1}{2n}\sum_{i=1}^n\left(\sum_{k=1}^nm_{\sigma_i\sigma_k}\right)\lambda_{\sigma_i}+\frac{1}{2n}\sum_{j=1}^n\left(\sum_{k=1}^nm_{\sigma_j\sigma_k}\right)\lambda_{\sigma_j}\\
&=\frac{1}{n}\sum_{i=1}^n\left(\sum_{k=1}^nm_{\sigma_i\sigma_k}\right)\lambda_{\sigma_i}=\frac{1}{n}\sum_{i=1}^n\left(\sum_{k=1}^nm_{\sigma_i\sigma_k}\right)(\lambda_{\sigma_j}+q_{\sigma_j\sigma_i}-q_{\sigma_i\sigma_j}),
\end{align*}
for any $j=1\ldots,n$, where \eqref{E-32} has been used in the last step.
Taking the average with respect to $j$ we get to
\begin{align}
I_2&=\frac{1}{n^2}\sum_{i,j=1}^n\left(\sum_{k=1}^nm_{\sigma_i\sigma_k}\right)(q_{\sigma_j\sigma_i}-q_{\sigma_i\sigma_j})\nonumber\\
&=\frac{1}{2n^2}\sum_{i,j=1}^n\left(\sum_{k=1}^n(m_{\sigma_i\sigma_k}-m_{\sigma_j\sigma_k})\right)(q_{\sigma_j\sigma_i}-q_{\sigma_i\sigma_j})\nonumber\\
&=\frac{1}{n^2}\sum_{i<j}\left(\sum_{k=1}^n(m_{\sigma_i\sigma_k}-m_{\sigma_j\sigma_k})\right)(q_{\sigma_j\sigma_i}-q_{\sigma_i\sigma_j}).\label{E-33}
\end{align}
On the other hand
\begin{equation}\label{E-34}
I_2=\sum_{j>i}q_{\sigma_i\sigma_j}+\sum_{i<j}(q_{\sigma_i\sigma_j}-q_{\sigma_j\sigma_i}).
\end{equation}
Putting \eqref{E-33}-\eqref{E-34} together we obtain
\[I=2\sum_{j>i}q_{ij}+\sum_{i<j}\left(1-\frac{1}{n^2}\sum_{k=1}^n(m_{\sigma_i\sigma_k}-m_{\sigma_j\sigma_k})\right)(q_{\sigma_i\sigma_j}-q_{\sigma_j\sigma_i}).\]
Finally, notice that for every $i<j$, the condition \eqref{E-32} entails
\[q_{\sigma_i\sigma_j}-q_{\sigma_j\sigma_i}=\sum_{m=i}^{j-1}(q_{\sigma_m\sigma_{m+1}}-q_{\sigma_{m+1}\sigma_m}),\]
and, consequently
\[I=2\sum_{j>i}q_{ij}+\sum_{k=1}^na_m(q_{\sigma_m\sigma_{m+1}}-q_{\sigma_{m+1}\sigma_m}),\]
where the coefficients read
\begin{align*}
a_m&=\sum_{i=1}^m\sum_{j=m+1}^n\left(1-\frac{1}{n^2}\sum_{k=1}^n(m_{\sigma_i\sigma_k}-m_{\sigma_j\sigma_k})\right)\\
&=m(n-m)-\frac{1}{n}\sum_{i=1}^m\sum_{j=m+1}^nm_{\sigma_i\sigma_j}.
\end{align*}
Here, \eqref{E-sticking-supercritical-explicit} has been used again in the last identity. Since $a_m$ are all non-negative by \eqref{E-sticking-critical-explicit} and $\lambda_{\sigma_i}$ are ordered by \eqref{E-35}, we can conclude that $I\geq 0$ and this ends the proof.
\end{proof}

\end{document}